\documentclass[12pt]{article}
\usepackage{graphicx}
\usepackage{amsthm,amsmath,amsfonts,amssymb}
\usepackage{natbib} 
\usepackage{url} 
\usepackage{subcaption}
\usepackage{longtable}
\usepackage{threeparttable} 
\usepackage{multirow} 
\usepackage{booktabs}
\usepackage{setspace}
\usepackage{array}
\usepackage{float}
\usepackage{hyperref}
\usepackage{cleveref}
\hypersetup{colorlinks=true,linkcolor=blue,citecolor=blue}
\usepackage{xr}
\allowdisplaybreaks

\usepackage[]{lineno}
\newcommand*\patchAmsMathEnvironmentForLineno[1]{%
	\expandafter\let\csname old#1\expandafter\endcsname\csname 
	#1\endcsname
	\expandafter\let\csname oldend#1\expandafter\endcsname\csname 
	end#1\endcsname
	\renewenvironment{#1}%
	{\linenomath\csname old#1\endcsname}%
	{\csname oldend#1\endcsname\endlinenomath}}%
\newcommand*\patchBothAmsMathEnvironmentsForLineno[1]{%
	\patchAmsMathEnvironmentForLineno{#1}%
	\patchAmsMathEnvironmentForLineno{#1*}}%
\AtBeginDocument{%
	\patchBothAmsMathEnvironmentsForLineno{equation}%
	\patchBothAmsMathEnvironmentsForLineno{align}%
	\patchBothAmsMathEnvironmentsForLineno{flalign}%
	\patchBothAmsMathEnvironmentsForLineno{alignat}%
	\patchBothAmsMathEnvironmentsForLineno{gather}%
	\patchBothAmsMathEnvironmentsForLineno{multline}%
}

\newcommand{\blind}{0}
\def\mathbold{\boldsymbol}

\def\calA{{\cal A}}
\def\bv{\mathbold{v}}
\def\bA{\mathbold{A}}
\def\bX{\mathbold{X}}
\def\bx{\mathbold{x}}
\def\by{\mathbold{y}}
\def\R{{\mathbb{R}}}
\def\P{{\mathbb{P}}}
\def\E{{\mathbb{E}}}
\def\tby{{\widetilde{\by}}}
\def\bdelta{\mathbold{\delta}}
\def\tbbeta{\widetilde{\bbeta}}
\def\hbdelta{\widehat{\bdelta}{}}
\def\bSigma{\mathbold{\Sigma}}
\def\hbSigma{\widehat{\bSigma}{}}
\def\eps{\epsilon}
\def\bbeta{\mathbold{\beta}}
\def\hbbeta{\widehat{\bbeta}{}}

\def\argmin{\mathop{\rm arg\, min}}
\def\defas{\stackrel{\text{\tiny def}}{=}}
\def\bu{\mathbold{u}}
\def\hbu{\widehat{\bu}{}}
\def\bv{\mathbold{v}}
\def\hbv{\widehat{\bv}{}}
\newtheorem{theorem}{Theorem}[section]

\newtheorem{lemma}[theorem]{Lemma}
\newtheorem{assumption}{Assumption}

\newtheorem*{remark}{Remark}

\usepackage{xcolor}

\newcommand{\bruce}{}

\begin{document}

\def\spacingset#1{\renewcommand{\baselinestretch}%
{#1}\small\normalsize} \spacingset{1}


\title{Transfer Learning with Large-Scale Quantile Regression}

\if0\blind
{
  \author{Jun Jin$^1$, Jun Yan$^{1,2}$, Robert H. Aseltine$^2$, Kun Chen$^{1,2}$\thanks{Corresponding author; kun.chen@uconn.edu}\\    
$^1$\textit{Department of Statistics, University of Connecticut}\\
$^3$\textit{Center for Population Health, University of Connecticut Health Center}
}  
} \fi

\maketitle 

\begin{abstract}
  Quantile regression is increasingly encountered in modern big data
  applications due to its robustness and flexibility.
  We consider the scenario of learning the conditional quantiles of a specific
  target population when the available data may go beyond the target and be supplemented from other sources that possibly share similarities with the
  target. A crucial question is how to properly distinguish and utilize
  useful information from other sources to improve the quantile estimation and
  inference at the target. We develop transfer learning methods for
  high-dimensional
  quantile regression by detecting informative sources whose models
  are similar to the target and utilizing them to improve
  the target model.
  We show that under reasonable conditions, the detection of the
  informative sources based on sample splitting is consistent. Compared to the
  naive estimator with
  only the target data, the transfer learning estimator achieves a
  much lower error rate as a function of the sample sizes,
  the signal-to-noise ratios, and the similarity measures among the target and
  the
  source models. Extensive simulation studies demonstrate the superiority of our
  proposed approach. We apply our methods to tackle the problem of detecting
  hard-landing risk for flight safety and show the benefits and insights gained
  from transfer learning of three different types of airplanes: Boeing
  737, Airbus A320, and Airbus A380.

\bigskip
\noindent%
\end{abstract}

\doublespacing

\section{Introduction}
\label{sec:intro}

Quantile regression \citep{koenker1978regression,Koenker2001} is commonly encountered in
modern applications. It goes beyond mean regression and
allows the relationship between the outcome of interest and the features to
vary across the conditional quantiles of the outcome distribution. One
prominent example arises in analyzing and predicting hard landing for flights
with Quick Access Recorder (QAR) data \citep
{wang2014analysis,chen2021detection}, where the main objective is to study the
risk of high vertical acceleration at the touchdown point of a flight with
measurements from QAR during the flight. There are thousands of QAR
measurements, while the number of flights is relatively small for some airplane
models, such as the Airbus-380. Such high-dimensional but small-sample quantile
regression problems \citep{belloni2011, wang2012quantile} are prevalent in a
variety of fields. Some other examples include the modeling of the quantiles of
suicide risk \citep{rogers2018severity} to identify high-risk patients with
medical diagnoses and medication records from electronic health records
(EHR), and the use of quantile regression to reveal a comprehensive picture of
the varying effects of market factors on asset pricing \citep{maiti2021quantile}.
Applications with quantile regression are also seen in actuarial science \citep{baione2021application}, finance \citep{demir2022fintech}, environmental research \citep{reich2012spatiotemporal}, and medical studies \citep{wei2006quantile,huang2017quantile}. For an overview of the theory, methods and applications of quantile regression, we refer to \citet{Koenker2017}.

In many applications, the goal is to learn the conditional quantiles of an
outcome for a specific target population, while the available data may go well
beyond the target and may be from various other sources that bear certain
similarities with the target. For example, in the airplane hard landing problem,
an investigator may be only interested in Airbus-380, for which the target
population consists of only flights with Airbus-380; however, data from other
airplanes such as Airbus-320 or Boeing-737 may also be available. In healthcare,
a specific healthcare provider, e.g., a children's medical center, may only be
interested in building a risk prediction model of a disease/condition for its
own patient population, while patient data from other age ranges or other
hospitals may also be available \citep{XuChang2022}. In asset pricing, the main interest may be on a
new portfolio with limited historical data, while much more data may be
available on other similar portfolios or market indices like S\&P 500 index \citep{gu2020empirical}.

In all the above problems, a crucial question is how to properly transfer the
information from the sources to the target to best perform the estimation and
inference of the conditional quantiles for the target. Apparently, neither the
approach of only using the information pertaining to the target or the approach
of naively combining all the available information to fit a one-size-fits-all
model would work well. While the former fails to seek any potential gain by
completely ignoring the additional information from the sources, the latter may
result in ``negative transfer'' \citep{rosenstein2005transfer} by ignoring the
potential heterogeneity among the sources and the target. Such heterogeneity
may arise from various aspects and ultimately makes the transfer learning
problem challenging. First, even if we assume the sources and the target share
the same set of features, these features may follow different distributions
across the sources and the target. Second, the true model of the target may not
be the same as that of each of the sources; rather, it is more realistic to
assume that these models can be different and only a subset of source models
may be similar enough to the target model to permit ``positive transfer''. Last
but not the least, the error distributions of different models could be
different, not only in their magnitudes (variances) and types (Gaussian vs.
non-Gaussian), but also in their quantile levels. To our knowledge, existing
transfer learning frameworks mainly deal with the first and the second types of
heterogeneity. \cite{bastani2021predicting} designed a two-step procedure
(proxy step and joint step) of transfer learning for high-dimensional linear
regression with one target and one source. \citet{li2020transfer} proposed a
procedure to realize transfer learning of high-dimensional linear models
between one target and several sources, which first trains a fusion model with
all data and then performs debiasing with the target data. \citet
{tian2022transfer} made an extension along the same line for the generalized
linear models. A concurrent study from \citet
{huang2022transfer} explored the direct extension of the
two-step framework proposed in \citet{li2020transfer} for quantile regression
models. See \citet{pan2009survey}, \citet{weiss2016survey}, and \citet
{cheng2020hybrid} for some comprehensive surveys on transfer learning in
machine learning, statistics, and engineering.

We propose a novel transfer learning procedure for high-dimensional quantile
regression to properly utilize information from external data sources while
allowing for all three types of heterogeneity in feature distributions, model
coefficients, and error distributions. Our approach is built upon the two-step
framework in \citet{li2020transfer}, and a transformation step is added to
handle the heterogeneity of the error types and quantile levels, which is
specific to quantile regression problems \citep{chen2020distributed}. Compared
with
\citet{li2020transfer} and \citet{tian2022transfer}, our proposed method relaxes
 the condition on error distribution and enables transfer learning for quantile
 regression with a non-differentiable loss. Comparing with \citet
 {huang2022transfer}, our proposed method executes sparse quantile regression problems distributively on the different sources rather than pooling them and utilizes surrogate responses to avoid the quantile loss, making our method more suitable for large-scale real-world applications.
We develop an informative-source detection procedure based on sample splitting
and show that under reasonable conditions, the procedure is consistent.
Compared to the naive model fitted with the target data alone, the proposed
transfer-learning model enjoys a sharper convergence rate as a function of
sample sizes, signal-to-noise ratio, and similarities among the target and the
source models.
Extensive simulation studies verify the superiority of our approach. 
We apply the
proposed approach to tackle the problem of detecting hard-landing risk for
flight safety, and show that with the help of Airbus-320 and Boeing-737
datasets, the transfer learning model for Airbus-380 gives more accurate
estimation 
and the resulting feature contributions are visualized and highly interpretable.  

The rest of the paper is organized as follows. In \Cref{sec:model}, we present
the transfer learning setup with high-dimensional quantile regression. We
propose the estimation procedures via transfer learning in \Cref{sec:est-method}
and provide their theoretical guarantees in \Cref{sec:theory}. 
Numerical experiments and the application on airplane hard-landing are provided
in \Cref{sec:simulation} and \Cref{sec:real-data}, respectively.
\Cref{sec:discussion} gives some concluding remarks and future research
directions. All the technical details are provided in the Appendix.

\section{Transfer Learning Setup for Quantile Regression}
\label{sec:model}

Before carrying out the model setup, we introduce some notations. For two constants $c_1$ and $c_2$, we denote $c_1 \vee c_2 = \max(c_1,c_2)$ and $c_1 \wedge c_2 = \min(c_1,c_2)$. 
For a vector $\bv = {( {{v_1}, \ldots ,{v_n}} )^\top}$, we denote $\| \bv \|_0 = \sum\nolimits_{i = 1}^n I(v_i\neq 0)$ where $I(\cdot)$ is the indicator function as the $\ell_0$ norm, $\| \bv \|_1 = \sum\nolimits_{i = 1}^n | {{v_i}} |$ as the $\ell_1$ norm,  $\| \bv \| = \sqrt {\sum\nolimits_{i = 1}^n {v_i^2} } $ as the $\ell_2$ norm, and $\| \bv \|_\infty = {\max _i}| {{v_i}} |$ as the infinity norm of $\bv$. 
For a matrix $\bA \in {R^{n \times n}}$, we denote its infinity norm ${\| \bA \|_\infty } = {\max _i}\sum\nolimits_{j = 1}^n {| {{A_{ij}}}|} $. For two sequences $\{ {{a_n}} \}_{n = 1}^\infty$ and $\{ {{b_n}} \}_{n = 1}^\infty$, we denote ${a_n} \asymp {b_n}$ if and only if ${a_n} = O( {{b_n}} )$ and ${b_n} = O( {{a_n}} )$. For two non-zero real sequences $\{ {{c_n}} \}_{n = 1}^\infty$ and $\{ {{d_n}} \}_{n = 1}^\infty$, ${c_n} \gg {d_n}$ stands for $| {{c_n}/{d_n}} | \to 0$. We use $\lesssim$ and $\gtrsim$ to indicate that the inequality holds up to some multiplicative numerical constant. 


Our main interest is to learn the conditional quantiles of an outcome variable
of interest, $y^{(0)} \in \mathbb{R}$, given a set of feature variables,
$\bx^{(0)} \in \mathbb{R}^{p}$, at a specified target population. 
We assume that the following quantile regression model holds at the target, 
\begin{equation}\label{eq:target-model}
\left\{ \begin{array}{l}
y^{(0)} = \bx^{(0)^\top}{\bbeta _\tau ^{\left( 0 \right)}} + \eps^{(0)};\\
\P\left( {\eps^{(0)} \le 0} \right) = {\tau},
\end{array} \right.
\end{equation}
where 
${\bbeta _\tau ^{( 0 )}}$ is the true regression coefficient for the
target population under the quantile level $\tau$, and $\eps^{(0)}\in\R$ is
the random error term with a density function ${f^{( 0 )}}(  \cdot  )$. When $\tau$ is specified and no confusion arises, we may
abbreviate ${\bbeta _\tau ^{( 0 )}}$ as $\bbeta^{(0)}$. To deal with
high dimensionality, we assume that the true coefficient vector $\bbeta^{(0)}$
is sparse with $s_0$ nonzero elements and ${\| {{\bbeta ^{( 0 )}}} \|_\infty }$ is bounded. Here $\bx^{(0)}$ is independent of the
noise $\eps^{(0)}$, and that the density of $\eps^{(0)}$ exists with $\P(
{\eps^{(0)} \le 0} ) = \tau$. This model thus allows the error term to be
heavy-tailed, such as Cauchy distribution or $t$ distribution. It is worth
noting that the true coefficient vector $\bbeta^{(0)}$ is the solution of the
following population-level optimization problem
\begin{equation}\label{eq:population-est}
  \bbeta^{(0)} = \argmin_{\bbeta  \in {\R^p}} \E \rho_{\tau} \left(y^{(0)} - \bx^{(0)^\top}\bbeta\right), 
\end{equation}
where $\rho_{\tau}(u) = u(\tau - I(u\le 0))$ is the quantile loss
function \citep{koenker2005quantile}.

In addition to the target population, suppose we also have access to $K$ source
populations with the same set of features and the same response. For each $k
=1,\ldots, K$, we assume that the following source-specific quantile regression
model holds, 
\begin{equation}\label{eq:source-model}
\left\{ \begin{array}{l}
y^{(k)} = {\bx^{(k)^\top}}{\bbeta _\tau ^{\left( k \right)}} + \eps^{(k)};\\
\P\left( {\eps^{(k)}\le 0} \right) = {\tau}.
\end{array}\right. 
\end{equation}
Here the regression coefficient vector ${\bbeta _\tau ^{( k )}}$ or
simply ${{\bbeta ^{( k )}}}$, can be different from that of the
target $\bbeta^{(0)}$. Similarly, we assume $\bbeta^{(k)}$ has bounded infinity norm, sparse with $s_k$
nonzero elements, and the error term $\eps^{(k)}$ is independent with
${{\bx^{(k)}}}$ and satisfies the same density quantile condition. ${f^{( k )}}( \cdot )$ is the density function of the random error $\eps^{(k)}$.

Intuitively, the rationale of transfer learning is that, if some of the source
models are similar ``enough'' to the target model, we may expect these
sources to provide additional valuable information about the target model. It is
then pivotal to quantify the similarity between the target model
in~\eqref{eq:target-model} and the source models in~\eqref{eq:source-model}. A
natural idea is to use the $\ell_1$ distance between the coefficient vector of
the target model and that of each source model, i.e., $\|{\bdelta ^{( k )}}\|_1 = \|\bbeta^{(0)} -
\bbeta^{(k)}\|_1$. When this distance is smaller than certain threshold $d>0$, we
say that the $k$th source is $d$-informative. 
Consequently, we define
$$
\mathcal{A} = \left\{ {1 \le
    k \le K: {\|\bdelta ^{\left( k \right)}}\|_1 \le d}\right\}
$$
to be the $d$-informative set.


Now suppose we observe $n_0$ independently and identically distributed (i.i.d.) samples from the target population under the model in \eqref{eq:target-model}, consisting of a feature matrix $\bX^{(0)} = (\bx_1^{(0)}, \ldots,
\bx_{n_0}^{(0)})^\top\in\R^{n_0\times p}$ \bruce{whose rows are i.i.d. from Gaussian distribution with mean zero and covariance matrix ${\bSigma ^{( 0 )}}$}, and a response/outcome vector $\by^{(0)} = (y_1^{(0)}, \ldots,
y_{n_0}^{(0)})^\top\in\R^{n_0}$. 
In addition to the target data, suppose we also have i.i.d. samples from each of the $K$ source populations under the models in \eqref{eq:source-model}. For $k = 1,\ldots, K$, the data from source $k$ are denoted as $\{(\bX^{(k)}, \by^{(k)})\}_{k=1}^K$ with $\bX^{(k)}\in \R^{n_k\times p}$ \bruce{whose rows are i.i.d. from Gaussian distribution with mean zero and covariance matrix ${\bSigma ^{( k )}}$}, and $\by^{(k)}\in \R^{n_k}$ where $n_k$ is the sample size at source $k$. In the sequel we denote ${\underline{n}} = \min_{k = 1}^{K} {n_k}$, ${\underline{n}_{\calA}}  = {\min _{k \in {\calA \cup \{ 0 \}}}}{n_k}$, and ${\overline{n}_{\calA}}  = {\max _{k \in {\calA \cup \{ 0 \}}}}{n_k}$.

Our main task is then to develop a transfer learning procedure to estimate
$\bbeta^{(0)}$ in the target model~\eqref{eq:target-model} with both the target
data $(\bX^{(0)}, \by^{(0)})$ and the additional source data $\{(\bX^{(k)}, \by^{(k)})\}_{k=1}^K$. In particular, we are interested in clarifying how the potential benefit of transfer learning is related to the specifications of the error distributions, the distributional-shift of the features, the dimensions and sparsity levels of the models, and the similarity levels of the sources models to the target.

Before we dive into the transfer learning approach, 
we first present a model that one would use if only provided the target data. Specifically, we consider the following $\ell_1$-penalized quantile regression estimator 
\begin{equation}
  \label{eq:classical}
  \hbbeta^{(0)} = \argmin_{\bbeta  \in {\R^p}} \frac{1}{n_0} \sum_{i=1}^{n_0} \rho_{\tau} (y_i^{(0)} - {\bx_i^{(0)^\top}}\bbeta) + \lambda_0 \left\|\bbeta\right\|_1,
\end{equation}
where $\lambda_0$ is the regularization parameter. This estimator enjoys consistency and support recovery properties under some regular conditions on the target data \citep{fan2014adaptive}. We refer to this method as the ``target-only method'' which serves as the baseline to be compared with transfer-learning approach. 

\section{Estimation Procedures via Transfer Learning}
\label{sec:est-method}

We develop procedures for estimating the quantile regression model at the target via transfer learning. We start with the oracle case that the informative set $\mathcal{A}$ is given and then consider the estimation of the informative set $\mathcal{A}$ to complete the puzzle.

\subsection{Oracle Procedure with Known Informative Set}\label{sec:oracle-estimation}

Suppose the informative set $\calA$ is given, that is, we've known a priori that it is the most beneficial to consider information transfer from the sources belonging to $\calA$. In this case, we realize transfer learning by a four-step procedure, 
consisting of (1) single-source modeling, (2) response surrogation, (3) fusion learning, and (4) debiasing.

\begin{figure}[tbp]
\centering
\includegraphics[width=0.45\textwidth]{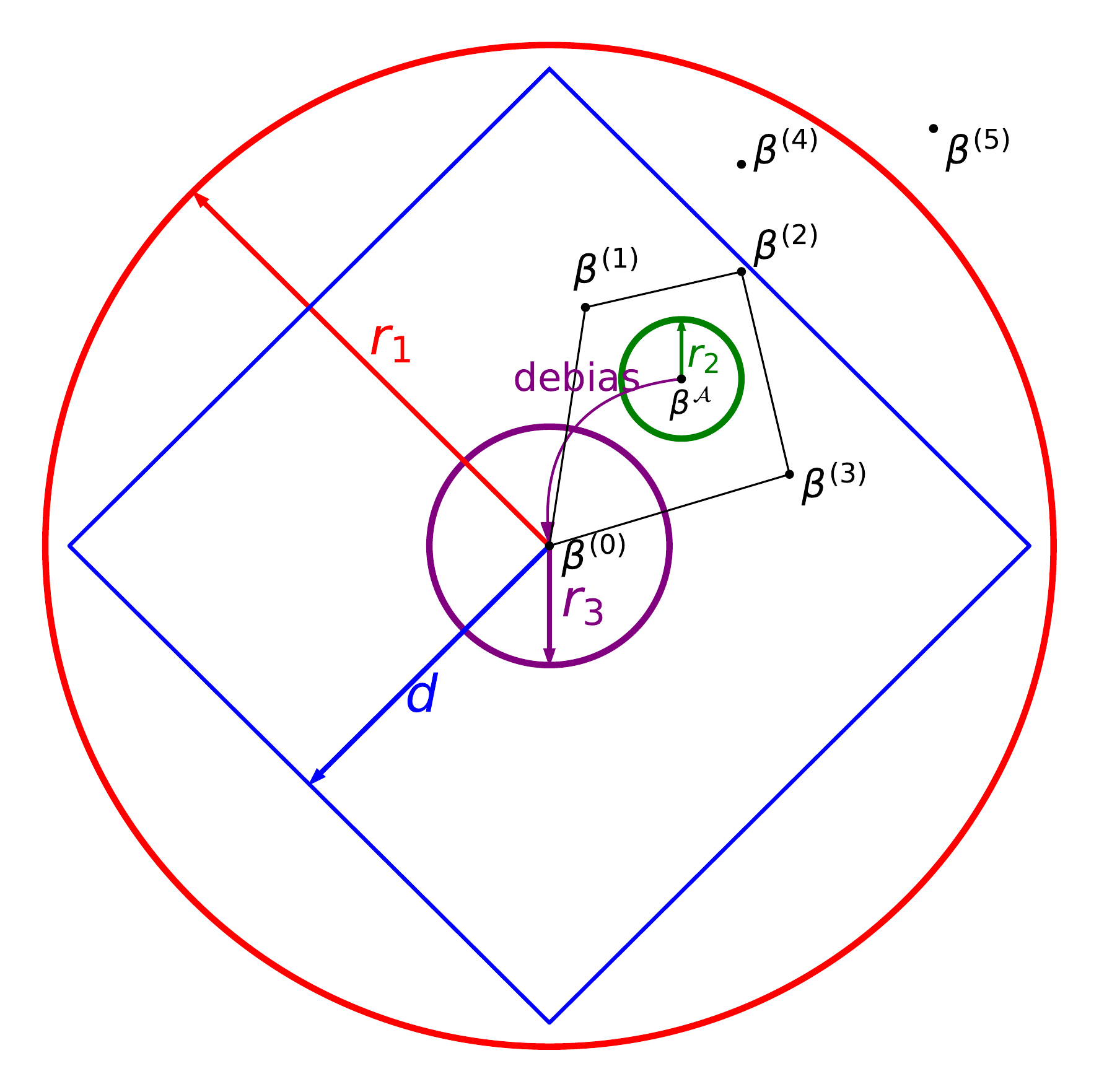}
\caption{Schematic of the procedures in transfer learning quantile regression.}
\label{fig:schematic}
\end{figure}

Figure~\ref{fig:schematic}, inspired by Figure 5 in \citet{tian2022transfer}, illustrates how the proposed procedure can possibly lead to an improved estimator of $\bbeta^{(0)}$ with information from the informative sources. With the limited sample size $n_0$ at the target, the convergence rate of the target-only estimator under $\ell_2$-norm is ${r_1} = \sqrt {{s_0}{\log ( {p \vee {n_0}} )}/{n_0}} $, which is the radius of the red circle centered at the true coefficient vector $\bbeta^{(0)}$. Suppose sources 1, 2, and 3 are $d$-informative among the five sources, their corresponding true coefficient vectors  ${\bbeta ^{( 1 )}}$, ${\bbeta ^{( 2 )}}$, and ${\bbeta ^{( 3 )}}$ are then within the blue square. To start the algorithm, firstly, for each $k\in \calA \cup \{0\}$, we estimate $\bbeta^{(k)}$ using $\ell_1$-penalized quantile regression as in \eqref{eq:classical} (Step 1), and then adopt the method in \citet{chen2020distributed} to construct surrogate responses in order to simplify the quantile regression problem to linear regression (Step 2). Steps 1 and 2 are designed to locate ${\bbeta ^{( 1 )}}$, ${\bbeta ^{( 2 )}}$ and ${\bbeta ^{( 3 )}}$ as the best linear regression coefficients of the surrogate responses on their covariates. The third and the fourth steps then conduct fusion learning and debiasing, respectively \citep{li2020transfer}. To be more specific, Step 3 locates $\hbbeta^{\cal A}$ through combining information from the target and the informative sources; this is with a convergence rate $r_2$ (the radius of the green circle) and the estimator is inside the convex hull established by ${\bbeta ^{( k )}}$, $k \in \calA \cup \{0\}$. The exact location of ${\bbeta ^\mathcal{A}}$, which is the target or the probabilistic limit for $\hbbeta^{\cal A}$, is determined by a weighted linear transformation of the $\bbeta^{(k)}$s from the informative sources, where the weights are associated with the sample size and the feature covariance matrix of each source. More details can be found in Section~\ref{sec:proofs} in the Appendix.
Finally, Step 4 performs debiasing only using the target information, by which the final estimator is within the circle around $\bbeta^{(0)}$ with a radius $r_3$. Typically, $r_3$ is expected to be larger than $r_2$ due to a bias-variance trade-off, but it could be much smaller than $r_1$, the convergence rate of the target-only model.

The details of our proposed procedure are presented below.\\
\noindent {\bf Step 1 (Single-source modeling):} For each $k\in \calA \cup \{0\}$, we estimate $\hbbeta^{(k)}$ by
  an $\ell_1$-penalized quantile regression,
  \begin{equation}\label{eq:initial-est}
  \hbbeta^{(k)} = \argmin_{\bbeta  \in {\R^p}}  {\frac{1}{{{n_k}}}\sum\limits_{i = 1}^{{n_k}} \{\tau - I( {y_i^{\left( k \right)} - \bx_i^{(k)^\top}\bbeta  \le 0} )\} \left( {y_i^{\left( k \right)} - \bx_i^{(k)^\top}\bbeta } \right)}  + {\lambda _0}\left\|\bbeta\right\|_1.
  \end{equation}
  
\noindent {\bf Step 2 (Response surrogation):} For each $k \in \calA \cup \{0\}$, the kernel
  density estimator for the density of the error evaluated at the point of zero is given by
  \begin{equation}
  \label{eq:est_density}
     { {\widehat f}^{\left( k \right)}}( 0 ) = \frac{1}{{n_k}h^{\left( k \right)}}\sum\limits_{i = 1}^{{n_k}} G\left(\frac{ y_i^{\left( k \right)} - \bx_i^{(k)^\top}{\hbbeta ^{\left( k \right)}}}{h^{\left( k \right)}}\right),
  \end{equation}
  where $G$ is a kernel function satisfying Assumption~\ref{assum:A6} (to be shown in Section \ref{sec:theory}) with $h^{( k )}$ being its bandwidth. A surrogate response is then constructed as
  \begin{equation}
  \widetilde{y}_i^{\left( k \right)} = \bx_i^{{{\left( k \right)}^\top}}{\hbbeta^{\left( k \right)}} - \{ {\widehat f}^{\left( k \right)}( 0 )\}^{-1}\left\{ {I\left( {y_i^{\left( k \right)} - \bx_i^{(k)^\top}{\hbbeta^{\left(k \right)}} \le 0} \right) - \tau} \right\}, \qquad i = 1,\ldots, n_k. \label{eq:surrogate} 
  \end{equation}
Denote 
${\tby^{( k )}} = (\widetilde{y}_1^{( k )}, \ldots, \widetilde {y}_{n_k}^{( k )})^\top$ as the surrogate response vector. 
  
\noindent {\bf Step 3 (Fusion learning):} We fit an $\ell_1$-penalized regression with the surrogate responses
  $\tby^{(k)}$ with all data
  \begin{equation}
  \label{eq:fusion-lasso}
  \hbbeta^{\cal A}= \argmin_{\bbeta  \in {\R^p}} \frac{1}{2(n_{\cal A}+ {n_0})}
  \sum_{k \in \calA \cup \{0\}} \left\| \tby^{(k)} - \bX^{(k)}\bbeta\right\|^2 + {\lambda _1} \left\|\bbeta\right\|_1,
  \end{equation}
  where $n_{\cal A} = \sum_{k \in \cal A} n_k$ is the total sample size of $\cal A$.
  
\noindent {\bf Step 4 (Debiasing):} We fit an $\ell_1$-penalized regression with the target data 
\begin{equation}\label{eq:debias}
{\hbdelta ^{\cal A}} = \mathop {\arg \min }\limits_{\bdelta  \in {\R^p}} \frac{1}{{2{n_0}}}{\left\| {{\tby^{\left( 0 \right)}} - {\bX^{\left( 0 \right)}}\left( {{\hbbeta ^{\cal A}} - \bdelta } \right)} \right\|^2} + {\lambda _2}{\left\| \bdelta  \right\|_1}
\end{equation}
Finally, the oracle quantile regression estimator is given as 
\begin{equation}\label{eq:oracle-est}
  {\widehat \bbeta } = {\hbbeta^\mathcal{A}} - {\hbdelta^\mathcal{A}}.
\end{equation}

In comparison to the concurrent study by \citet{huang2022transfer}, our proposed
estimation method could lead to superior computational efficiency in large-scale
settings. The fusion learning step in the concurrent study solves a pooled
$\ell_1$-penalized quantile regression with sample size $n_{\cal A}$. 
In contrast, our method trains a $\ell_1$-penalized quantile regression for the target and each source individually, which can be implemented in parallel with computation time now largely driven by the maximum local sample size ${{\max }_{k \in {\cal A} \cup \{ 0 \}}}n_k$. 
Moreover, our subsequent fusion step with the pooled samples becomes a Lasso problem through the adoption of the surrogate response, which is much easier to solve than its counterpart of $\ell_1$-penalized quantile regression.

\subsection{Identification of Informative Set}

In the previous section, the informative set $\calA$ is assumed to be known as a
priori, which is not practical in many real applications. Motivated by \citet{tian2022transfer}, we develop an effective method to detect the informative set $\calA$. 

The proposed detection method has three steps. Firstly, we randomly
split the target data into two sub-samples of equal size, one for training (${\cal I}$) and another for testing (${\cal I}^c$). Second, we run the transfer learning procedure with each source data and the target training data, to produce a set of single-source transfer learning estimators. These transfer learning estimators are evaluated on the target testing data. Finally, we determine the informative sources by comparing the losses incurred by the transfer learning estimators to that of the target-only estimator. Figure~\ref{fig:schematic-detect} illustrates the procedure of informative sources detection and shows its similarity to the idea of cross validation. The blue lines indicate the operations related to the half of the target data for training (blue half-circle on the right), and the red lines indicate the operations related to the other half of the target data for testing (red half-circle on the left), while the black lines indicate the comparisons between the target-only estimator and the single-source transfer learning estimators for the detection of informative sources.

\begin{figure}[tbp]
\centering
\includegraphics[width=0.5\textwidth]{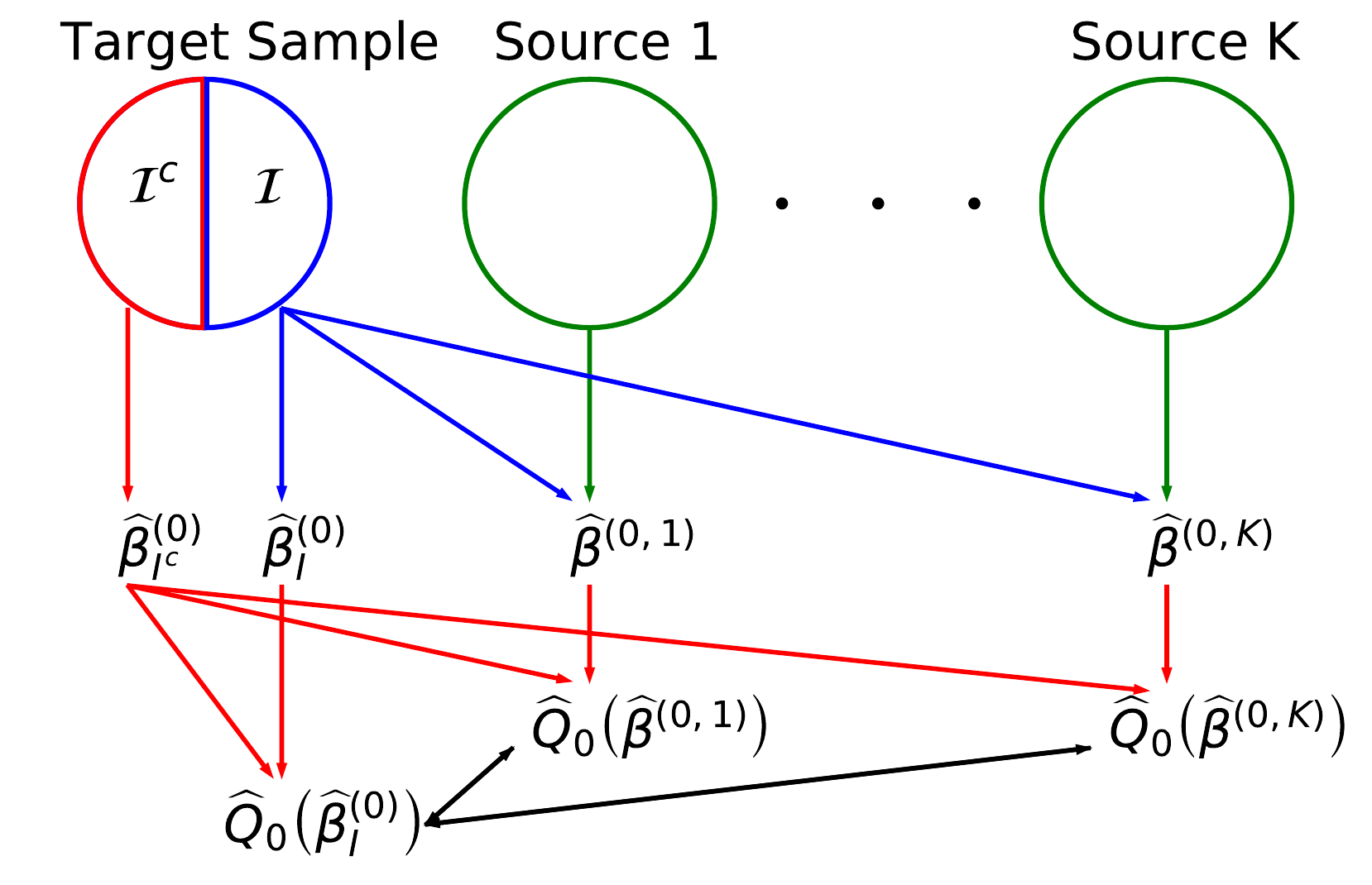}
\caption{Schematic of the procedures in informative set detection.}
\label{fig:schematic-detect}
\end{figure}

To present the details, we first introduce the loss function that is used to evaluate the estimators. The population-level loss on the target model is defined as 
\begin{equation}
\label{eq:Q0}
{Q_0}\left( \bbeta \right) = \E\left\{ {{{\left( {\widetilde y^{(0)} - {\bx^{(0)^\top}}\bbeta} \right)}^2}} \right\},
\end{equation}
where $\widetilde y^{(0)} = {\bx^{(0)^\top}}{\bbeta ^{( 0 )}} - \{ {f^{( 0 )}}( 0 ) \}^{-1}\{ {I( {y^{(0)} - {\bx^{(0)^\top}}{\bbeta
      ^{( 0 )}} \le 0} ) - \tau} \}$ and the expectation is taken with respect to the target distribution. This loss function was adopted by \citet{chen2020distributed} in a distributed quantile regression setting.  The corresponding finite-sample loss evaluated on the target testing samples, with indices $i \in {\cal I}^c$, is calculated as:
\begin{equation}\label{eq:emprical-loss}
{\widehat Q_0}\left( \bbeta \right) = \frac{1}{{{n_0/2}}}\sum\limits_{i \in {{\cal I}^c}} {{{\left( {\widetilde y_i^{(0)} - \bx_i^{(0)^\top}\bbeta} \right)}^2}} ,
\end{equation}
where $\widetilde y_i^{(0)} = \bx_i^{(0)^\top} \hbbeta _{{{\cal I}^c}}^{( 0 )} -
\{{{\widehat f}^{( 0 )}}( 0 )\}^{-1}\{ {I( {y_i^{(0)} - \bx_i^{(0)^\top}\hbbeta
_{{{\cal I}^c}}^{( 0 )} \le 0} ) - \tau}\}$ and 
\begin{equation*}
{{\widehat f}^{\left( 0 \right)}}( 0 )
= \frac{2}{{n_0}h_{{\cal I}^c}^{\left( 0 \right)}}\sum\limits_{i \in {{\cal I}^c}} {G\left( {\frac{{y_i^{(0)} - \bx_i^{(0)^\top}\hbbeta _{{{\cal I}^c}}^{\left( 0 \right)}}}{h_{{\cal I}^c}^{\left( 0 \right)}}} \right)} ,
\end{equation*}
with $\hbbeta _{{{\cal I}^c}}^{( 0 )}$ denoting the $\ell_1$-penalized quantile regression estimator with the target testing data.

With the above definitions, the detailed procedure for detecting the
informative set $\calA$ is as follows.

\noindent {\bf Step 1}: We randomly divide the target sample into two 
  subsamples $S_1 = (\bX^{(0)}_{\cal I}, \by^{(0)}_{\cal I})$, $S_2 = (\bX^{(0)}_{{\cal I}^c}, 
  \by^{(0)}_{{\cal I}^c})$, where ${\cal I}$ is a subset of $\{1, 2, \ldots, n_0\}$ with 
  cardinality $| {\cal I} | \approx n_0/2$.  
  Then we calculate $\hbbeta _{\cal I}^{( 0 )}$ by performing $\ell_1$-penalized quantile regression with $S_1$. For each $1\le k\le K$, we obtain $\hbbeta^{(0,k)}$ by running Steps 1--3 in \Cref{sec:oracle-estimation} with the sample $({\bX_{\cal I}^{( 0 )},\by_{\cal I}^{( 0 )}} ) \cup ( {{\bX^{( k )}},{\by^{( k )}}} )$. 

\noindent {\bf Step
  2}: On $S_2 = (\bX^{(0)}_{{\cal I}^c}, 
  \by^{(0)}_{{\cal I}^c})$, we run $\ell_1$-penalized quantile regression to obtain $\hbbeta
  _{{{\cal I}^c}}^{( 0 )}$, then calculate the loss $\widehat Q_0(\hbbeta _{\cal I}^{( 0 )})$ and $\widehat Q_0(\hbbeta^{(0,k)})$ defined in~\eqref{eq:emprical-loss} for
  each $k\in \{1, 2, \ldots, K\}$.

\noindent {\bf Step 3}: The estimated informative set is
    \begin{equation}\label{eq:est-of-A}
  \widehat {\cal A} = \left\{ {1\leq k \leq K:{{\widehat Q}_0}\left( {{\hbbeta^{\left( 0,k \right)}}} \right) \le \left( {1 + {\varepsilon _0}} \right){{\widehat Q}_0}\left(\hbbeta _{\cal I}^{\left( 0 \right)} \right)} \right\},
  \end{equation}
  with a given small value $\varepsilon_0 \in ( {0, c_\varepsilon} ]$ that controls the strictness of selection. Here $c_\varepsilon$ is a constant that satisfies Assumption~\ref{assum:A8} below. In practice, we can estimate ${c_\varepsilon }$ as,
\begin{equation}
\label{eq:E0suggest}
{\widehat c_\varepsilon } = \mathop {\inf }\limits_{k \in \left\{ {1, \ldots ,K} \right\}} \frac{{{{\left( {{{\hbbeta }^{\left( k \right)}} - {{\hbbeta }^{\left( 0 \right)}}} \right)}^\top}{{\hbSigma }^{\left( 0 \right)}}\left( {{{\hbbeta }^{\left( k \right)}} - {{\hbbeta }^{\left( 0 \right)}}} \right)}}{{{{\widehat Q}_0}\left( {{{\hbbeta }^{\left( 0 \right)}}} \right)}}
\end{equation}
where ${\hbbeta ^{( k )}}$, $k=0,\ldots, K$ are the initial single-source estimators from \eqref{eq:initial-est}, and $\widehat{\bSigma}{} ^{( 0 )}$ is an empirical estimator of ${\bSigma ^{( 0 )}}$ with the target data, e.g., the sample covariance matrix of $\bX^{( 0 )}$.

Finally, we run the proposed transfer learning procedure in \Cref{sec:oracle-estimation} with $\widehat\calA$ to obtain our Trans-Lasso
quantile regression estimator, denoted as $\tbbeta$, to be distinguished from the oracle estimator $\hbbeta$ defined in~\eqref{eq:oracle-est}.

\section{Theoretical Analysis}\label{sec:theory}

We impose the following assumptions.

\begin{assumption}[Sparsity]\label{assum:A1} 
  For each $k = 0,1,\ldots,K$, let ${S_k} = \{ {0 \le i \le p:\beta _i^{( k )} \ne 0}
  \}$ and ${s_k} = | S_k |$, and assume ${s_k} = O( {n_k^r} )$ for some $0<r<1/3$. 
\end{assumption}

\begin{assumption}[Feature distributions]\label{assum:A2} 
  For each $k = 0,1,\ldots,K$, the rows of $\bX^{( k )}$
  are i.i.d. Gaussian distributed with mean zero and covariance matrix ${{\bSigma}
    ^{( k )}}$. The eigenvalues of each ${{\bSigma} ^{( k )}}$ are bounded away from 0 and $\infty$.
\end{assumption}

  With the potential distributional-shift in features, i.e., ${{\bSigma}^{( k )}}$s are allowed to be different, the following measurments are defined to quantify the maximum difference between ${{\bSigma} ^{( 0 )}}$ and ${{\bSigma} ^{( k )}}, k \in {\cal A}$:
  \begin{equation*}
  C_{\bSigma}^{\cal A}= 1 + \max_{1\le j\le p}\max_{k \in {\cal A}}\left\|\boldsymbol{e}_j^\top (\bSigma^{(k)}-\bSigma^{(0)})\left(\sum_{k \in {\cal A} \cup \left\{ 0 \right\}} \alpha _k\bSigma^{(k)}\right)^{-1}\right\|_1,
  \end{equation*}
  where $\boldsymbol{e}_j$ is the a unit vector with 1 on the $j$th position and $\alpha_k = n_k/(n_{\calA} + n_0)$. 

  \begin{assumption}[Sample size and finite sources]\label{assum:A3}
    We assume
    $$\max \left\{ {{s_0}\log p/\left(\underline{n} + n_0\right),{C_{\bSigma}^{\cal A}}d{{\left( {\log p/{n_0}} \right)}^{1/2}}} \right\} = o\left( 1 \right).
    $$
    Besides, we assume $K = O( {{n_0}} )$ and $| {\cal A}|$ is bounded by a constant.
\end{assumption}

\begin{assumption}[Irrepresentable condition]\label{assum:A4} 
  For each $k = 0,\ldots, K$, $${\left\| {{\bSigma}^{(k)}_{{S_k^c}
  \times S_k}} ({\bSigma}^{(k)}_{S_k \times S_k})^{-1}\right\|_\infty } \le 1 -
  a_k$$ for some $0 < a_k  < 1$.  
\end{assumption}

\begin{assumption}[Error distributions]\label{assum:A5} For each $k = 0, \ldots,
 K$, the density function of the noise, ${{f^{( k )}}( \cdot )}$, is bounded
 and Lipschitz continuous in a neighborhood around 0, i.e., $\mathcal{B}(0,l) = \{x \in \mathbb{R}; |x| < l\}$ with the radius $l$ being a positive constant. Moreover, we assume ${{f^{( k )}}( 0 )}$ is
 bounded away from $0$ for every $k$. 
\end{assumption}


\begin{assumption}[Smoothing kernel]\label{assum:A6} 
  The kernel function $G(\cdot)$ is set to be integrable with $\int_{ - \infty
  }^{ + \infty } {G( u )du = 1} $. Meanwhile, $G( u ) = 0$
  if $| u | \ge 1$. Also, $G(\cdot)$ is differentiable and $G'(\cdot)$ is bounded. To determine the bandwidth, we take $h^{( k )} \asymp \sqrt {{s_k}\log {n_k}/{n_k}} $, $h_{\cal I}^{( 0 )} = h_{{{\cal I}^c}}^{( 0 )} \asymp \sqrt {{s_0}\log ( {{n_0}} )/( {n_0} )} $.
\end{assumption}

\Cref{assum:A1} allows the sparsity level $s_k$ to grow at the rate  $n_k^
 {1/3}$ for each model, where $1/3$ is determined to ensure that the lasso estimator based on the surrogate response
 from each source is at least as good as the
 initial estimator from the single-source modeling in terms of their convergence rates. 
\Cref{assum:A2} is on the distributions of the features. The boundedness of
 eigenvalues of $\bSigma^{(k)}$ is commonly adopted in high-dimensional
 problems; examples can be found in, e.g., \citet{cai2017} and \citet
 {chen2020distributed}. Accordingly, we impose measures $C_{\bSigma}^{\cal A}$
 to quantify the potential distributional shift between the sources and the
 target. Under \Cref{assum:A2}, $C_{\bSigma}^
 {\cal A}$ can be further bounded, as shown in \citet{li2020transfer} and \citet
 {tian2022transfer}. For example, for each target and source, the design matrix
 $\bX^{( k )}$ can be generated from a zero-mean Gaussian distribution with
 (i) a block diagonal covariance matrix ${{{\bSigma} ^{( k )}}}$ with constant
 block sizes, (ii) a Toeplitz covariance matrix as in our simulation study for
 the heterogeneous design, and (iii) a covariance matrix with autoregressive
 structure ${\bSigma} _{i,j}^{( k )} = {a^{| {i - j} |}}$ with $a < 1$. Under
 these models, $C_{\bSigma}^\mathcal{A}$ can be bounded by a constant, and
 consequently, Assumption 3 amounts to require (a) the true model is
 sufficiently sparse, i.e., $s_0\log{p}/(\underline{n}+n_0)=o
 (1)$, and (b) the discrepancy in regression coefficients between the target
 and the sources is not too large, i.e., $d(\log p/n_0)^{1/2}=o(1)$.
\Cref{assum:A4} is the irrepresentable condition commonly assumed for achieving sparsity recovery with $\ell_1$ regularization in high-dimensional statistics; see, e.g., \citet{zhao2006model}, \citet{wainwright2009}, and \citet{Buhlmann2011}. In our problem, it is required to control the error of estimating the error density functions by giving sparsity recovery property to the initial estimators in Step 1~\citep{chen2020distributed}, which are used in the construction of the surrogate responses; see~\Cref{sec:proofs} for more details.
\Cref{assum:A5} is a regularity condition on the smoothness of the error density functions near the origin. Commonly used distributions, including the double-exponential and various stable distributions like Cauchy distribution, all meet this assumption. \Cref{assum:A6} is a standard condition on the kernel function
$G(\cdot)$ for approximating the Dirac delta function, which is the derivative of the indicator function. An example of $G(\cdot)$ could be: 
\begin{equation}
\label{eq:smooth-kernel}
G(u)=
\begin{cases}
0 & \text{if $~\left| u\right|\ge 1$};\\
- \frac{{315}}{{64}}{u^6} + \frac{{735}}{{64}}{u^4} - \frac{{525}}{{64}}{u^2} + \frac{{105}}{{64}} & \text{if $~\left| u\right|< 1$}.
\end{cases}
\end{equation}

Under these conditions, \Cref{thm:oracle} gives the error
rates for the estimation and prediction with the oracle Trans-Lasso QR
estimator in Section \ref{sec:oracle-estimation}. 

\begin{theorem}[Convergence rate of the oracle estimator]
  \label{thm:oracle}
  \sloppy Suppose Assumptions \ref{assum:A1}--\ref{assum:A6} hold true. We take ${\lambda_1} \asymp \sqrt {\log ( {p \vee {\overline{n}_{\calA}}} )/( {{n_{\cal A}} + {n_0}} )} $ and $\lambda_0 \asymp \lambda _2 \asymp \sqrt {\log ( {p \vee {n_0}} )/{n_0}} $, then with
  \begin{equation*}
  \xi_1 \left( {n_0,n_{\calA},{\overline{n}_{\calA}} ,s_0,p} \right) =  {\frac{{{s_0}\log \left( {p \vee {\overline{n}_{\calA}}} \right)}}{n_{\calA} + n_0} + \frac{s_0\log \left( {p \vee n_0} \right)}{n_0} \wedge C_{\bSigma} ^{\calA}d\sqrt {\frac{{\log \left( {p \vee {n_0}} \right)}}{n_0}}  \wedge {{\left( C_{\bSigma} ^{\calA}d \right)}^2}},
  \end{equation*}
  we have
  \begin{align*}
  \phantom{=\,}&\P\left(\frac{1}{{{n_0}}}\left\|{{\bX^{(0)}}( {\widehat \bbeta  - {\bbeta ^{\left( 0 \right)}}} )}\right\|^2 \vee \left\|\hbbeta  - \bbeta^{(0)}\right\|^2 \lesssim \xi_1 \left( {n_0,n_{\calA},{\overline{n}_{\calA}} ,s_0,p} \right) \right)\\
  \ge \, & 1-C_1\exp\left(-C_2 {\underline{n}_{\calA}}\right) - C_3\left({\underline{n}_{\calA}}\right)^{-C_4 {\underline{s}_{\calA}}} - \exp\left(-C_5\log p\right),
  \end{align*}
  for some positive constants $C_1$,\ldots, $C_5$, where  ${\underline{s}_{\calA} } = {\min _{k \in {\calA \cup \{ 0 \}}}}{s_k}$.
\end{theorem}

Compared to the target-only estimator whose rate of squared $\ell_2$ error bound on $\bbeta$ is
${s_0}\log ({p \vee n_0})/ {n_0}$
\citep{fan2014adaptive}, the oracle transfer learning estimator has a
sharper rate when ${C_{\bSigma}^{\cal A}}d \le \sqrt {{s_0}\log ( {p \vee {n_0}} )/{n_0}}$ and
${n_\mathcal{A}} \gg{n_0}$. 
In the worst-case scenario that there is no informative source, i.e. $\mathcal{A} = \emptyset $, the oracle estimator has a
convergence rate ${{s_0}\log ( {p \vee {n_0}} )/{n_0}}$, which is no
worse than that of the target-only estimator. Theorem~\ref{thm:oracle} also reveals that the smaller the heterogeneity across the informative sources and the target, as measured by $C_{\bSigma}^\mathcal{A}$, the better the model performance.

Now we consider the informative source detection problem. We define ${{\bbeta ^{( {0,k} )}}}$ as the underlying fusion coefficient vector between the target and the source $k$ arising from Step 1 of the proposed source detection procedure, 
  $${\bbeta ^{\left( {0,k} \right)}} = {\left( {\frac{{n_0}/2}{{{n_0}/2 + {n_k}}}{\bSigma ^{\left( 0 \right)}} + \frac{{{n_k}}}{{{n_0}/2 + {n_k}}}{\bSigma ^{\left( k \right)}}} \right)^{ - 1}}\left( {\frac{{{n_0}/2}}{{{n_0}/2 + {n_k}}}{\bSigma ^{\left( 0 \right)}}{\bbeta ^{\left( 0 \right)}} + \frac{{{n_k}}}{{{n_0}/2 + {n_k}}}{\bSigma ^{\left( k \right)}}{\bbeta ^{\left( k \right)}}} \right).$$

\begin{assumption}[Weak sparsity condition]\label{assum:A7} 
   With $s'$ such that ${\| {{\bbeta ^{( k )}}} \|_0} \le s'$ for all $k \in {{\cal A}^c}$, there exists a set ${S'_k} \subset \{ {1, \ldots ,p} \}$ such that $| {{S'_k}} | \le s'$ and ${\| {\bbeta _{{{S'_k}^c}}^{( {0,k} )}}\|_1} \le d'$ for any $k \in {{\cal A}^c}$ with $d' = o( 1 )$. 
\end{assumption}

\begin{assumption}[Identifiability of ${\cal A}$]\label{assum:A8} 
  Denote ${s^ * } = {s_0} \vee s'$ and ${d^ * } = C_{\bSigma} ^\calA d \vee d'$. Denote ${\lambda _{\min }}$ as the minimum eigenvalue of the covariance matrix ${\bSigma ^{( 0 )}}$ and
  \begin{equation*}
    {\xi _2}
    = \left( {\frac{{{{\left( {{s^ * }} \right)}^{3/2}}\log \left( {p \vee {n_0} \vee \underline{n}} \right)}}{{{n_0} + \underline{n}}} + {d^ * }\sqrt {\frac{{{s^ * }\log \left( {p \vee {n_0} \vee \underline{n}} \right)}}{{{n_0} + \underline{n}}}} } \right) \vee {s^ * }\sqrt {\frac{{\log \left( {p \vee {n_0}} \right)}}{{{n_0}}}} .
  \end{equation*}
For any $k \in {{\cal A}^c}$, there exists a positive constant $c_\varepsilon$ such that
  \begin{equation*}
  \left\| {{\bbeta ^{\left( k \right)}} - {\bbeta ^{\left( 0 \right)}}} \right\|^2 \ge \lambda _{\min }^{ - 1}\left\{ {c_\varepsilon {Q_0}\left( {{\bbeta ^{\left( 0 \right)}}} \right)+ 2\xi_2} \right\}.
\end{equation*}
Meanwhile, we require ${d^2} = o( {{Q_0}( {{\bbeta ^{( 0 )}}} )} )$.
\end{assumption}


To achieve consistent informative set detection, Assumption~\ref{assum:A7} 
assumes that the sparsity pattern of ${\bbeta ^{( k )}}, k \in \mathcal{A}^c$ is similar to ${\bbeta ^{( 0 )}}$, which implies that the corresponding fusion coefficient ${{\bbeta ^{( {0,k} )}}}$ remain to be ``weakly'' sparse \citep{tian2022transfer}. Assumption~\ref{assum:A8} 
states that the gap between the target and those non-informative sources, as measured by $\|{{\bbeta ^{( k )}} - {\bbeta ^{( 0 )}}}\|^2$, $k \in {{\cal A}^c}$, should be sufficiently large, and the gap between the target and those informative sources, as measured by $d$, should be relatively small. More specifically, it implies that
\begin{align*}
  \mathop {\inf }\limits_{k \in {\calA^c}} \left\| {{\bbeta ^{\left( k \right)}} - {\bbeta ^{\left( 0 \right)}}} \right\|_1^2 \ge
\,&\mathop {\sup }\limits_{k \in \calA} {\left\| {{\bbeta ^{\left( k \right)}} - {\bbeta ^{\left( 0 \right)}}} \right\|_1^2},
\end{align*}
and that ${c_\varepsilon } \le {\lambda _{\min }}{\| {{\bbeta ^{( k )}} - {\bbeta ^{( 0 )}}} \|^2}/{Q_0}( {{\bbeta ^{( 0 )}}} )$ for any $k \in {\mathcal{A}^c}$, which motivates the estimation of $c_\varepsilon$ in practice given in~\eqref{eq:E0suggest}.

The following theorem shows the consistency in identifying the informative set.

\begin{theorem}[Consistency of $\widehat {\cal A}$]\label{thm:detect} 
Suppose Assumptions~\ref{assum:A1}--\ref{assum:A8} hold. When ${\xi _2} = o( 1 )$,  $\widehat\calA$ obtained in~\eqref{eq:est-of-A} is consistent in identifying $\calA$, such that
  \begin{equation*}
  \P\left({\widehat {\cal A} = {\cal A}} \right) \ge 1-C_1\exp\left(-C_2 \left(\underline{n} \wedge n_0\right)\right) - C_3\left(\underline{n} \wedge n_0\right)^{-C_4 \underline{s}} -2\exp\left(-C_5\log p \right)
  \end{equation*}
  for some positive constants $C_1$ to $C_5$, where $\underline{s} = \min_{k = 0}^{K} {s_k}$.
\end{theorem}

\Cref{thm:detect} guarantees that the Trans-Lasso QR estimator enjoys the same
 convergence rate as the oracle Trans-Lasso QR estimator shown in \Cref
 {thm:oracle}. All the proofs are provided in~\Cref{sec:proofs}.

\section{Simulation Studies}\label{sec:simulation}

We conduct simulation studies to compare the following estimators:
\begin{itemize}
  \item Lasso QR, i.e. $\ell_1$-penalized quantile regression with only
    target sample;
  \item Naive Trans-Lasso QR, which naively treats all the sources as informative without any detection procedure;
  \item Oracle Trans-Lasso QR, which assumes known informative set;  
  \item Pseudo Trans-Lasso QR, which performs informative set detection with known number of informative sources $m$;  
  \item Trans-Lasso QR, which is the proposed methods that performs informative set detection without any prior knowledge; 
  \item Oracle Trans-Lasso Pooled QR assuming known informative set, presented in Section 2 of \citet{huang2022transfer}.
  \item Trans-Lasso Pooled QR with informative set detection, presented in Sections 4 and 5 of \citet{huang2022transfer}.
\end{itemize}
In the Pseudo Trans-Lasso QR, with known $m$, the informative set is estimated as
\begin{equation}
  \widehat {\cal A} = \left\{ {k \ne 0:{{\widehat Q}_0}\left( {{\hbbeta ^{\left( 0,k \right)}}} \right) \le {{\widehat Q}_0}\left( {{\hbbeta ^{\left( 0,k \right)}}} \right)_{\left( m \right)}} \right\},
\end{equation}
where ${{\widehat Q}_0}( {{\hbbeta^{( 0,k )}}} )_{( m )}$ represents the $m$th order statistic among all the
empirical loss ${{\widehat Q}_0}( {{\hbbeta^{( 0,k )}}})$ calculated by Step 2 in the informative set detection. In the Trans-Lasso QR, for simplicity, we set ${\varepsilon_0} = 0.01$ in \eqref{eq:est-of-A} for the informative set detection, which falls into $\left( {0,{{\widehat c}_\varepsilon }} \right]$ with ${{{\widehat c}_\varepsilon }}$ calculated by \eqref{eq:E0suggest} in all of our settings. 



\subsection{Homogeneous Design}\label{sim:homo}

We take $p\in \{75, 150, 225\}$, $n_0 = 150$ and ${n_1} = {n_2} =  \cdots  =
{n_K} = 150$ with $K=20$. The covariate vectors $\bx^{( k )} \in \mathbb
{R}^p$ are i.i.d. from the Gaussian distribution with mean zero and covariance
${\bSigma} =[{0.5^{| {i - j} |}}]_{p\times p}$, for all $k$ from $0$ to $K$.
Note that $\bSigma$ is assumed to be the same across the target and all the
sources, which we refer to as the ``homogeneous design'' scenario.

For the target, the coefficient vector $\bbeta ^{( 0 )}$ is set as $\beta _j^{
( 0 )} = 1$ for $j \in \{ {1,2, \ldots,{s_0}} \}$ with $s_0=20$, and $\beta _j^
{( 0 )} = 0$ otherwise. Similar to \citet{li2020transfer}, the coefficient
vectors ${\bbeta ^{( k )}}$ for the sources are set as follows. Let ${H_k}$ be
a random subset of $\{1,\ldots,p\}$ with size $| {{H_k}} | = p/2$.
\begin{itemize}
  \item If $k \in {\cal A}$, let
  \begin{equation*}
  \beta _j^{\left( k \right)} = \beta _j^{\left( 0 \right)} + {\xi _j}I\left( {j \in {H_k}} \right),
  \end{equation*}
  where ${\xi _j}\mathop  \sim \limits^{i.i.d.} Laplace( {0,2d/p} )$.
  We have that  $\E{\|{\bbeta ^{( k )}} -
  {\bbeta ^{( 0 )}}\|_1} = \sum\nolimits_{j \in {H_k}} {\E|
  {{\xi _j}} |}  = d$. 
  \item If $k \notin {\cal A}$, let
  \begin{equation*}
  \beta _j^{\left( k \right)} = \beta _j^{\left( 0 \right)} + {\xi _j}I\left( {j \in {H_k}} \right),
  \end{equation*}
  where ${\xi _j}\mathop  \sim \limits^{i.i.d.} Laplace( {0,140/p})$.
  We have that $\E{\|{\bbeta ^{( k )}} - {\bbeta ^{( 0 )}}\|_1} = 70$.
\end{itemize}
We consider different levels of informativeness with $d \in \{ {2,20,60} \}$ and different numbers of informative sources $| {\cal A} | \in \{{0,1,2,4,6,8,10,12,14,16,18,20}\}$. 


We set the quantile level to be the same for the target and the sources as $\tau = 0.8$. For the error $\eps_i^{( k
 )}$,  we consider the following two cases:
\begin{itemize}
  \item Case 1: $\eps_i^{( k )} \sim N( {{\Phi ^{ -
  1}}( {0.2} ),{\bbeta ^{( k )\top}}\bSigma {\bbeta ^{(
  k )}}/\eta } )$ with \bruce{the signal-to-noise ratio (SNR) $\eta$ selected from} $\eta  \in \{
  {20,10,5} \}$, where ${{\Phi ^{ - 1}}}$ is the quantile function of
  normal distribution with mean $0$ and variance ${\bbeta ^{( k )\top}}\bSigma
  {\bbeta ^{( k )}}/\eta$.
  \item Case 2: $\eps_i^{( k )} \sim Cauchy( {{\Omega ^{
  - 1}}( {0.2} ),{\bbeta ^{( k )\top}}\bSigma {\bbeta
  ^{( k )}}/\eta } )$ with $\eta  \in \{ {20,10,5} \}$, where ${{\Omega ^{ - 1}}}$ is the quantile
  function of Cauchy distribution with mean $0$ and scale parameter ${\bbeta ^{( k )\top}}\bSigma {\bbeta ^{( k )}}/\eta$.
\end{itemize}

The data are then generated by the target and source models in \eqref{eq:target-model} and \eqref{eq:source-model}. Each setting is replicated by 100 times. 


To successfully implement the transfer learning methods, appropriate
smoothing kernel and hyper-parameters are desired. For the choice of the
smoothing kernel in the response surrogation step,  we take the kernel $G$ from \eqref{eq:smooth-kernel} and use 10-fold cross validation to select the bandwidth $h$, by minimizing the quantile loss of the lasso estimator with surrogate response. 
In the 
single-source modeling step, we use 10-fold cross validation to tune $\lambda_0$. For the parameters in the fusion learning step and debiasing step, ${\lambda _1}$ is selected by 10-fold cross validation. Following \citet{li2020transfer}, with the selected $\lambda_1$, we then set ${\lambda _2} = {\lambda _1}\sqrt {( {{n_{\cal A}} + {n_0}} )/ {n_0}}$ as guided by the results in Theorem~\ref{thm:oracle}.

To evaluate the performance of different methods, we compute the mean squared errors (MSE) on estimating the target coefficient vector $\bbeta^{(0)}$, i.e., $\|{\hbbeta^{( 0 )}} - {\bbeta ^{( 0 )}}\|^2$, 
and the out-of-sample quantile loss (QL), i.e., $\sum\nolimits_{i = 1}^{n_t} {{\rho _\tau }} ( {y_i^{( t )} - \bx_i^{( t )\top}{\hbbeta^{( 0 )}}} )/{n_t}$, where $( {y_i^{( t )},\bx_i^{( t )}})_{i = 1}^{{n_t}}$ is an independently generated testing data of size $n_t = 150$ from the target model. We report 10\% trimmed mean and standard deviation of the quantile loss values from 100 repetitions, to mitigate the impact of extreme outliers in the Cauchy error setting. In contrast, the mean and standard deviation of the MSE values are reported without trimming.

\begin{figure}[tbp]
\centering
\includegraphics[width=1.0\textwidth]{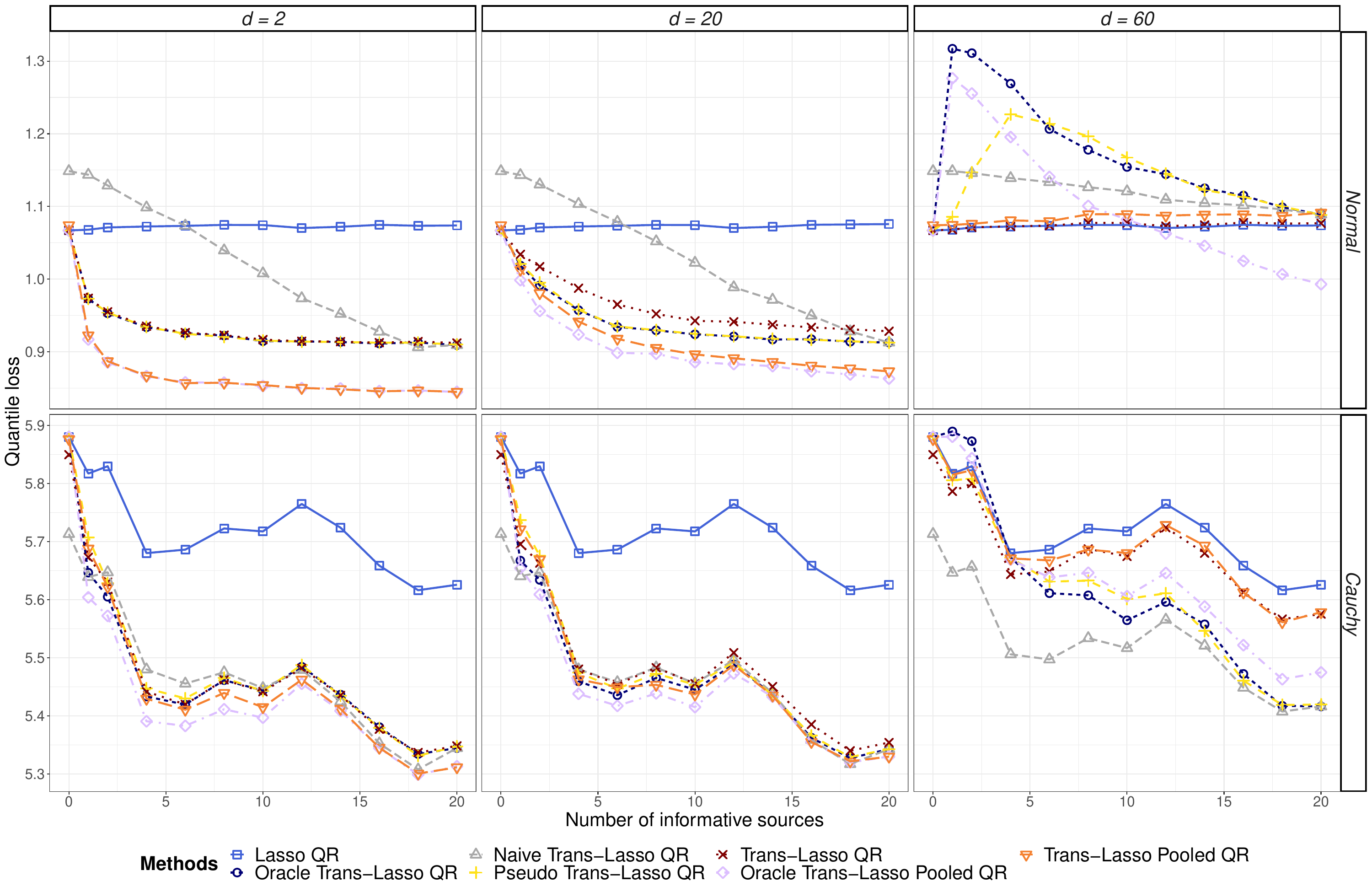}
\caption{Average quantile loss among different $d$ with $\eta=20$ in homogeneous setting for normal and Cauchy error.}
\label{fig:homo-loss}
\end{figure}

\begin{figure}[tbp]
\centering
\includegraphics[width=1.0\textwidth]{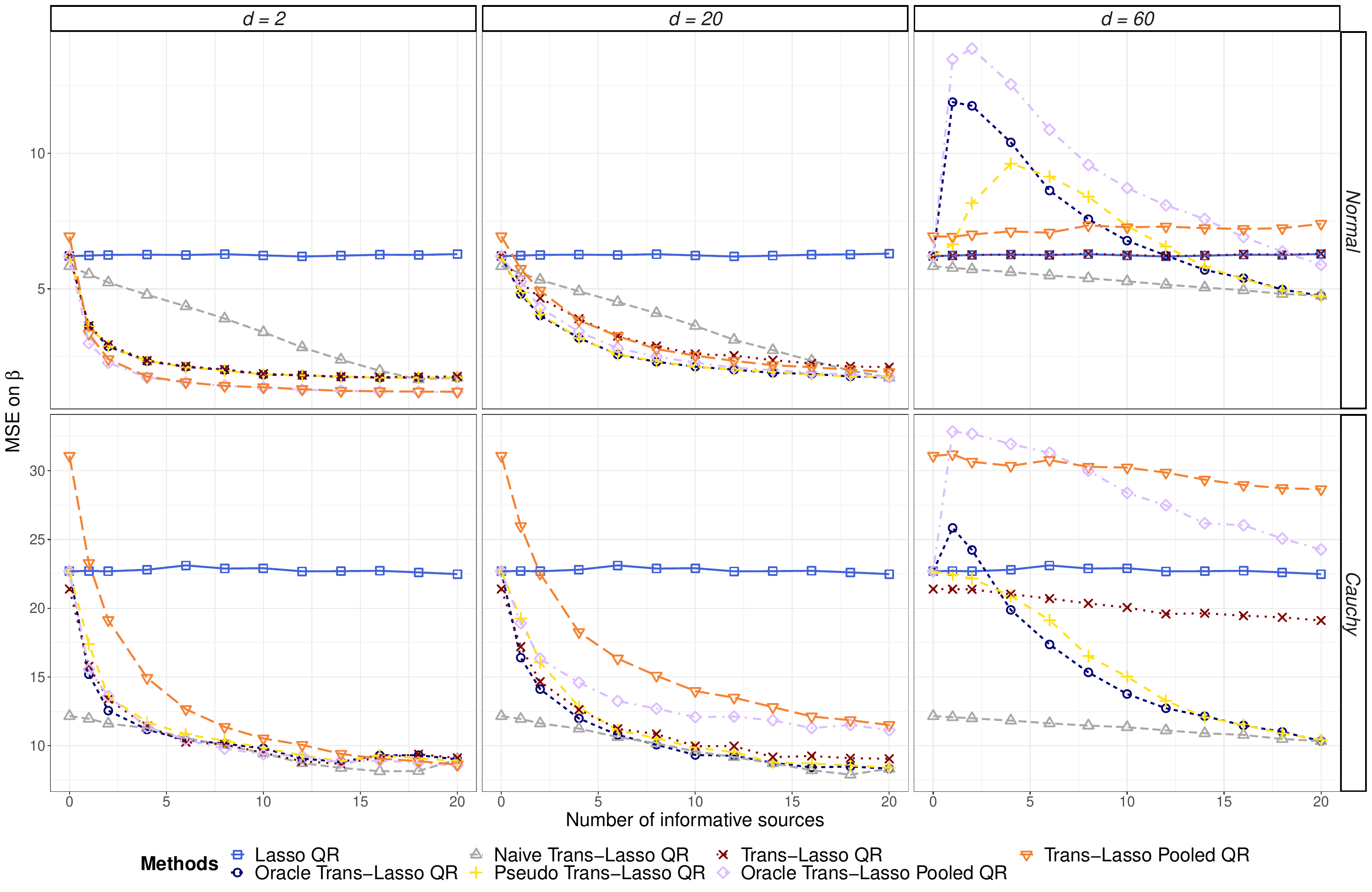}
\caption{Average MSE on $\bbeta$ among different $d$ with $\eta=20$ in homogeneous setting for normal and Cauchy error.} 
\label{fig:homo-beta}
\end{figure}

Figures~\ref{fig:homo-loss}-\ref{fig:homo-beta} plot the QL values and
the MSE values, respectively, for the cases with $p=150$ and $\eta = 20$,
averaged over replications. Detailed results are also shown in Tables~\ref{tab:ci-normal-mse}-\ref
{tab:ci-cauchy-ql} in the Appendix.

When $d=2$ or $d=20$, the informative sources are truly informative, as seen by the fact that all the transfer learning methods could greatly outperform Lasso QR, except for the naive approach. In these cases, the performance of both Trans-Lasso QR and Trans-Lasso Pooled QR closely aligns with their corresponding oracle procedures, affirming the efficacy of the informative source detection process. Interestingly, the pseudo Trans-Lasso QR, informed a priori about $|{\cal A} |$, performs comparably to the oracle. Conversely, the naive Trans-Lasso QR, which indiscriminately includes all sources, may underperform relative to Lasso QR when $| {\cal A} |$ is small.

When $d = 60$, the informativeness of the sources diminishes, particularly in the normal residual setting, which is inherently simpler than the Cauchy residual setting. This is reflected in the performance of the oracle Trans-Lasso QR and oracle Trans-Lasso Pooled QR: under normal settings, they significantly underperform compared to Lasso QR with a small $| {\cal A} |$, but their performance incrementally improves as $| {\cal A} |$ increases, primarily due to the larger sample size of weakly informative samples. The pseudo Trans-Lasso QR exhibits similar behavior to the oracle, as it is constrained to use sources that are not genuinely informative. It is then interesting to see that the proposed Trans-Lasso QR method is very robust and performs as well as Lasso QR. The results show that the proposed informative source selection procedure works well; the Trans-Lasso QR is able to utilize the truly information sources to achieve positive transfer while not being misled by the ``fake'' informative sources that could cause negative transfer.

  It is interesting to further compare Trans-Lasso QR and Trans-Lasso Pooled QR. In terms of the quantile loss, Trans-Lasso Pooled QR shows better performance than Trans-Lasso QR, especially under the normal residual setup with a small $d$ value. This is not surprising as the Trans-Lasso Pooled QR directly targets on the originally quantile loss. However, in terms of the MSE on $\bbeta$, Trans-Lasso Pooled QR can be substantially outperformed by Trans-Lasso QR, especially in the Cauchy residual setup. This may be due to the instability or optimization issues with fitting the pooled sparse quantile regression models. We have also conducted additional simulation studies to compare their computational efficiency with large-scale problems. Specifically, consider the settings of Gaussian residuals, with $d = 2$, $| {\mathcal{A}} | = 10$, $p \in \{75, 150, 300, 600\}$, and $n = {n_0} =  \cdots  = {n_{20}} \in \{1000, 2500, 5000,
  10000\}$. Figure~\ref{fig:time_vs_pn} highlights the superior computational efficiency of Trans-Lasso QR over Trans-Lasso Pooled QR; the efficiency gain is consistent and substantial especially in large-scale settings with large $p$ and/or $n$. More detailed results are reported in Tables~\ref{tab:c1h1ls_mse}-\ref{tab:c1h1ls_time} in the Appendix.

We also observe that, as expected, in general all the methods perform better
under normal error than under Cauchy error, and all the methods perform better
with smaller $d$ and larger SNR; the different versions of transfer learning
methods perform better with larger number of informative sources as long as $d$
is not too large.

To save space, we report the simulation results for $\eta \in \{5,10\}$ with
normal residuals and $p \in \{ 75,225\}$ under $\eta = 10$ in Tables~\ref
{tab:c1h1eta10-mse}-\ref{tab:c2h2p225-ql} in the Appendix. 
The results indicate the superiority of the proposed approach under both 
low-dimension and high-dimension scenarios.

\begin{figure}[tbp]
\minipage{0.46\textwidth}
  \includegraphics[width=\linewidth]{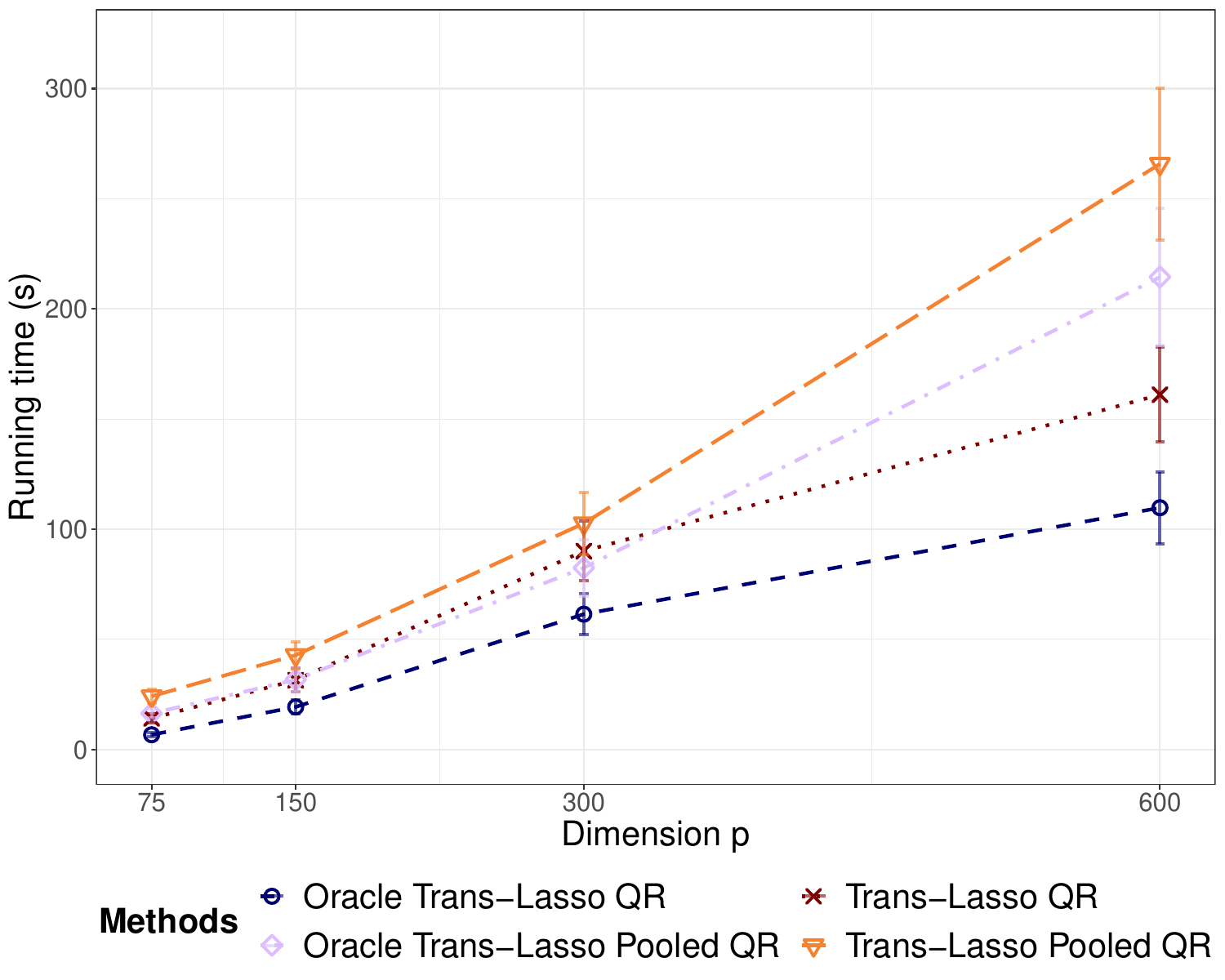}
  \subcaption{Average running time comparison under $n = 10000$ with different $p$}\label{fig:time_vs_p}
\endminipage\hfill
\minipage{0.46\textwidth}
  \includegraphics[width=\linewidth]{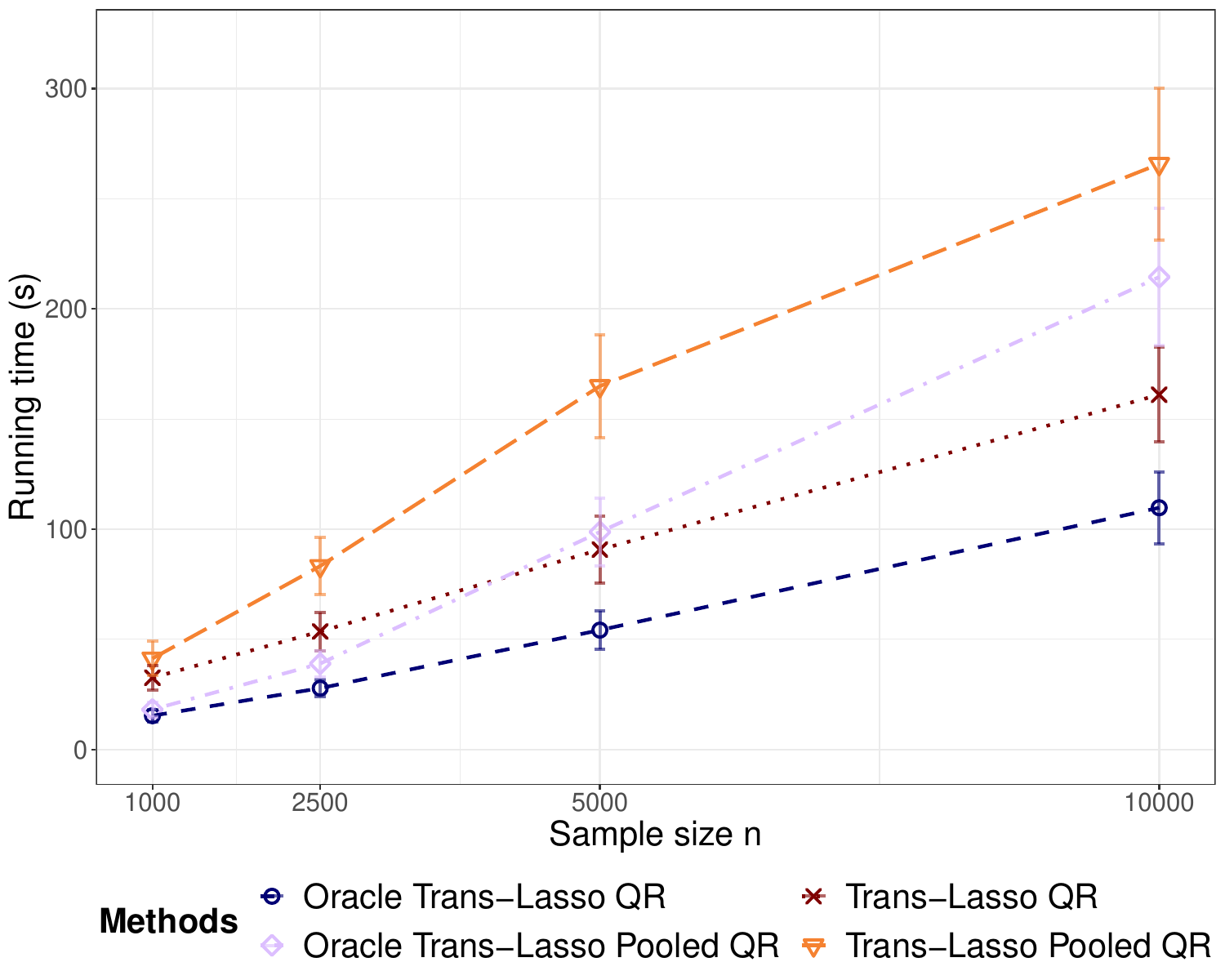}
  \subcaption{Average running time comparison  under $p = 600$ with different $n$}\label{fig:time_vs_n}
\endminipage
\caption{Running time comparison between Trans-Lasso QR and Trans-Lasso Pooled QR. Each setting is replicated 100 times, and each error bar extends $\pm 1.96 \times sd$ around the average. The experiment was run on a computational cluster with 4 Intel Xeon Gold 6230 CPUs at 2.10 GHz.} 
\label{fig:time_vs_pn}
\end{figure}

\subsection{Heterogeneous Designs}
\label{subsec:hetero}

We assume that the covariates from the target are generated from the Gaussian distribution with mean zero and the identity covariance matrix, and the covariates from the $k$th source are generated from the Gaussian distribution with mean zero and a covariance matrix with a Toeplitz structure
\begin{singlespace}
\begin{equation*}
\bSigma  = \left( {\begin{array}{*{20}{c}}
{\begin{array}{*{20}{c}}
{{\sigma _0 \:}}&{{\sigma _1 \:}}&{{\sigma _2}}\\
{{\sigma _1 \:}}&{{\sigma _0 \:}}&{{\sigma _1}}\\
{{\sigma _2 \:}}&{{\sigma _1 \:}}&{{\sigma _0}}
\end{array}}&{\begin{array}{*{20}{c}}
 \cdots \\
{}\\
{}
\end{array}}&{\begin{array}{*{20}{c}}
{}&{{\sigma _{p - 2}}}&{{\sigma _{p - 1}}}\\
{}&{}&{{\sigma _{p - 2}}}\\
{}&{}&{}
\end{array}}\\
{\begin{array}{*{20}{c}}
 \vdots &{\quad}&{\quad}
\end{array}}& \ddots &{\begin{array}{*{20}{c}}
{\quad}&{\quad}& \vdots 
\end{array}}\\
{\begin{array}{*{20}{c}}
{}&{}&{}\\
{{\sigma _{p - 2}}}&{}&{}\\
{{\sigma _{p - 1}}}&{{\sigma _{p - 2}}}&{}
\end{array}}&{\begin{array}{*{20}{c}}
{}\\
{}\\
 \cdots 
\end{array}}&{\begin{array}{*{20}{c}}
{}&{}&{}\\
{}&{{\sigma _0 \:}}&{{\sigma _1}}\\
{}&{{\sigma _1 \:}}&{{\sigma _0}}
\end{array}}
\end{array}} \right).
\end{equation*}
\end{singlespace}
with its first row given as  
\begin{equation*}
(\sigma_0,\ldots, \sigma_{p-1}) = ( {1,\underbrace {1/\left( {k + 1} \right), \ldots ,1/\left( {k + 1} \right)}_{2k - 1},{0_{p - 2k}}} ).
\end{equation*}
We refer to this as the ``heterogeneous design'' Scenario. The other settings are same as in Section \ref{sim:homo}.

The simulation results are reported in Figures~\ref{fig:hete-beta}-\ref{fig:hete-loss} and Tables~\ref{tab:cii-normal-mse}-\ref{tab:cii-cauchy-ql} in~\Cref{sec:addition-result}. We omit detailed discussions as the main observations are consistent with those made in the homogeneous design scenario and with our theoretical results.

\section{Application on Airplane Hard Landing}\label{sec:real-data}

In aviation industry all over the world, quick access record (QAR) recorders are widely installed on
aircraft to collect real-time data. With such real-time QAR data, statistical learning and machine learning approaches can be used to help prevent
incidents during a flight's most crucial time - approaching and
landing. For example, \citet{hong2008fuel} analyzed QAR data for fuel control, \citet{lan2012flight} 
investigated QAR data for the safety of landing at
high-attitude airports, and \citet{sun2012research} used QAR data to
characterize pilots with their operations.

We study the risk of hard landing, i.e., high vertical acceleration at the touchdown point of a flight, with QAR data for 3 types of airplanes: Boeing 737 (B737), Airbus A320
(A320), and Airbus A380 (A380). The maximum vertical acceleration (MVA) during landing is a commonly-used standard to measure the loading during landing \citep{wang2014analysis,qian2017improved}, so it can serve as a numerical proxy of the risk of hard landing and is used as the response variable in our study. The sensors on the airplane record 15 flight attributes during the flight, including Airspeed, Altitude, Bank Angle, Flap
 Angle, Gross Weight, Ground Speed, Wind Speed and Direction, Stabilizer Angle,
 Temperature, Fuel Consumption, Fuel Flow, and the Acceleration. In our study, the covariates consist of these attributes measured at several critical time points before or during the landing procedure, as well as various summary statistics (average, max, etc) of these attributes computed for different phases of a flight. A full list of the covariates is provided in Section~\ref{sec:feature-dictionary} in the Appendix. 
 Indeed, previous catastrophes indicate that there exist strong linkages between the MVA (or the loading on the landing gear) and some flight attributes captured in QAR. For instance, in 1992, Martinair Flight 495 crashed on its landing runway with a high loading partially because of the low airspeed caused by a micro-burst.

To study the associations between the MVA and the flight attributes, quantile regression is more suitable than mean regression for various reasons. First, we mainly concern those high-risk flights with higher than usual MVA values. Second, the empirical distribution of MVA is generally highly right skewed as seen from Figure~\ref{fig:qqmva} in the Appendix, which may make the normality assumption in mean regression models inappropriate. 
Another reason is that the associations between flight attributes and the MVA often exhibit heteroscedasticity; see Figure~\ref{fig:hetero-show} in the Appendix for an example. However, existing works on modeling the hard landing with QAR data take a target-only modeling approach by focusing on a single type of aircraft \citep{qian2017improved} or a one-size-fits-all approach by modeling all types of aircraft together \citep{chen2021detection}.  

We apply the proposed transfer learning approaches to perform quantile regression analyses of the MVA on the QAR features at $\tau = 0.9$, for three types of flights, namely, Boeing 737 (B737), Airbus A320 (A320), and Airbus A380 (A380). We consider each of the three types of airplanes as the target and the rest two types as the sources. Under each of the three settings, we randomly split the target data to 80\% training and 20\% testing. All the features are standardized within each dataset prior to performing the train-test split. The training data are used to fit QR models and the tuning parameters are selected by the cross validation procedure as described in Section~\ref{sec:simulation}. The testing data are then used to evaluate the final tuned models. We consider Lasso QR, the Naive Trans-Lasso QR (assuming that all sources are informative), and the proposed Trans-Lasso QR methods. For the case that A380 is the target, we have also fitted transfer learning models with all different choices of informative set. The random splitting procedure is repeated 20 times under each setting, from which we compute the means and the standard deviations of the out-of-sample quantile loss values for different methods.

Figure~\ref{fig:realdata}(a) shows the results of the out-of-sample quantile loss values for the three settings of different targets. For each bar, its height represents the average quantile loss value and its error bar extends two standard deviations to indicate how the individual loss values are dispersed around the average. The Trans-Lasso QR method is the most beneficial for modeling A380 as the target, where it leads to substantially lower quantile loss than Lasso QR and Naive Trans-Lasso QR. This may be due to the fact that the sample size for A380 is quite low so it is crucial to borrow information from other similar types of flights. Indeed, the number of flight records for A380 is around 10,000, which is much smaller than that of B737 and A320, which exceed 630,000. On the other hand, neither B737 nor A320 appears to benefit from transfer learning, as the sample size for each type is already quite adequate. Nevertheless, the performance of Trans-Lasso QR is at least comparable to Lasso QR and is always slightly better than Naive Trans-Lasso QR, which suggests that Trans-Lasso QR is quite robust and can avoid negative transfer with data-driven informative set selection.

Figure~\ref{fig:realdata}(b) shows the results of using different informative sets for the case that A380 is the target. Very interestingly, the results suggest that using A320 as the single informative set leads to the best transfer learning model for A380. By cross referencing the two figures, it can be seen that this informative set is also the one selected by the proposed Trans-Lasso QR method. Indeed, A380 is more similar to A320 than to B737 since the former two are essentially cousins from the same company. The Boeing series aircrafts adopt the control wheel and
control column to operate, while the Airbus aircrafts are equipped with the
control stick. With a control stick that uses wire technology, the operating
commands are delivered to the computer for translation and assignment, however,
with the control wheel and column, the signal is delivered physically, and it
directly connects with all the control surfaces including several flaps, slats,
stabilizers, etc. See \citet{tomczyk2002proposal} and \citet{filburn2020flight} for more discussions about this topic. This difference in operating mechanisms may lead to
the difference in landing performance under even the same external conditions, making A320 a more informative source than B737 to A380.

\begin{figure}[tbp]
\minipage{0.46\textwidth}
  \includegraphics[width=\linewidth]{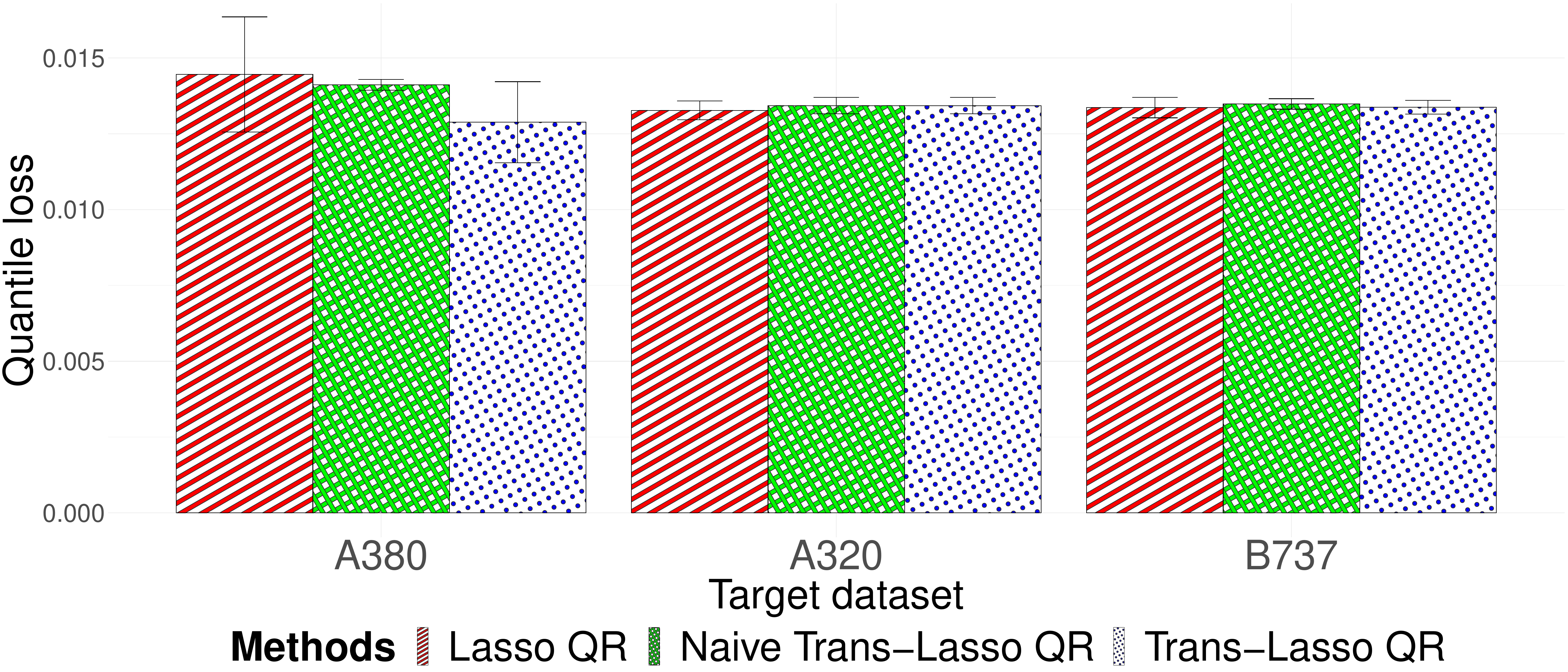}
  \subcaption{Quantile loss on different selected target datasets with different algorithms.}\label{fig:realdata-loss}
\endminipage\hfill
\minipage{0.46\textwidth}
  \includegraphics[width=\linewidth]{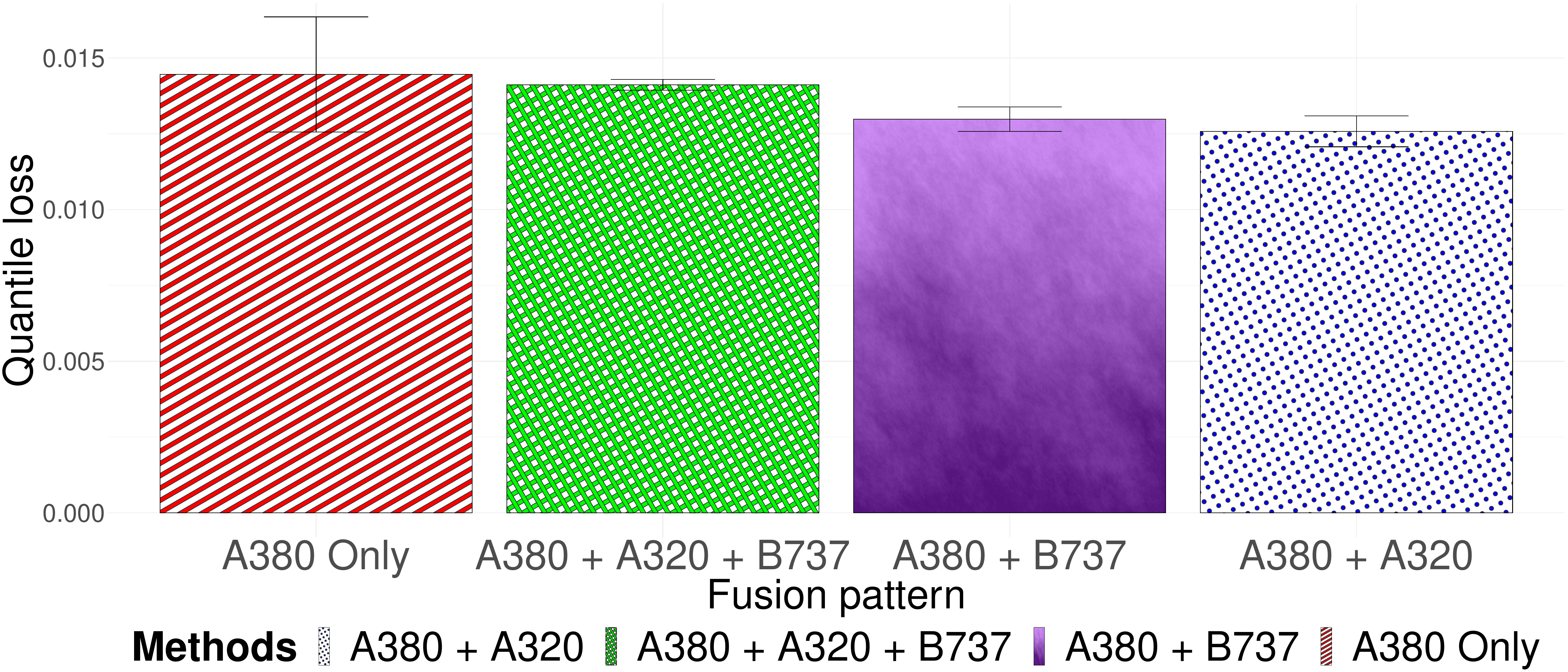}
  \subcaption{Quantile loss on A380 as the target with all different fusion combinations.}\label{fig:realdata-combine}
\endminipage
\caption{Performances comparison for the hard landing analysis with given quantile level $\tau=0.9$.} 
\label{fig:realdata}
\end{figure}

\begin{figure}[tbp]
  \centering
  \includegraphics[width=0.45\linewidth]{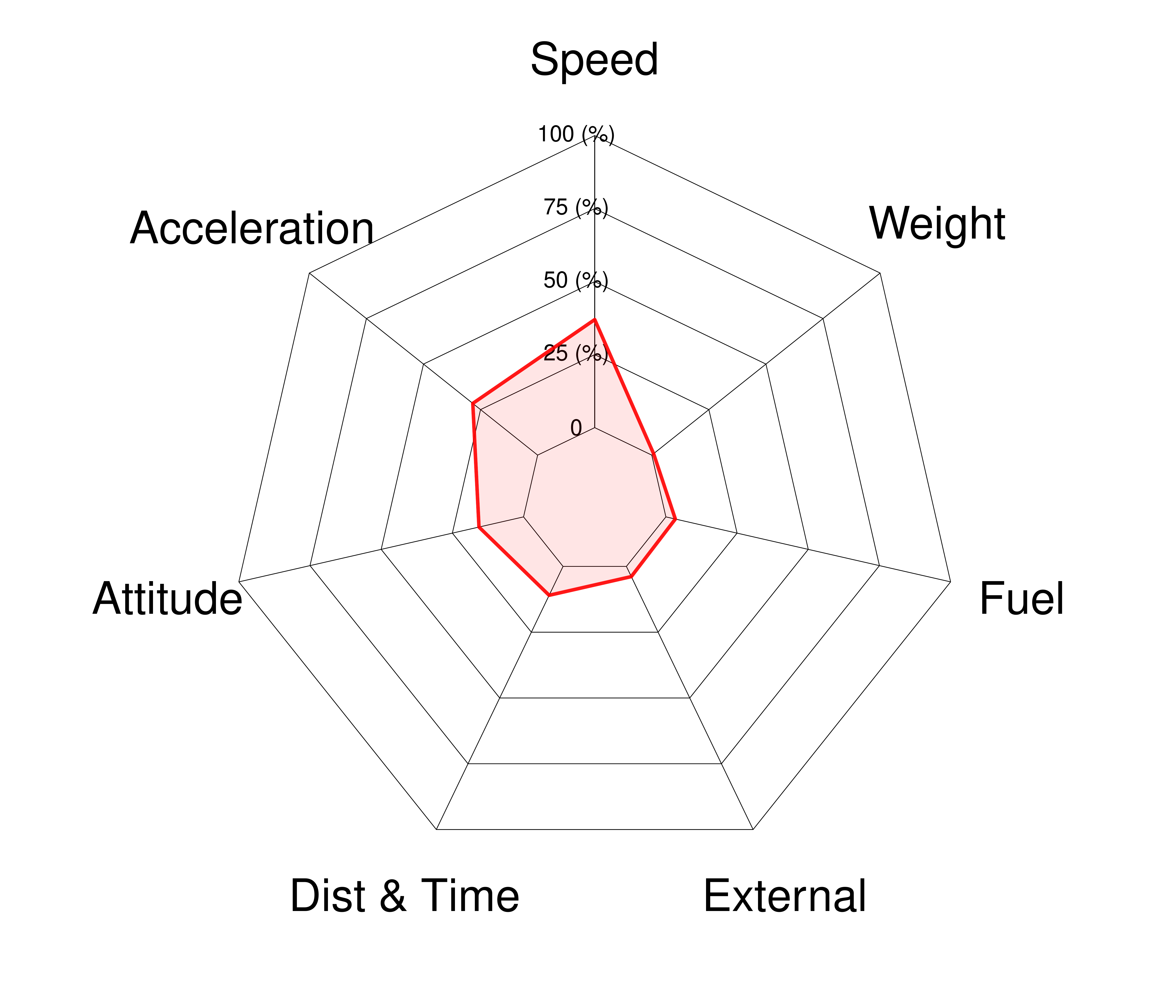}
  \caption{Radar plot for categorized feature contribution from A380 Tran-Lasso QR with $\tau=0.9$.}\label{fig:radplot}
\end{figure}

To gain more insights on the risk factors of hard landing for A380, we categorize all features and quantify their group-level relative contributions to the MVA, based on the fitted Trans-Lasso QR model with A380 being the target. Specifically, all features are categorized into seven categories: Speed, Fuel (consumption), Distance and Time, Acceleration (non-vertical), Weight, Attitude, and External. The detailed categorization is provided by the 
 ``Type'' column of the feature table in~\Cref{sec:feature-dictionary}. To measure the contribution of each category of features, we compute the sum of the absolute values of their estimated coefficients. 
These measures from different feature categories are then divided by the total contributions from all the categories to make them summing up as $1$ and become a set of relative contributions within $\left[ {0,1} \right]$. Figure~\ref
{fig:radplot} visualizes the relative contributions of the seven feature categories using a radar plot. We can see that speed, acceleration and attitude measurements appear to be the more relevant to the hard landing risk. Interesting, many features in these three categories could potentially be controlled or altered by the pilot team in normal circumstances, while the features in the other three of four categories, namely, fuel, weight, and external measurements, arguably, are mostly external and ``uncontrollable''. 
As such, the results suggest that it is hopeful to reduce the risk of hard landing by minimizing the bad impact or maximizing the good impact from the controllable ``human factors''. In other words, training better pilots may be the most important.

\begin{figure}[htp]
\centering
\includegraphics[width=1\textwidth]{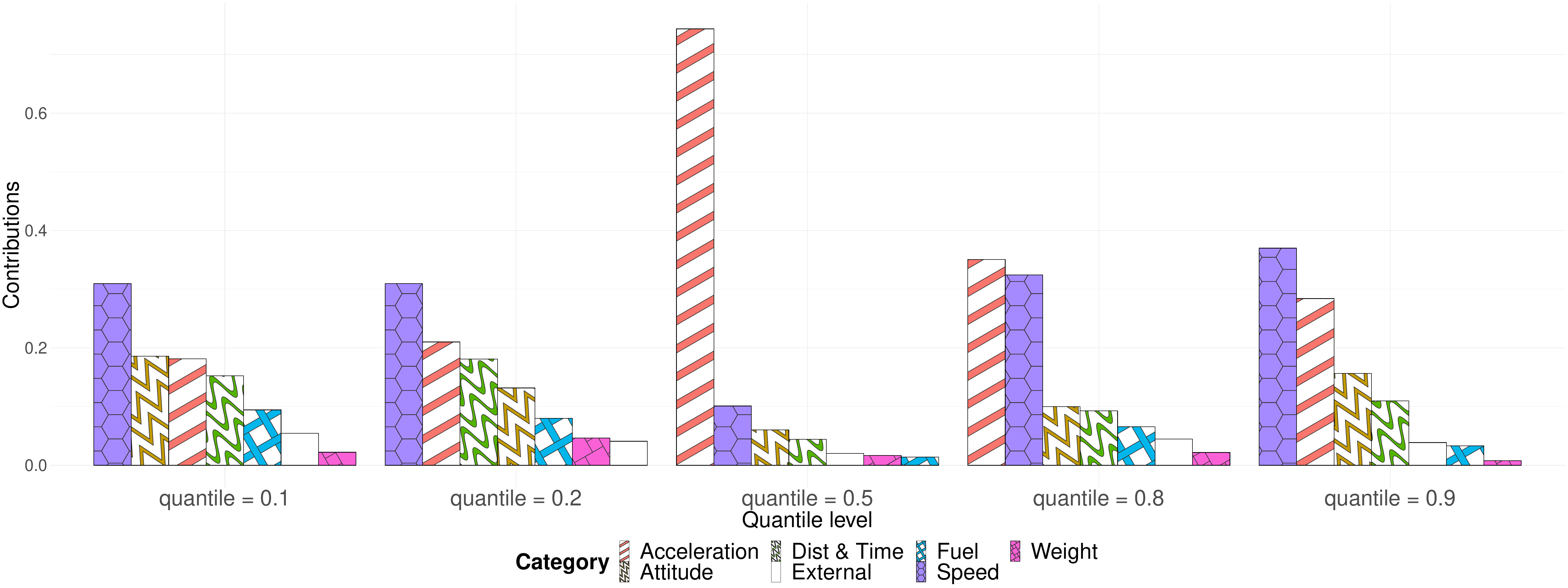}
\caption{Barplot of relative contributions of seven feature categories from A380 Trans-Lasso QR models with varying quantile levels. 
}
\label{fig:bar_all}
\end{figure}

It is also interesting to compare the risk factors for hard landing across quantile levels. We thus apply the Trans-Lasso QR with A380 being the target under three quantile levels, namely, $\tau \in \{ 0.1, 0.2, 0.5, 0.8, 0.9\}$, where $\tau \in \{0.1,0.2\}$ represent good landing, $\tau = 0.5$ represents normal landing, and $\tau \in \{0.8,0.9\}$ represent hard landing. Figure~\ref{fig:bar_all} show the comparison of the relative contributions of the risk factors from the fitted models with different quantile levels. 
There are several interesting observation. The three controllable or human risk factors, namely, speed, acceleration, and attitude, are always among the most important factors, and their overall importance tends to be higher with the deterioration of the quality of landing. For good landings, the ``External'' factor also plays an important role besides the three human factors. For normal landing, the acceleration factor stands out as a single most important factor, while for hard landing, both acceleration and speed become important. This suggests that hard landing is more likely to occur when both acceleration and speed go wrong.

\section{Discussion}\label{sec:discussion}

We have developed Trans-Lasso QR, to enable transfer learning for fitting high-dimensional quantile regression. Our method utilizes the idea of response surrogation to simplify the quantile loss and the idea of sampling splitting to distinguish informative sources for ensuring positive transfer. With QAR data, we are able to build an improved quantile regression models for the landing load of A380 airplane with the help from A320 and B737. 

There are several directions for future research. A potential extension is to investigate transfer learning for binary quantile regression through latent variable models, which may be preferred in modeling rare events with highly sparse binary outcomes. Moreover, it is interesting to consider the scenario of multiple targets, and how to fit multiple or multivariate transfer learning models in a collaborative fashion remains a challenging yet rewarding problem. Furthermore, enabling transfer learning of quantile regression without pooling raw data from the sources to the target could be very useful; this ``transfer learning without transfer'' approach would leverage the benefits of integrative learning while minimizing the need of data integration and preserving the confidentiality of individual sources. Another interesting problem is to utilize the composite quantile regression framework \citep{zou2008composite} under the transfer learning setup. We will also exploring the connections between transfer learning and hierarchical Bayesian modeling \citep{karbalayghareh2018optimal},  
  and apply the proposed approach to tackle more real-world problems in engineering and public health to better utilize data from disparate sources.

\appendix

\section{Proofs}\label{sec:proofs}

\subsection{Proofs of Theorem~\ref{thm:oracle}}

We acknowledge that the proof of Theorem~\ref{thm:oracle} follows similar structure as those of Theorem 1 and Theorem 4 in \citet{li2020transfer}, and the main challenge is to handle the quantile loss via a surrogate response \citep{chen2020distributed}. For the sake of completeness, we still present the relevant materials and results from existing works that are adapted to our setting. For simplicity, we do not use different notations for the constants appear in rates or probability expressions. For example, we always use $c_1$ to describe any constants in probability expressions like $1-\exp\{-c_1n_0\}$.

\begin{lemma}[Proposition 11 in \citet{chen2020distributed}]
  \label{lem-1}
  Denote ${\hbSigma^{( k )}} = \sum\nolimits_{i = 1}^{{n_k}} {{\bX^{( k )}}^\top} {\bX^{( k )}}/{n_k}$ for $k = 0,\ldots, K$. Recall 
${\tby^{( k )}}$ is the surrogate response as defined in \eqref{eq:surrogate} of the main paper. Under Assumption~\ref{assum:A1},~\ref{assum:A2},~\ref{assum:A4},~\ref{assum:A5} and~\ref{assum:A6}, for any $k = 0,\ldots, K$, we have
  \begin{align*}
                {{\left\| {\frac{1}{n_k}{\bX^{\left( k \right)^\top}}{\tby^{\left( k \right)}} - {\hbSigma^{\left( k \right)}}{\bbeta ^{\left( k \right)}}} \right\|}_\infty } \lesssim
   \sqrt {\frac{{\log \left( {p \vee {n_k}} \right)}}{{{n_k}}}}  + \frac{{{s_k}\log \left( {p \vee {n_k}} \right)}}{{{n_k}}}.
  \end{align*}
  with probability at least $1-c_1n_k^{ - c_{2} s_k}-c_3\exp ( { - {c_{4}}{n_k}} )$ for positive constants ${c_1}$ to ${c_4}$.
\end{lemma}

\begin{remark}
The Assumption~\ref{assum:A1} is necessary here by noticing that following the proof of Proposition 11 in \citet{chen2020distributed} will lead to
\begin{align*}
{{\left\| {\frac{1}{n_k}{\bX^{\left( k \right)^\top}}{\tby^{\left( k \right)}} - {\hbSigma^{\left( k \right)}}{\bbeta ^{\left( k \right)}}} \right\|}_\infty } \lesssim \sqrt {\frac{{\log p}}{{{n_k}}}}  + \sqrt {\frac{{{s_k}{a_n}\log {n_k}}}{{{n_k}}}}  + a_n^2 ,
\end{align*}
where ${a_n} = \sqrt {{s_k}\log {n_k}/{n_k}} $ is the convergence rate of the initial estimator of $\bbeta_k$. As such, to obtain the desired rate in Lemma S.1, we need ${s_k}{a_n} = O( 1 )$, i.e., $\sqrt {s_k^3\log {n_k}/{n_k}}  = O( 1 )$, which requires $s_k = O( {n_k^r} )$ with $r \in ( {0,1/3})$ as shown in Assumption~\ref{assum:A1}.

The Assumption~\ref{assum:A4} (Irrepresentable condition) is also necessary here to control the error of estimating the error density functions, which is preliminary for this Lemma~\ref{lem-1}.

Namely, with ${{\widehat f}^{( k )}}( 0 )$ defined in~\Cref{eq:est_density}, Assumption~\ref{assum:A4} contributes to obtain $$\left| {{{\widehat f}^{\left( k \right)}}\left( 0 \right) - {f^{\left( k \right)}}\left( 0 \right)} \right| = O\left( {\sqrt {\frac{{{s_k}\log {n_k}}}{{{n_k}{h^{\left( k \right)}}}}}  + \sqrt {\frac{{{s_k}\log {n_k}}}{{{n_k}}}} } \right).$$
Specifically, denote $$D_{n,{h^{\left( k \right)}}}^{\left( k \right)}\left( \bbeta  \right) = \frac{1}{{n{h^{\left( k \right)}}}}\sum\limits_{i = 1}^{{n_k}} {G\left( {\frac{{y_i^{\left( k \right)} - \bx_{i,{S_k}}^{\left( k \right)^\top}{\bbeta _{{S_k}}}}}{{{h^{\left( k \right)}}}}} \right)}, $$
in which $\bx_i^{(k)^\top}{\hbbeta ^{( k )}}$ is replaced by $\bx_{i,{S_k}}^{( k )^\top}{\bbeta _{{S_k}}}$. If ${{\hat S}_k} \subseteq {S_k}$ holds true with high probability (for which we would need the irrepresentable condition), then we have ${{{\widehat f}^{( k )}}( 0 ) = D_{n,{h^{( k )}}}^{( k )}( {{{\hat \bbeta }^{( k )}}} )}$ with high probability and can bound $| {{{\widehat f}^{( k )}}( 0 ) - {f^{( k )}}( 0 )} |$ by $$\left| {{{\widehat f}^{\left( k \right)}}\left( 0 \right) - {f^{\left( k \right)}}\left( 0 \right)} \right| \le \mathop {\sup }\limits_{\left\| {{\bbeta _{{S_k}}} - \bbeta _{{S_k}}^{\left( k \right)}} \right\| \le {a_n}} \left| {D_{n,{h^{\left( k \right)}}}^{\left( k \right)}\left( \bbeta  \right) - {f^{\left( k \right)}}\left( 0 \right)} \right|,$$
where ${a_n} = {\sqrt {s_k\log n_k/n_k} } $ is the convergence rate of the initial estimator of $\bbeta_k$ provided by \citet{fan2014adaptive}.

The main consequence, as shown in \citet{chen2020distributed}, is that the estimation of the density only relies on the $s_k$ dimensions of the initial estimator, not all the $p$ dimensions, which facilitates the next step of using sub-interval cutting technique on each dimension to further bound the error of estimating the error density. That is, for each dimension $i$ in the support index $S_k$, we divide the interval $[ {\beta _i^{( k )} - {a_n},\beta _i^{( k )} + {a_n}} ]$ into $n^M$ small subintervals and each has length $2{a_n}/{n^M}$, forming a set of points $\{ {{\bbeta _j},1 \le j \le {q_n}} \}$ in ${\mathbb{R}^p}$ with cardinality ${q_n} \le {n^{M{s_k}}}$ such that for any $\bbeta$ in the ball $\| {{\bbeta _{{S_k}}} - \bbeta _{{S_k}}^{( k )}} \| \le {a_n}$, we have $\| {{\bbeta _{{S_k}}} - \bbeta _{j,{S_k}}^{( k )}} \| \le 2\sqrt {{s_k}} {a_n}/{n^M}$ for some $1 \le j \le {q_n}$ and $\| {{\bbeta _{j,{S_k}}} - \bbeta _{{S_k}}^{( k )}} \| \le {a_n}$. Hence, we can further link ${\sup _{\| {{\bbeta _{{S_k}}} - \bbeta _{{S_k}}^{( k )}} \| \le {a_n}}}| {D_{n,{h^{( k )}}}^{( k )}( \bbeta  ) - {f^{( k )}}( 0 )} |$ with ${\sup _{1 \le j \le {n^{M{s_k}}}}}| {D_{n,{h^{( k )}}}^{( k )}( {{\bbeta _j}} ) - {f^{( k )}}( 0 )} |$ and bound the later one with the help of the exponential inequality. This leads to the desired bound for the error of estimating density functions.
\end{remark}

\begin{lemma}[Theorem 1 in \citet{raskutti2010restricted}]
  \label{lem-2}
Denote ${\hbSigma^{\cal A}} = \sum\nolimits_{k \in {\cal A} \cup \{ 0 \}} {{\alpha _k}{\hbSigma^{( k )}}} $, ${\bSigma^{\cal A}} = \sum\nolimits_{k \in {\cal A} \cup \{ 0 \}} {{\alpha _k}{\bSigma^{( k )}}} $ and ${\rho ^2}( \bSigma  ) = \mathop {\max }\limits_{j = 1, \ldots ,p} {\Sigma _{j,j}}$. Under Assumption~\ref{assum:A2}, we have
  \begin{equation*}
{u^\top}{\hbSigma ^{\cal A}}u \ge \frac{1}{32}{u^\top}{\bSigma ^{\cal A} }u - 81{\rho ^2}\left( \bSigma ^{\cal A}  \right)\frac{{\log p}}{n_{\cal A}+n_0}\left\| u \right\|_1^2
  \end{equation*}
  for any $u \in {\R^p}$ with probability at least $1 - \exp ( { - {c_{1}}{({n_\calA}+ n_0 )}} )$ and 
  \begin{equation*}
{u^\top}{\hbSigma ^{\left( 0 \right)}}u \ge \frac{1}{32}{u^\top}{\bSigma ^{\left( 0 \right)}}u - 81{\rho ^2}\left( {\bSigma ^{\left( 0 \right)}}  \right)\frac{{\log p}}{n_0}\left\| u \right\|_1^2
\end{equation*}
  for any $u \in {\R^p}$ with probability at least $1 - \exp ( { - {c_{1}}{n_0}} )$ for positive constant ${c_1}$.
\end{lemma}


\begin{lemma}[Bernstein's inequality (Theorem 1.13 of \citet{rigollet2015high})]
  \label{lem-3}
  We say a random variable $X$ follows a sub-exponential distribution with parameter $\lambda$ (denoted as $X \sim subE( \lambda  )$) if $\E[ X ] = 0$ and its moment generating function satisfies
  \begin{equation*}
\E\left[ {{e^{sX}}} \right] \le \exp \left( {\frac{1}{2}{s^2}{\lambda ^2}} \right),\begin{array}{*{20}{c}}
{}&{\forall \left| s \right| \le \frac{1}{\lambda }}.
\end{array}
  \end{equation*}
  If ${X_1}, \ldots ,{X_n}$ are independent random variables with $\E{X_i} = {\mu _i}$ and ${X_i} - {\mu_i} \sim subE( {\lambda} )$, then for any $t >0$ we have
  \begin{equation*}
  \P\left( {\frac{1}{n}\left| {\sum\limits_{i = 1}^n {\left( {{X_i} - {\mu _i}} \right)} } \right| \ge t} \right) \le \exp \left( { - \frac{n}{2}\min \left\{ {\frac{{{t^2}}}{\lambda ^2},\frac{t}{\lambda}} \right\}} \right).
  \end{equation*}
\end{lemma}

We now give the outline of proving Theorem~\ref{thm:oracle}, followed by detailed derivations. Denote ${\bbeta ^\calA}$ as the probabilistic limit of ${\hbbeta ^\calA}$ in Step 3 of the oracle algorithm, which is the solution of the following moment equation
\begin{align}
\label{eq:moment_informative}
\E\left[ {\sum\limits_{k \in {\calA} \cup \left\{ 0 \right\}} {{{\left( {{\bX^{\left( k \right)}}} \right)}^\top}\left( {{{\tilde y}^{\left( k \right)}} - {\bX^{\left( k \right)}}{\bbeta ^{\calA}}} \right)} } \right] = 0,
\end{align}
where $\widetilde y^{(k)} = {\bx^{(k)^\top}}{\bbeta ^{( k )}} - \{ {f^{( k )}}( 0 ) \}^{-1}\{ {I( {y^{(k)} - {\bx^{(k)^\top}}{\bbeta
      ^{( k }} \le 0} ) - \tau} \}.$ Thus, ${\bbeta ^{\calA}}$ has the explicit form as ${\bbeta ^\calA} = {\bbeta ^{( 0 )}} + {\bdelta ^\calA}$, where ${\bdelta ^{\cal A}} = {( {{\bSigma ^{\cal A}}} )^{ - 1}}\sum_{k \in {\cal A} \cup \{ 0 \}} {{\alpha _k}{\bSigma ^{( k )}}{\bdelta ^{( k )}}}$. It implies ${\bbeta ^{\calA}}$ is a pooled version of ${\bbeta ^{( 0 )}}$ and ${\bbeta ^{( k )}}$ for all $k \in \calA$ such that ${\bbeta ^{\cal A}} = {( {{\bSigma ^{\cal A}}} )^{ - 1}}\sum_{k \in {\cal A} \cup \{ 0 \}} {{\alpha _k}{\bSigma ^{( k )}}{\bbeta ^{( k )}}}$.
Define
\begin{align*}
{E_1} = \, & \left\{ {\frac{1}{{{n_{\cal A}} + {n_0}}}{\left\|\sum\limits_{k \in \calA \cup \left\{ 0 \right\}} {{{ {\bX^{(k)^\top}} }}( {\tby^{(k)} - \bX^{(k)}{\bbeta^\mathcal{A} }} )} \right\|_\infty } \le \frac{{{\lambda _1}}}{2}} \right\},\\
E_2 = \, & \left\{ { \frac{1}{{{n_0}}}\left\|{{{\bX^{(0)^\top}} }}\left( {\tby^{(0)} - \bX^{(0)}{\bbeta ^{\left( 0 \right)}}} \right)\right\|_\infty  \le \frac{{{\lambda _2}}}{2}} \right\}.
\end{align*}
We will show the following statements:
\begin{itemize}
  \item (i) Denote ${\hbu^{\cal A}} = \hbbeta ^{\calA} - \bbeta ^{\calA}$. On the event of $E_1$, we have \begin{align*}
  \hbu^{\calA^\top} \hbSigma ^{\cal A} \hbu^{\cal A} \vee {\left\| {\hbu^{\cal A}} \right\|^2} \lesssim \, & {{s_0}\lambda _1^2 + {\lambda _1}C_{\bSigma} ^{\cal A}d}, \\
  \left\| {\hbu^{\cal A}} \right\|_1 \lesssim \, &  {{s_0}{\lambda _1} + C_{\bSigma} ^{\cal A}d} 
  \end{align*}
  with probability at least $1 - \exp ( { - c_{1}{({n_\calA}+ n_0)}} )$.
  \item (ii) Denote ${\hbv^{\cal A}} = \hbdelta ^{\cal A} - \bdelta ^{\cal A}$. On the event of $E_2$, we have
  \begin{align*}
  {\hbu^{\calA^\top}}{{\hbSigma }^{\left( 0 \right)}}{\hbu^\calA} \vee {\left\| {{\hbv^\calA}} \right\|^2} \lesssim \, &  {s_0\lambda _1^2 + {\lambda _1}C_{\bSigma} ^\calA d}, \\
  {{\hbv^{\calA}}^\top}{\hbSigma^{\left( 0 \right)}}{\hbv^\calA} \lesssim \, & {{\lambda _2}C_{\bSigma} ^\calA d},\\
  {\left\| {{{\hbv}^\calA}} \right\|^2} \lesssim \, &  {{\lambda _2}C_{\bSigma} ^\calA d \wedge {{\left( {C_{\bSigma} ^\calA d} \right)}^2}}
  \end{align*}
  with probability at least $1 - 3\exp ( { - c_{1}{n_0}} )$.
\end{itemize}
We will then show that $\P\left( {{E_1} \cap {E_2}} \right) \rightarrow 1$ as $p \rightarrow \infty$ and ${\underline{n}_{\calA}} \rightarrow \infty$, where ${\underline{n}_{\calA}} =\mathop {\min }\limits_{k \in {\cal A} \cup \left\{ 0 \right\}} {n_k}$ is the minimum sample size among informative sources and the target.
Finally, the proof of Theorem~\ref{thm:oracle} is completed by utilizing Lemma~\ref{lem-2} and the fact that
\begin{align*}
\left\| {\hbbeta  - {\bbeta ^{\left( 0 \right)}}} \right\| = \, & \left\| {\left( {{\hbbeta ^\calA} - {\hbdelta^\calA}} \right) - \left( {{\bbeta ^\calA} - {\bdelta ^\calA}} \right)} \right\|\\
 \le \, & \left\| {\hbu^\calA} \right\| + \left\| {\hbv^\calA} \right\|.
\end{align*}

We now prove statement (i). We have the following oracle inequality, the same as (A.1) in the supplementary material of \citet{li2020transfer}. For simplicity of the notation, we denote $S = {S_0}$. By applying Taylor expansion on the loss function in~\eqref{eq:fusion-lasso}, 
\begin{align*}
  \phantom{ =\, } & \frac{1}{2}{\hbu^{\calA^\top}}\hbSigma ^{\cal A}\hbu^{\cal A} \\
  \le \,& {\lambda _1}\left\|\bbeta^{\calA}\right\|_1 - {\lambda _1}\left\|\hbbeta^{\cal A}\right\|_1 + \left|\frac{{{{ {\hbu^{\calA^\top}}}}}}{{\left( {{n_\calA} + {n_0}} \right)}}\sum\limits_{k \in \calA \cup \left\{ 0 \right\}} {{{ {\bX^{(k)^\top}} }}\left( {\tby^{(k)} - \bX^{(k)}\bbeta^{\calA}} \right)}\right|\\
   \le \, & {\lambda _1}\left\|\bbeta^{\calA}\right\|_1 - {\lambda _1}\left\|\hbbeta^{\calA}\right\|_1 + \frac{{{\lambda _1}}}{2}\left\|\hbu^{\calA}\right\|_1\\
 \le \, & \frac{3}{2}{\lambda_1}{\left\| {\hbu_S^{\cal A}} \right\|_1} + {\lambda _1}{\left\| {\bbeta _{{S^c}}^{\cal A}} \right\|_1} - {\lambda _1}{\left\| {\hbbeta _{{S^c}}^{\cal A}} \right\|_1} + \frac{1}{2}{\lambda _1}{\left\| {\hbu_{{S^c}}^{\cal A}} \right\|_1}\\
 \le \, & \frac{3}{2}{\lambda_1}{\left\| {\hbu_S^{\cal A}} \right\|_1} + \frac{3}{2}{\lambda _1}{\left\| {\bbeta _{{S^c}}^{\cal A}} \right\|_1} - \frac{1}{2}{\lambda _1}{\left\| {\hbu_{{S^c}}^{\cal A}} \right\|_1},
\end{align*}
where the third line holds true due to event $E_1$ and H{\"o}lder's inequality, the fourth line holds true due to ${\| {\hbbeta _S^\calA - \bbeta _S^\calA} \|_1} \ge {\| {\hbbeta _S^\calA} \|_1} - {\| {\bbeta _S^\calA} \|_1}$, and the last line holds true due to ${\| {\hbu_{{S^c}}^\calA} \|_1} \le {\| {\hbbeta _{{S^c}}^\calA} \|_1} + {\| {\bbeta _{{S^c}}^\calA} \|_1}$.

  If ${\| {\hbu_S^{\cal A}} \|_1} \ge {\| {\bbeta _{{S^c}}^{\cal A}} \|_1}$, with a positive constant $c_4$, we can further have
  \begin{align*}
{\left\| {{\hbu^\calA}} \right\|^2} - \frac{{\log p}}{{{n_\calA} + {n_0}}}\left\| {{\hbu^\calA}} \right\|_1^2 \lesssim \, & \frac{1}{2}{\hbu^{\calA^\top}}\hbSigma ^{\cal A}\hbu^{\cal A}\\
 \le \, & 3{\lambda _1}{\left\| {\hbu_S^{\cal A}} \right\|_1} - \frac{1}{2}{\lambda _1}{\left\| {\hbu_{{S^c}}^\calA} \right\|_1}\\
 \le \, & 3{\lambda _1}{\left\| {\hbu_S^{\cal A}} \right\|_1}\\
 \le \, & 3{\lambda _1}\sqrt {{s_0}} {\left\| {\hbu_S^\calA} \right\|}\\
 \le \, & 3{\lambda _1}\sqrt {{s_0}} {\left\| {\hbu^\calA} \right\|},
\end{align*}
where the first line holds true with probability at least $1 - \exp ( { - {c_{1}}{({n_\calA}+ n_0)}} )$ due to Lemma~\ref{lem-2} and Assumption~\ref{assum:A4}, and the fourth line holds true due to $\ell_1$-$\ell_2$ inequality. It further leads to 
  \begin{align}
  {\left\| {\hbu^\calA} \right\|^2} \lesssim \, &  {s_0\lambda _1^2} \label{eq:inter-1}, \\ 
  {\left\| {\hbu^\calA} \right\|_1} \lesssim \, & {s_0{\lambda _1}} \label{eq:inter-2}
  \end{align}
  by utilizing ${s_0}\log p/( {{n_\calA} + {n_0}} ) = o( 1 )$ derived from Assumption~\ref{assum:A3}.

  On the other hand, if ${\| {\hbu_S^{\cal A}} \|_1} \le {\| {\bbeta _{{S^c}}^{\cal A}} \|_1}$, we have
  \begin{equation*}
  0 \le \frac{1}{2}{\hbu^{\calA^\top}}{\hbSigma^{\cal A}}{\hbu^{\cal A}} \le 3{\lambda _1}{\left\| {\bbeta _{{S^c}}^{\cal A}} \right\|_1} - \frac{1}{2}{\lambda _1}{\left\| {\hbu_{{S^c}}^{\cal A}} \right\|_1}.
  \end{equation*}
It follows that 
  \begin{equation*}
  {\left\| {\hbu_{{S^c}}^{\cal A}} \right\|_1} \le 6{\left\| {\bbeta _{{S^c}}^{\cal A}} \right\|_1}, 
  \end{equation*}
and
  \begin{equation}
  \label{eq:inter-3}
  {\left\| {{\hbu^\calA}} \right\|_1} = {\left\| {\hbu_S^\calA} \right\|_1} + {\left\| {\hbu_{{S^c}}^\calA} \right\|_1} \le 7{\left\| {\bbeta _{{S^c}}^\calA} \right\|_1} \le 7{\left\| {\bdelta _{{S^c}}^\calA} \right\|_1} \le 7C_{\bSigma}^\calA d,
  \end{equation}
  where the last step is due to the definition of $C_{\bSigma}^\calA $ in Assumption~\ref{assum:A2}. This implies ${\| {\hbu^\calA} \|} \le 7C_{\bSigma} ^\calA d$ by the $\ell_1$-$\ell_2$ inequality. By Lemma~\ref{lem-2}, we have 
  \begin{align*}
{\left\| {\hbu^\calA} \right\|^2} - \frac{{\log p}}{{{n_\calA} + {n_0}}}\left\| {\hbu^\calA} \right\|_1^2 \lesssim \, & \frac{1}{2}{\hbu^{\calA^\top}}{\hbSigma^{\cal A}}{\hbu^{\cal A}}\\
 \le \, & 3{\lambda _1}{\left\| {\bbeta _{S^c}^\calA} \right\|_1} - \frac{1}{2}{\lambda _1}{\left\| {\hbu_{S^c}^\calA} \right\|_1}\\
 \le \, & 3{\lambda _1}{\left\| {\bbeta _{S^c}^\calA} \right\|_1}\\
 \le \, & 3{\lambda _1}C_{\bSigma} ^\calA d
  \end{align*}
  with probability at least $1 - \exp ( { - {c_{1}}{({n_\calA}+ n_0)}} )$. 
  Noticing that $C_{\bSigma} ^\calA d\sqrt {\log p/( {{n_A} + {n_0}} )}  = o( 1 )$ as derived from Assumption~\ref{assum:A3}, we get 
  \begin{equation}
  \label{eq:inter-4}
  {\left\| {\hbu^\calA} \right\|^2} \lesssim {{\lambda _1}C_{\bSigma} ^\calA d}.
  \end{equation}
Statement (i) is then established by combining results from~\eqref{eq:inter-1},~\eqref{eq:inter-2},~\eqref{eq:inter-3},~\eqref{eq:inter-4} and Lemma~\ref{lem-2}.

Next, we prove statement (ii). Under the event $E_2$, we have the following oracle inequality, the same as the one from the proofs of Theorems 1 and 4 in \citet{li2020transfer}. By applying the Taylor expansion on the loss in~\eqref{eq:debias}, 
  \begin{align*}
  \phantom{=\,}&\frac{1}{2}{ {\hbv^{\calA^\top}} }{\hbSigma^{\left( 0 \right)}}\hbv^{\calA} \\
  \le \, &{\lambda _2}\left\|{\bdelta^{\calA}}\right\|_1 - {\lambda _2}\left\|\hbdelta^{\calA}\right\|_1 + \left|\frac{{{{ {\hbv^{\calA^\top}} }}}}{{{n_0}}}\bX^{(0)^\top}\left\{ {\tby^{(0)} - \bX^{(0)}\left( {\hbbeta^{\calA} - {\bdelta^{\calA}}} \right)} \right\}\right|\\
  = \, &{\lambda _2}\left\|{\bdelta^{\calA}}\right\|_1 - {\lambda _2}\left\|\hbdelta^{\calA}\right\|_1 + \left|\frac{{{{{\hbv^{\calA^\top}}}}}}{{{n_0}}}\bX^{(0)^\top}\left\{ {\tby^{(0)} - \bX^{(0)}{\bbeta ^{\left( 0 \right)}} - \bX^{(0)}\left( {\hbbeta^{\calA} - {\bdelta^{\calA}} - {\bbeta ^{\left( 0 \right)}}} \right)} \right\}\right|\\
  = \, &{\lambda _2}\left\|{\bdelta^{\calA}}\right\|_1 - {\lambda _2}\left\|\hbdelta^{\calA}\right\|_1 + \left|\frac{{{{{\hbv^{\calA^\top}}}}}}{{{n_0}}}\bX^{(0)^\top}\left\{ {\left( {\tby^{(0)} - \bX^{(0)}{\bbeta ^{\left( 0 \right)}}} \right) - \bX^{(0)}\hbu^{\calA}} \right\}\right|\\
  \le \,& {\lambda _2}\left\|{\bdelta^{\calA}}\right\|_1 - {\lambda _2}\left\|\hbdelta^{\calA}\right\|_1 + \frac{{{\lambda _2}}}{2}\left\|\hbv^{\calA}\right\|_1 + { {\hbu^{\calA^\top}} }{\hbSigma^{\left( 0 \right)}}\hbu^{\calA} + \frac{1}{4}{ {\hbv^{\calA^\top}} }{\hbSigma^{\left( 0 \right)}}\hbv^{\calA},
  \end{align*}
where the last step is due to the event $E_2$ and the inequality
  $|{ab}|\le {a^2} + {b^2}/4$. Hence, by noticing ${\| {\hbv^\calA} \|_1} = {\| {{\hbdelta ^\calA} - {\bdelta ^\calA}} \|_1} \ge \| {\hbdelta ^\calA} \|_1 - {\| {\bdelta ^\calA} \|_1}$, we have
  \begin{equation*}
  \frac{1}{4}\hbv^{\calA^\top}{\hbSigma^{\left( 0 \right)}}\hbv^{\calA} \le {2\lambda _2}\left\|{\bdelta^{\calA}}\right\|_1 - \frac{{{\lambda _2}}}{2}\left\|\hbv^{\calA}\right\|_1 + { {\hbu^{\calA^\top}} }{\hbSigma^{\left( 0 \right)}}\hbu^{\calA}.
  \end{equation*}

  If ${ {\hbu^{\calA^\top}} }{\hbSigma^{( 0 )}}\hbu^{\calA} \le {\lambda _2}\|{\bdelta^{\calA}}\|_1$, then
  \begin{align}
  0 \le \frac{1}{4}\hbv^{\calA^\top}{\hbSigma^{\left( 0 \right)}}\hbv^{\calA} \le \, & {3\lambda _2}\left\|{\bdelta^{\calA}}\right\|_1 - \frac{{{\lambda _2}}}{2}\left\|\hbv^{\calA}\right\|_1 \nonumber \\
  \le \, & {3\lambda _2}\left\|{\bdelta^{\calA}}\right\|_1  \nonumber \\
  \le \, & 3{\lambda _2}C_{\bSigma} ^\calA d, \label{eq:inter-5}
  \end{align}
  which further leads to ${\| {\hbv^\calA} \|_1} \le 6{\| {\bdelta ^\calA} \|_1} \le 6C_{\bSigma} ^\calA d$ and 
  \begin{equation}
  \label{eq:inter-6}
  \left\| {\hbv^\calA} \right\| \le  \left\| {\hbv^\calA} \right\|_1 \le 6C_{\bSigma} ^\calA d.
\end{equation}
Meanwhile, from Lemma~\ref{lem-2}, we have 
  \begin{align*}
  {\left\| {\hbv^\calA} \right\|^2} - \frac{{\log p}}{{{n_0}}}\left\| {\hbv^\calA} \right\|_1^2 \lesssim \, & \frac{1}{4}\hbv^{\calA^\top}{\hbSigma^{\left( 0 \right)}}\hbv^{\calA}\\
  \le \, & {3\lambda _2}\left\|{\bdelta^{\calA}}\right\|_1 \\
  \le \, & 3{\lambda _2}C_{\bSigma} ^\calA d,
  \end{align*}
  with probability at least $1-\exp(-c_{1} n_0)$, and then
  \begin{equation}
  \label{eq:inter-7}
  {\left\| {\hbv^\calA} \right\|^2} \lesssim {\lambda _2}C_{\bSigma} ^\calA d 
  \end{equation}
  by noticing ${C_{\bSigma}^{\cal A}}d{{( {\log ( {p \vee {n_0}} )/{n_0}} )}^{1/2}} = o( 1 )$ as derived from Assumption~\ref{assum:A3}. 

On the other hand, if ${ {\hbu^{\calA^\top}} }{\hbSigma^{( 0 )}}\hbu^{\calA} \ge {\lambda _2}\|{\bdelta^{\calA}}\|_1$, we have
  \begin{equation*}
  0 \le \frac{1}{4}\hbv^{\calA^\top}{\hbSigma^{\left( 0 \right)}}\hbv^{\calA} \le { 3{\hbu^{\calA^\top}} }{\hbSigma^{\left( 0 \right)}}\hbu^{\calA} - \frac{{{\lambda _2}}}{2}\left\|\hbv^{\calA}\right\|_1.
  \end{equation*}
It follows that 
  \begin{align}
{\left\| {{{\hbv}^\calA}} \right\|_1} \le \, & \frac{6}{{{\lambda _2}}}{\hbu^{\calA^\top}}{\hbSigma^{\left( 0 \right)}}{{\hbu}^\calA} \nonumber \\
 \lesssim \, &  {\frac{{\left\| {{\hbu^\calA}} \right\|}^2}{{{\lambda _2}}}} \nonumber \\
 \lesssim \, & {\frac{{{s_0}\lambda _1^2 + {\lambda _1}C_{\bSigma} ^\calA d}}{{{\lambda _2}}}} \label{eq:inter-8}
  \end{align}
with probability at least $1- 2\exp(-c_{1}  n_0)$, where the second line holds true with probability at least $1-\exp(-c_{1}n_0)$ due to Theorem 1.6 of \citet{zhou2009restricted}, and the third line holds true with probability at least $1-\exp(-c_{1}(n_{\calA}+n_0))$ due to~\eqref{eq:inter-1} and~\eqref{eq:inter-4}. By Lemma~\ref{lem-2} and Theorem 1.6 of \citet{zhou2009restricted}, once again, we have
  \begin{align*}
  {\left\| {\hbv^\calA} \right\|^2} - \frac{{\log p}}{{{n_0}}}\left\| {\hbv^\calA} \right\|_1^2 \lesssim \, & \frac{1}{4}\hbv^{\calA^\top}{\hbSigma^{\left( 0 \right)}}\hbv^{\calA}\\
  \le \, & { 3{\hbu^{\calA^\top}} }{\hbSigma^{\left( 0 \right)}}\hbu^{\calA} \\
  \lesssim \, & {{\left\| {{\hbu^\calA}} \right\|}^2} \\
  \lesssim \, & {{{s_0}\lambda _1^2 + {\lambda _1}C_{\bSigma} ^\calA d}}
  \end{align*}
with probability at least $1-3\exp(-{c_1} n_0)$. Thus, from~\eqref{eq:inter-8} that shows $\| {{\hbv^\calA}} \|_1^2\log p/{n_0} \le {( {{\lambda _2}{{\| {{\hbv^\calA}} \|}_1}} )^2} = o( 1 )$, we have
  \begin{equation}
  \label{eq:inter-9}
  {\hbu^{\calA^\top}}{\hbSigma^{\left( 0 \right)}}{\hbu^{\calA}} \vee {\left\| {{\hbv^\calA}} \right\|^2} \lesssim {{s_0}\lambda _1^2 + {\lambda _1}C_{\bSigma} ^\calA d}.
  \end{equation}
Statement (ii) is established by combining results from~\eqref{eq:inter-5},~\eqref{eq:inter-6},~\eqref{eq:inter-7},~\eqref{eq:inter-8} and~\eqref{eq:inter-9}.

We now prove that $\P( {{E_1} \cap {E_2}} ) \ge 1 - c_1{({\underline{n}_{\calA}} )}^{-c_2 \underline{s}_{\calA}} - c_3\exp(-c_4{\underline{n}_{\calA}}) - \exp(-c_5\log p)$ for some positive constants $c_1$ to $c_5$. We utilize the inequality $\P( {{E_1} \cap {E_2}}) \ge \P( {{E_1}} ) + \P( {{E_2}} ) - 1$ and then bound $\P( {{E_1}} )$ and $\P( {{E_2}} )$ separately. For $\P( {{E_2}})$, we have $\P( {{E_2}}) \ge 1 - c_1n_0^{ - c_2 s_0} - c_3\exp (-c_4 n_0)$ by Lemma~\ref{lem-1}. For $\P( {{E_1}} )$, we bound it by the following decomposition:
  \begin{align*}
  \phantom{=\,}&\frac{1}{{{n_{\cal A}} + {n_0}}}{\left\|\sum\limits_{k \in \calA \cup \left\{ 0 \right\}} {{{ {\bX^{(k)^\top}} }}( {\tby^{(k)} - \bX^{(k)}{\bbeta^\mathcal{A} }} )} \right\|_\infty }\\
  = \, & {\left\|\frac{1}{{{n_{\cal A}} + {n_0}}}\sum\limits_{k \in \calA \cup \left\{ 0 \right\}} {{{{\bX^{(k)^\top}} }}\tby^{(k)}}  - \sum\limits_{k \in \calA \cup \left\{ 0 \right\}} {{\alpha _k}{\hbSigma^{( k )}}{\bbeta ^{( k )}}}  + \sum\limits_{k \in \calA \cup \left\{ 0 \right\}} {{\alpha _k}{\hbSigma^{( k )}}{\bbeta ^{( k )}}}  - {\hbSigma^{\cal A}}\bbeta^{\calA}\right\|_\infty }\\
  = \, & {\left\|\sum\limits_{k \in {\cal A} \cup \left\{ 0 \right\}} {{\alpha _k}\left( \frac{1}{{{n_k}}}{{{ {\bX^{(k)^\top}} }}\tby^{(k)} - {\hbSigma^{( k )}}{\bbeta ^{( k )}}} \right)}  + \sum\limits_{k \in \calA \cup \left\{ 0 \right\}} {{\alpha _k}{\left( \hbSigma^{( k )} - \bSigma^{( k )} \right) }{\bdelta ^{( k )}}}  - ( {{\hbSigma^{\cal A}} - {{\bSigma} ^{\cal A}}} )\bdelta^{\calA}\right\|_\infty }\\
  \le \, & \left\|\sum\limits_{k \in {\cal A} \cup \left\{ 0 \right\}} {{\alpha _k}\left( \frac{1}{{{n_k}}}{{{ {\bX^{(k)^\top}} }}\tby^{(k)} - {\hbSigma^{( k )}}{\bbeta ^{( k )}}} \right)} \right\|_\infty + \left\|\sum\limits_{k \in \calA \cup \left\{ 0 \right\}} {{\alpha _k}{\left( \hbSigma^{( k )} - \bSigma^{( k )} \right) }{\bdelta ^{( k )}}}\right\|_\infty \\
  \phantom{=\,} & + \left\| \left({{\hbSigma^{\cal A}} - {{\bSigma} ^{\cal A}}} \right)\bdelta^{\calA}\right\|_\infty \\
  \defas \, &  {T_1} + {T_2} + {T_3}\,,
  \end{align*}
  where the third line holds true due to ${\bbeta ^{( k )}} = {\bbeta ^{( 0 )}} + {\bdelta ^{( k )}}$ and ${\bbeta ^\calA} = {\bbeta ^{( 0 )}} + {\bdelta ^\calA}$. For $T_1$, based on Lemma~\ref{lem-1}, we have
  \begin{align*}
  {\alpha _k}\left\|\frac{1}{{{n_k}}}{{{\bX^{(k)^\top}} }}\tby^{(k)} - {\hbSigma^{( k )}}{\bbeta ^{( k )}}\right\|_\infty  \lesssim \, &  \frac{{{n_k}}}{{{n_{\cal A}} + {n_0}}}\left( {\sqrt {\frac{{\log \left( p \vee {n_k} \right)}}{{{n_k}}}}  + \frac{{{s_k}\log\left( p \vee {n_k} \right)}}{{{n_k}}}} \right)\\
 \lesssim \, &{\frac{{\sqrt {{n_k}\log \left( p \vee {n_k} \right)} }}{{{n_{\cal A}} + {n_0}}} + \frac{{{s_k}\log \left( p \vee {n_k} \right)}}{{{n_{\cal A}} + {n_0}}}} \\
 \lesssim \, &{\frac{{\sqrt {{n_k}\log \left( p \vee {n_k} \right)} }}{{{n_{\cal A}} + {n_0}}}} \\
 \lesssim \, & {\sqrt {\frac{{\log \left( p \vee {n_k} \right)}}{{{n_{\cal A}} + {n_0}}}} }
  \end{align*}
  for any $k \in \mathcal{A} \cup \{ 0 \}$ with probability at least $1 - c_1n_k^{ - c_2 s_k} - c_3\exp(-c_4 n_k)$, where the third line holds true due to ${s_k} = O( {n_k^r} )$ with $0 < r < 1/3$ from Assumption~\ref{assum:A1}. Hence, considering that $| \cal A |$ is bounded, 
  we have
  \begin{align}
  {T_1} \lesssim \, &  \mathop {\max }\limits_{k \in {\cal A} \cup \left\{ 0 \right\}} \sqrt {\frac{{\log \left( p \vee {n_k} \right)}}{{{n_{\cal A}} + {n_0}}}} \nonumber \\
  \lesssim \, & \sqrt {\frac{{\log \left( p \vee {{n_{\calA}}^*} \right)}}{{{n_{\cal A}} + {n_0}}}} \label{eq:E1_T1}
  \end{align} 
  with probability at least $1 - c_1{{\underline{n}_{\calA}} }^{-c_2 \underline{s}_{\calA}} - c_3\exp(-c_4 {\underline{n}_{\calA}} )$. For $T_2$ and $T_3$, similar with proofs of Theorems 1 and 4 of \citet{li2020transfer}, with a positive constant $c_5$, we have 
  \begin{align*}
\phantom{=\,}&\P\left( {{{\left\| {\sum\limits_{k \in {\calA} \cup \left\{ 0 \right\}} {{\alpha _k}\left( {{\hbSigma ^{\left( k \right)}} - {\bSigma ^{\left( k \right)}}} \right){\bdelta ^{\left( k \right)}}} } \right\|}_\infty } \ge t} \right)\\
 \le \, & \P\left( {\frac{1}{{{n_{\calA}} + {n_0}}}{{\left\| {\sum\limits_{k \in {\calA} \cup \left\{ 0 \right\}} {\left( {{\bX^{\left( k \right)^\top}}{\bX^{\left( k \right)}} - E\left[ {{\bX^{\left( k \right)^\top}}{\bX^{\left( k \right)}}} \right]} \right){\bdelta ^{\left( k \right)}}} } \right\|}_\infty } \ge t} \right)\\
 \le \, & p \cdot \P\left( {\mathop {\max }\limits_{j \le p} \frac{1}{{{n_{\calA}} + {n_0}}}\left| {{{\left\{ {\sum\limits_{k \in {\calA} \cup \left\{ 0 \right\}} {\left( {{\bX^{\left( k \right)^\top}}{\bX^{\left( k \right)}} - E\left[ {{\bX^{\left( k \right)^\top}}{\bX^{\left( k \right)}}} \right]} \right){\bdelta ^{\left( k \right)}}} } \right\}}_j}} \right| \ge t} \right)\\
 \le \, & p\mathop {\max }\limits_{j \le p} \exp \left( { - c_{4}\min \left\{ {\frac{{\left( {{n_\calA} + {n_0}} \right){t^2}}}{{2\mathop {\max }\limits_{k \in {\calA} \cup \left\{ 0 \right\}} {\bdelta ^{\left( k \right)^\top}}{\bSigma ^{\left( k \right)}}{\bdelta ^{\left( k \right)}}}},\frac{{\left( {{n_{\calA}} + {n_0}} \right)t}}{{\mathop {\max }\limits_{k \in {\calA} \cup \left\{ 0 \right\}} \sqrt {{\bdelta ^{\left( k \right)^\top}}{\bSigma ^{\left( k \right)}}{\bdelta ^{\left( k \right)}}} }}} \right\}} \right)\\
 = \, & p\mathop {\max }\limits_{j \le p,k \in {\calA} \cup \left\{ 0 \right\}} \exp \left( { - c_{4}\min \left\{ {\frac{{\left( {{n_{\calA}} + {n_0}} \right){t^2}}}{{2{\bdelta ^{\left( k \right)^\top}}{\bSigma ^{\left( k \right)}}{\bdelta ^{\left( k \right)}}}},\frac{{\left( {{n_{\calA}} + {n_0}} \right)t}}{{\sqrt {{\bdelta ^{\left( k \right)^\top}}{\bSigma ^{\left( k \right)}}{\bdelta ^{\left( k \right)}}} }}} \right\}} \right), 
  \end{align*}
  where the fourth line holds true by combining Lemma~\ref{lem-3} and the fact that $x_{i,j}^{( k )}\bx{_i^{{( k )}^\top}}{\bdelta ^{( k )}}$ is sub-exponential. Hence, applying similar procedure on $T_3$ with $T_2$, by taking $t \gtrsim \mathop {\max }\limits_{k \in {\calA} \cup \{ 0 \}} \sqrt {{\bdelta ^{( k )^\top}}{\bSigma ^{( k )}}{\bdelta ^{( k )}}\log p/( {{n_\calA} + {n_0}} )}$ and noticing that $\log p  = o( {n_{\calA} + {n_0}} )$, we have
  \begin{equation}
  \label{eq:E1_T23}
  \P\left( {\max \left( {{T_2},{T_3}} \right) \gtrsim \mathop {\max }\limits_{k \in {\cal A} \cup \left\{ 0 \right\}} \sqrt {{{ {{\bdelta ^{( k )^\top}}} }}{{\bSigma} ^{( k )}}{\bdelta ^{( k )}}\log p/( {{n_{\cal A}} + {n_0}} )} } \right) \le \exp\left( -c_{4}\log p \right).
  \end{equation}

  Combining~\eqref{eq:E1_T1},~\eqref{eq:E1_T23}, we have 
  \begin{equation*}
  \P\left( {{E_1}} \right) \ge 1 - c_1{\underline{n}_{\calA}}^{-c_2 \underline{s}_{\calA}} - c_3\exp\left( -c_{4} {\underline{n}_{\calA}}\right) - \exp\left( -c_{5}\log p \right),
  \end{equation*}
   which further leads to the desired conclusion that $\P( {{E_1} \cap {E_2}} ) \rightarrow 1$ as $p \rightarrow \infty$ and ${\underline{n}_{\calA}} \rightarrow \infty$. 

\subsection{Proofs of Theorem~\ref{thm:detect}}

The proof of Theorem~\ref{thm:detect} follows the same outline as the proof of Theorem 4 in \citet{tian2022transfer}.

We present several lemmas first. Recall that $Q_0$ is defined in~\eqref{eq:Q0}, ${\widehat Q}_0$ is defined in~\eqref{eq:emprical-loss}, and $\hbbeta ^{( 0,k )}$ is defined in Step 1 of informative set detection procedure. We define ${\bbeta ^{( {0,k} )}} = {\bbeta ^{( 0 )}} + {\bdelta ^{( {0,k} )}}$ where
\begin{equation*}
{\bdelta ^{\left( {0,k} \right)}} = {\left( {\frac{{{n_0}}}{{{n_0} + {n_k}}}{\bSigma ^{\left( 0 \right)}} + \frac{{{n_k}}}{{{n_0} + {n_k}}}{\bSigma ^{\left( k \right)}}} \right)^{ - 1}}\sum\limits_{j = 0,k} {\frac{{{n_j}}}{{{n_0} + {n_k}}}} {\bSigma ^{\left( j \right)}}{\bdelta ^{\left( j \right)}}.
\end{equation*} 
Hence, we have
\begin{equation*}
{\bbeta ^{\left( {0,k} \right)}} = {\left( {\frac{{{n_0}}}{{{n_0} + {n_k}}}{\bSigma ^{\left( 0 \right)}} + \frac{{{n_k}}}{{{n_0} + {n_k}}}{\bSigma ^{\left( k \right)}}} \right)^{ - 1}}\sum\limits_{j = 0,k} {\frac{{{n_j}}}{{{n_0} + {n_k}}}} {\bSigma ^{\left( j \right)}}{\bbeta ^{\left( j \right)}},
\end{equation*}
which indicates that ${\bbeta ^{( {0,k} )}}$ is a linear transform of ${\bbeta ^{( {0} )}}$ and ${\bbeta ^{( {k} )}}$.

\begin{lemma}
  \label{lem-4}
  Under Assumption~\ref{assum:A2}, we have
  \begin{equation*}
  \sup_{k \in {\cal A}} {Q_0}\left( {{\bbeta ^{\left( 0,k \right)}}} \right) - {Q_0}\left( {{\bbeta ^{\left( 0 \right)}}} \right)=O\left(  \sup_{k \in {\cal A}} \left\|{\bbeta ^{\left( 0,k \right)}} - {\bbeta ^{\left( 0 \right)}}\right\|^2\right)=O\left( {d^2}\right).
  \end{equation*}
\end{lemma}

\begin{proof}[Proof of Lemma~\ref{lem-4}]
  From the fact that 
  \begin{equation*}
  {Q_0}'\left( {{\bbeta ^{\left( 0 \right)}}} \right) = E\left[ {2{\bX^{\left( 0 \right)}}^\top\left( {{{\tby}^{\left( 0 \right)}} - {\bX^{\left( 0 \right)}}{\bbeta ^{\left( 0 \right)}}} \right)} \right] = 0, 
  \end{equation*}
  for any given $t \in ( {0,1} )$, we have
  \begin{align*}
\phantom{=\,}&{Q_0}'\left( {{\bbeta ^{\left( 0 \right)}} + t\left( {{\bbeta ^{\left( 0,k \right)}} - {\bbeta ^{\left( 0 \right)}}} \right)} \right)\\
 = \, &E\left[ {2{\bX^{\left( 0 \right)}}^\top\left\{ {{{\tby}^{\left( 0 \right)}} - {\bX^{\left( 0 \right)}}\left( {{\bbeta ^{\left( 0 \right)}} + t\left( {{\bbeta ^{\left( 0,k \right)}} - {\bbeta ^{\left( 0 \right)}}} \right)} \right)} \right\}} \right]\\
 = \, &2tE\left[ {{\bX^{\left( 0 \right)}}^\top{\bX^{\left( 0 \right)}}} \right]\left( {{\bbeta ^{\left( 0,k \right)}} - {\bbeta ^{\left( 0 \right)}}} \right)\\
 = \, &2t{\bSigma ^{\left( 0 \right)}}\left( {{\bbeta ^{\left( 0,k \right)}} - {\bbeta ^{\left( 0 \right)}}} \right).
  \end{align*}
  Then, by the mean value theorem, we have
  \begin{equation*}
  {Q_0}\left( {{\bbeta ^{\left( 0,k \right)}}} \right) - {Q_0}\left( {{\bbeta ^{\left( 0 \right)}}} \right) = 2t{\left( {{\bbeta ^{\left(0,k \right)}} - {\bbeta ^{\left( 0 \right)}}} \right)}^\top{\bSigma ^{\left( 0 \right)}}\left( {{\bbeta ^{\left( 0,k \right)}} - {\bbeta ^{\left( 0 \right)}}} \right) = O\left(  \left\|{\bbeta ^{\left( 0,k \right)}} - {\bbeta ^{\left( 0 \right)}}\right\|^2\right),
\end{equation*}
where the last step is due to Assumption~\ref
   {assum:A2}. The conclusion can then be reached by using the $\ell_1$-$\ell_2$ norm inequality and by noticing that ${\bbeta ^{( 0,k )}}$ is a linear transform of ${\bbeta ^{( 0 )}}$ 
   and ${\bbeta ^{( k )}}$.
\end{proof}

\begin{lemma}
  \label{lem-5}
  Recall ${s^ * } = {s_0} \vee s'$ and ${d^ * } = C_{\bSigma} ^\calA d \vee d'$. Denote
  \begin{align*}
  {\xi _{21}} = \frac{{{{\left( {{s^ * }} \right)}^{3/2}}\log \left( {p \vee {n_0} \vee \underline{n}} \right)}}{{{n_0} + \underline{n}}} + {d^ * }\sqrt {\frac{{{s^ * }\log \left( {p \vee {n_0} \vee \underline{n}} \right)}}{{{n_0} +\underline{n} }}} 
  \end{align*}
  and 
  \begin{align*}
  {\xi _{22}} = {s^ * }\sqrt {\frac{{\log \left( {p \vee {n_0}} \right)}}{{{n_0}}}} .
  \end{align*}
   Under Assumptions~\ref{assum:A1}--\ref{assum:A7} we have
  \begin{align*}
  \phantom{=\,}&\P\left( {\mathop {\sup }\limits_{k \ne 0} \left| {{{\widehat Q}_0}\left( {{\hbbeta^{\left( {0,k} \right)}}} \right) - {{\widehat Q}_0}\left( {{\bbeta ^{\left( {0,k} \right)}}} \right)} \right|  \lesssim {\xi _{21}}} \right) \\
  \ge \, & 1 - c_1\left(\underline{n} \wedge n_0\right)^{ - c_2 \underline{s}} -c_3 \exp \left( { - c_4\left(\underline{n} \wedge n_0\right)} \right) - \exp \left(-c_5 \log p\right)
  \end{align*}
  and 
  \begin{align*}
  \P\left( {\left| {{{\widehat Q}_0}\left( \hbbeta^{\left( 0 \right)}_{\cal I} \right) - {{\widehat Q}_0}\left( {{\bbeta ^{\left( 0 \right)}}} \right)} \right|  \lesssim {\xi _{22}} } \right) \ge 1 - c_1n_0^{ - c_2 s_0 } -c_3 \exp \left( { - c_4n_0} \right) 
  \end{align*}
  with positive constants $c_1$ to $c_5$.
\end{lemma}

\begin{proof}[Proof of Lemma~\ref{lem-5}]
  Notice that for each source $k \ne 0$, we have
  \begin{align}
  \phantom{=\,}&\left| {{{\widehat Q}_0}\left( {\hbbeta ^{\left( 0,k \right)}} \right) - {{\widehat Q}_0}\left( {{\bbeta ^{\left( 0,k \right)}}} \right)} \right| \nonumber\\
  \le \, & \frac{2}{{{n_0}}}\left|{\sum\limits_{i \in {{\cal I}^c}} {\left( {\widetilde y_i^{(0)} - \bx_i^{(0)^\top}{\bbeta ^{\left( 0 \right)}}} \right)\bx_i^{(0)^\top}\left( {{\hbbeta^{\left( 0,k \right)}} - {\bbeta ^{\left( 0,k \right)}}} \right)} }\right| + \frac{2}{{{n_0}}}\left| {\sum\limits_{i \in {{\cal I}^c}} {\bx_i^{(0)^\top}{\bbeta ^{\left( 0 \right)}}\bx_i^{(0)^\top}\left( {{\hbbeta ^{\left( 0,k \right)}} - {\bbeta ^{\left( 0,k \right)}}} \right)} } \right| \nonumber\\
  \phantom{=\,}&{ + \frac{1}{{{n_0}}}}
  \left| {\sum\limits_{i \in {{\cal I}^c}} {\bx_i^{(0)^\top}\left( {{\hbbeta ^{\left( 0,k \right)}} + {\bbeta ^{\left( 0,k \right)}}} \right)\left( {{\hbbeta^{\left( 0,k \right)}} - {\bbeta ^{\left( 0,k \right)}}} \right)\bx_i^{(0)}} } \right| \label{eq:qhat},
  \end{align}
  which is similar to the proof of Lemma 7 in \citet{tian2022transfer} for the case of linear model. For the first term on the right hand side of~\eqref{eq:qhat},
  \begin{align}
  \phantom{=\,}&\frac{2}{{{n_0}}}\left|  {\sum\limits_{i \in {{\cal I}^c}} {\left( {\widetilde y_i^{(0)} - \bx_i^{(0)^\top}{\bbeta ^{\left( 0 \right)}}} \right)\bx_i^{(0)^\top}\left( {{\hbbeta ^{\left( 0,k \right)}} - {\bbeta ^{\left( 0,k \right)}}} \right)} } \right| \nonumber \\
  = \, & \left|  {{{\left( {\frac{2}{{{n_0}}}\sum\limits_{i \in {{\cal I}^c}} {\bx_i^{(0)}\widetilde y_i^{(0)}}  - \widehat {\bSigma} _{{{\cal I}^c}}^{\left( 0 \right)}{\bbeta ^{\left( 0 \right)}}} \right)}}^\top\left( {{\hbbeta ^{\left( 0,k \right)}} - {\bbeta ^{\left( 0,k \right)}}} \right)} \right| \nonumber \\
  \le \, & \left\|\left( {\frac{2}{{{n_0}}}\sum\limits_{i \in {{\cal I}^c}} {\bx_i^{(0)}\widetilde y_i^{(0)}}  - \widehat {\bSigma} _{{{\cal I}^c}}^{\left( 0 \right)}{\bbeta ^{\left( 0 \right)}}} \right)\right\|_\infty\left\|{\hbbeta^{\left( 0,k \right)}} - {\bbeta ^{\left( 0,k \right)}}\right\|_1 \nonumber \\
  \lesssim \, & {\sqrt {\frac{\log {\left(p \vee n_0\right)}}{n_0}} } \left\|{\hbbeta^{\left( 0,k \right)}} - {\bbeta ^{\left( 0,k \right)}}\right\|_1 \label{eq:Q0hat-bound1},
  \end{align}
  with probability at least $1 - c_1n_0^{ - c_2 s_0} - c_3\exp(-c_4n_0)$, where the last inequality is due to Lemma~\ref{lem-1} and Assumption~\ref{assum:A1}. For the second and third terms on
  the right hand side of~\eqref{eq:qhat}, considering that $\bx_i^{( 0 )}$s are independent and normally distributed, we
  can assert that the centralized versions of $$\frac{2}{{{n_0}}}\sum\limits_{i \in {{\cal I}^c}}
  {\bx_i^{(0)^\top}{\bbeta ^{\left( 0 \right)}}\bx_i^{(0)^\top}\left( {{
  \hbbeta ^{\left( 0,k \right)}} - {\bbeta ^{\left( 0,k \right)}}} \right)} $$ and
  $$\frac{2}{{{n_0}}}\sum\limits_{i \in {{\cal I}^c}} {\bx_i^{(0)^\top}\left( {{\hbbeta ^{\left( 0,k \right)}} + {\bbeta ^{\left( 0,k \right)}}} \right)\left({{\hbbeta^{\left( 0,k \right)}} - {\bbeta ^{\left( 0,k \right)}}} \right)\bx_i^{(0)}} $$ 
  are both sub-exponential. Besides, both of their means are at most multiplicative to ${\|\bbeta ^{( 0,k )}\|}\| {\hbbeta^{( 0,k)} - {\bbeta ^{( 0,k )}}} \|$ 
  according to \citet{vershynin2018high}. Take the second term of~\eqref{eq:qhat} as an example: with Lemma~\ref{lem-3} and the notation ${\mu _k} = c_1\| {{\bbeta ^{( {0,k} )}}} \|\| {\hbbeta ^{( {0,k} )} - {\bbeta ^{( {0,k} )}}} \|$ with a positive constant $c_1$ and ${v_k} = \| {{\bbeta ^{( {0,k} )}}} \|\| {\hbbeta ^{( {0,k} )} - {\bbeta ^{( {0,k} )}}} \|$, we have
  \begin{align*}
\phantom{=\,}& \P\left( {\mathop {\sup }\limits_{k \ne 0} \frac{2}{{{n_0}}}\left| {\sum\limits_{i \in {\cal I}^c} {\bx_i^{\left( 0 \right)^\top}{\bbeta ^{\left( 0 \right)}}\bx_i^{\left( 0 \right)^\top}\left( {\hbbeta ^{\left( {0,k} \right)} - {\bbeta ^{\left( {0,k} \right)}}} \right)}  - {\mu _k}} \right| \ge t} \right)\\
 \le \, & \sum\limits_{k = 1}^K {\P\left( {\frac{2}{{{n_0}}}\left| {\sum\limits_{i \in {\cal I}^c} {\bx_i^{\left( 0 \right)^\top}{\bbeta ^{\left( 0 \right)}}\bx_i^{\left( 0 \right)^\top}\left( {\hbbeta ^{\left( {0,k} \right)} - {\bbeta ^{\left( {0,k} \right)}}} \right) - {\mu _k}} } \right| \ge t} \right)} \\
 \le \, & K\exp \left( { - \frac{n_0}{4}\min \left\{ {\frac{t^2}{v_k^2},\frac{t}{v_k}} \right\}} \right).
\end{align*}
Thus, by taking $t \gtrsim \mathop {\sup }\limits_{k \ne 0} \sqrt {\frac{{\log K}}{{{n_0}}}} \| {{\bbeta ^{( {0,k} )}}} \|\| {\hbbeta ^{( {0,k} )} - {\bbeta ^{( {0,k} )}}} \|$, we can assert
\begin{equation}
\label{eq:Q0hat-bound2}
\mathop {\sup }\limits_{k \ne 0} {\max \left( T_2,T_3\right)} \lesssim {\left( {1 + \sqrt {\frac{{\log K}}{{{n_0}}}} } \right)\mathop {\sup }\limits_{k \ne 0} \left\| {{\bbeta ^{\left( {0,k} \right)}}} \right\|\left\| {{\hbbeta ^{\left( {0,k} \right)}} - {\bbeta ^{\left( {0,k} \right)}}} \right\|}
\end{equation}
  with probability at least $1- 2\exp(-c_1n_0)$. Hence, combining~\eqref{eq:Q0hat-bound1} and~\eqref{eq:Q0hat-bound2}, we have 
  \begin{align*}
    \phantom{ =\, } &
    \sup_{k \ne 0} \left|
    \widehat{Q}_0 \left(\hbbeta^{(0,k)}\right) -
    \widehat{Q}_0 \left(\bbeta^{(0,k)} \right)
    \right| \\
     \lesssim \, &
 \sqrt {\frac{{\log \left( {p \vee {n_0}} \right)}}{{{n_0}}}} \mathop {\sup }\limits_{k \ne 0} {\left\| {{\hbbeta ^{\left( {0,k} \right)}} - {\bbeta ^{\left( {0,k} \right)}}} \right\|_1} + \left( {1 + \sqrt {\frac{{\log K}}{{{n_0}}}} } \right)\mathop {\sup }\limits_{k \ne 0} \left\| {{\bbeta ^{\left( {0,k} \right)}}} \right\|\mathop {\sup }\limits_{k \ne 0} \left\| {{\hbbeta ^{\left( {0,k} \right)}} - {\bbeta ^{\left( {0,k} \right)}}} \right\|
  \end{align*}
  with probability at least $1 - c_1n_0^{ - c_2 s_0} - c_3\exp ( { - c_4{n_0}} )$. From statement (i) in the proof of Theorem~\ref{thm:oracle}, with ${n'} = \mathop {\min }\limits_{k \in {\calA}} {n_k}$, we have 
  \begin{align*}
  \phantom{ =\, } &\P\left( {\mathop {\sup }\limits_{k \in \calA} \left\| {{\hbbeta ^{\left( {0,k} \right)}} - {\bbeta ^{\left( {0,k} \right)}}} \right\| \lesssim {\frac{{{s_0}\log \left( {p \vee {n_0} \vee {n'}} \right)}}{{{n' } + {n_0}}} + {C_{\bSigma} ^{\calA}d }\sqrt {\frac{{\log \left( {p \vee {n_0} \vee n'} \right)}}{{n' + {n_0}}}} }} \right) \\
  \ge \, & 1 - c_1\left(n' \wedge n_0\right)^{-c_2 \underline{s}_{\calA}} - c_3\exp\left( -c_4 \left(n' \wedge n_0\right)\right) - \exp\left(-c_5\log p\right),
  \end{align*}
  and 
  \begin{align*}
  \phantom{ =\, } &\P\left( {\mathop {\sup }\limits_{k \in \calA} \left\| {{\hbbeta ^{\left( {0,k} \right)}} - {\bbeta ^{\left( {0,k} \right)}}} \right\|_1 \lesssim {s_0}\sqrt {\frac{{\log \left( {p \vee {n_0} \vee n'} \right)}}{{n' + {n_0}}}}  + C_{\bSigma} ^\calA d} \right) \\
  \ge \, & 1 - c_1\left(n' \wedge n_0\right)^{-c_2 \underline{s}_{\calA}} - c_3\exp\left( -c_4 \left(n' \wedge n_0\right)\right) - \exp\left(-c_5\log p\right).
  \end{align*}
  For all $k \in {\calA}^c$, combining Assumption~\ref{assum:A7} with~\eqref{eq:inter-3}, similarly with $n''=\mathop {\min }\limits_{k \in {\calA}^c} {n_k}$, we have
   \begin{align*}
  \phantom{ =\, } &\P\left( {\mathop {\sup }\limits_{k \in {\calA}^c} \left\| {{\hbbeta ^{\left( {0,k} \right)}} - {\bbeta ^{\left( {0,k} \right)}}} \right\| \lesssim {\frac{{{s'}\log \left( {p \vee {n_0} \vee n''} \right)}}{n'' + {n_0}} + d'\sqrt {\frac{{\log \left( {p \vee {n_0} \vee n''} \right)}}{{n'' + {n_0}}}} }} \right) \\
  \ge \, & 1 - c_1\left(n'' \wedge n_0\right)^{-c_2 \underline{s}_{{\calA}^c}} - c_3\exp\left( -c_4 \left(n'' \wedge n_0\right)\right) - \exp\left(-c_5\log p\right),
  \end{align*}
  and 
  \begin{align*}
  \phantom{ =\, } &\P\left( {\mathop {\sup }\limits_{k \in {\calA}^c} \left\| {{\hbbeta ^{\left( {0,k} \right)}} - {\bbeta ^{\left( {0,k} \right)}}} \right\|_1 \lesssim s'\sqrt {\frac{{\log \left( {p \vee {n_0} \vee n''} \right)}}{{n'' + {n_0}}}}  + d'} \right) \\
  \ge \, & 1 - c_1\left(n'' \wedge n_0\right)^{-c_2 \underline{s}_{{\calA}^c}} - c_3\exp\left( -c_4 \left(n'' \wedge n_0\right)\right) - \exp\left(-c_5\log p\right).
  \end{align*}
  Therefore, the desired conclusion holds true by combining two ineqaulity above and noticing that $\mathop {\sup }\nolimits_{k \ne 0} \| {{\bbeta ^{( {0,k} )}}} \| \lesssim \sqrt{s_0 \vee s'}$ from Assuption~\ref{assum:A7}. Also, the conclusion holds true similarly for the target estimator.
\end{proof}

\begin{lemma}
  \label{lem-6}
  With the notation $\xi_{22}$ defined in Lemma~\ref{lem-5}, under Assumptions~\ref{assum:A1},~\ref{assum:A2},~\ref{assum:A4},~\ref{assum:A5},~\ref{assum:A6}, and~\ref{assum:A7}, we have
  \begin{align*}
  \phantom{=\,}&\P\left( {\mathop {\sup }\limits_{k \ne 0} \left| {{{\widehat Q}_0}\left( {{\bbeta^{\left( {0,k} \right)}}} \right) - {Q_0}\left( {{\bbeta ^{\left( {0,k} \right)}}} \right)} \right| \vee \left| {{{\widehat Q}_0}\left( {{{\bbeta }^{\left( 0 \right)}}} \right) - {Q_0}\left( {{\bbeta ^{\left( 0 \right)}}} \right)} \right| \lesssim {\xi _{22}}} \right) \\
  \ge \, & 1 - c_1n_0^{ - c_2 s_0} -c_3 \exp \left( { - c_4n_0} \right)
  \end{align*}
  with $c_1$ to $c_4$ are some positive constant.
\end{lemma}

\begin{proof}[Proof of Lemma~\ref{lem-6}]
Similar to the proof of Lemma 8 in \citet{tian2022transfer}, we have that
  \begin{align}
  \phantom{=\,}&\left| {{{\widehat Q}_0}\left( {{\bbeta ^{\left( 0,k \right)}}} \right) - {Q_0}\left( {{\bbeta ^{\left( 0,k \right)}}} \right)}\right| \nonumber\\
  \le \, & \frac{2}{{{n_0}}}\left| {\sum\limits_{i \in {{\cal I}^c}} {\left( {\widetilde y_i^{(0)} - \bx_i^{(0)^\top}{\bbeta ^{\left( 0 \right)}}} \right)\bx_i^{(0)^\top}{\bbeta ^{\left( 0 \right)}}} } \right| + \frac{2}{{{n_0}}}\left| {\sum\limits_{i \in {{\cal I}^c}} {{{\left( {\bx_i^{(0)^\top}{\bbeta ^{\left( 0,k \right)}}} \right)}^2} - E{{\left( {\bx_i^{(0)^\top}{\bbeta ^{\left( 0,k \right)}}} \right)}^2}} } \right| \nonumber\\
  \phantom{=\,}&{ + \frac{2}{{{n_0}}}\left| {\sum\limits_{i \in {{\cal I}^c}} {\bx_i^{(0)^\top}{\bbeta ^{\left( 0 \right)}}\bx_i^{(0)^\top}{\bbeta ^{\left( 0,k \right)}} - E\left( {\bx_i^{(0)^\top}{\bbeta ^{\left( 0 \right)}}\bx_i^{(0)^\top}{\bbeta ^{\left( 0,k \right)}}} \right)} } \right|} \label{eq:qorigin}.
  \end{align}
For the first term on the right hand side of~\eqref{eq:qorigin}, we bound it by:
  \begin{align}
  \phantom{=\,}&\frac{2}{{{n_0}}}\left| {\sum\limits_{i \in {{\cal I}^c}} {\left( {\widetilde y_i^{(0)} - \bx_i^{(0)^\top}{\bbeta ^{\left( 0 \right)}}} \right)\bx_i^{(0)^\top}{\bbeta ^{\left( 0 \right)}}} } \right|\nonumber \\
  \le \, &\left\|\frac{2}{{{n_0}}}\sum\limits_{i \in {{\cal I}^c}} {\bx_i^{(0)^\top}\widetilde y_i^{(0)}}  - \widehat {\bSigma} _{{\cal I}^c}^{\left( 0 \right)}{\bbeta ^{\left( 0 \right)}}\right\|_\infty \left\|{\bbeta ^{\left( 0 \right)}}\right\|_1 \nonumber \\
  \lesssim \, &\sqrt {\frac{{\log \left(p \vee n_0\right)}}{n_0}} \left\|{\bbeta ^{\left( 0 \right)}}\right\|_1 \label{eq:Q0hatQ0-bound1},
  \end{align}
  with probability at least $1 - c_1n_0^{ - c_2 s_0} - c_3\exp(-c_4n_0)$, where the last inequality is due to Lemma~\ref{lem-1} and Assumption~\ref{assum:A1}. For the second and the third terms on the right hand side of~\ref{eq:qorigin}, we know that $$\frac{2}{{{n_0}}}\sum\limits_{i \in {{\cal I}^c}}{{\left(
  {\bx_i^{(0)^\top}{\bbeta ^{\left( 0,k \right)}}} \right)}^2}$$ and $$
  \frac{2}{{{n_0}}}\sum\limits_{i \in {{\cal I}^c}} {\bx_i^{(0)^\top}{\bbeta ^{\left( 0
  \right)}}\bx_i^{(0)^\top}{\bbeta ^{\left( 0,k \right)}}} $$ are sub-exponential with their means at most multiplicative to
  $\|\bbeta ^{( 0,k )}\|^2$ and
  $\|\bbeta ^{( 0 )}\|\|{\bbeta ^{( 0,k )}}\|$ respectively, according to \citet{vershynin2018high}. We utilize Lemma~\ref{lem-3} to get
  \begin{equation}
  \label{eq:Q0hatQ0-bound2}
  \mathop {\sup }\limits_{k \ne 0} \left|  {\frac{1}{{{n_0}}}\sum\limits_{i \in {{\cal I}^c}} {{{\left( {\bx_i^{(0)^\top}{\bbeta ^{\left( 0,k \right)}}} \right)}^2} - E{{\left( {\bx_i^{(0)^\top}{\bbeta ^{\left( 0,k \right)}}} \right)}^2}} } \right|  \lesssim {\sqrt {\frac{{\log K}}{{{n_0}}}} \mathop {\sup }\limits_{k \ne 0} \left\|{\bbeta ^{\left( 0,k \right)}}\right\|^2},
  \end{equation}
  and 
  \begin{align}
  \phantom{=\,}&\mathop {\sup }\limits_{k \ne 0} \left| {\frac{2}{{{n_0}}}\sum\limits_{i \in {{\cal I}^c}} {\bx_i^{(0)^\top}{\bbeta ^{\left( 0 \right)}}\bx_i^{(0)^\top}{\bbeta ^{\left( 0,k \right)}} - E\left( {\bx_i^{(0)^\top}{\bbeta ^{\left( 0 \right)}}\bx_i^{(0)^\top}{\bbeta ^{\left( 0,k \right)}}} \right)} } \right| \nonumber \\
  \lesssim \, & {\sqrt {\frac{{\log K}}{{{n_0}}}} \left\|{\bbeta ^{\left( 0 \right)}}\right\|\mathop {\sup }\limits_{k \ne 0} \left\|{\bbeta ^{\left( 0,k \right)}}\right\|}, \label{eq:Q0hatQ0-bound3}
  \end{align}
  with probability at least $1- 2\exp(-c_1n_0)$. Combining~\eqref{eq:Q0hatQ0-bound1},~\eqref{eq:Q0hatQ0-bound2} and~\eqref{eq:Q0hatQ0-bound3} and noticing that $\| \bbeta^{(0)} \|_1$ is in level of $s_0$, $\| \bbeta^{(0)} \|$ and $\| \bbeta^{(0,k)} \|$ are in level of $\sqrt{s_0 \vee s'}$ under Assumption~\ref{assum:A7}, we reach the desired conclusion. The conclusion holds true similarly for the target estimator.
\end{proof}

We now prove Theorem~\ref{thm:detect}. On one hand, from the decomposition 
  \begin{align}
  \phantom{=\,}&\mathop {\sup }\limits_{k \in {\cal A}} \left| {{{\widehat Q}_0}\left( {{\hbbeta ^{\left( 0,k \right)}}} \right) - {{\widehat Q}_0}\left( {\hbbeta^{\left( 0 \right)}_{\cal I}} \right)} \right| \nonumber\\
  \le \,& \mathop {\sup }\limits_{k \in {\cal A}} \left| {{{\widehat Q}_0}\left( {{\hbbeta ^{\left( 0,k \right)}}} \right) - {{\widehat Q}_0}\left( {{\bbeta ^{\left( 0,k \right)}}} \right)} \right| +  \left| {{{\widehat Q}_0}\left( {{\hbbeta ^{\left( 0 \right)}_{\cal I}}} \right) - {{\widehat Q}_0}\left( {{\bbeta ^{\left( 0 \right)}}} \right)} \right| \nonumber\\
  \phantom{=\,}& + \mathop {\sup }\limits_{k \in {\cal A}} \left| {{{\widehat Q}_0}\left( {{\bbeta ^{\left( 0,k \right)}}} \right) - {{\widehat Q}_0}\left( {{\bbeta ^{\left( 0 \right)}}} \right) - {Q_0}\left( {{\bbeta ^{\left( 0,k \right)}}} \right) + {Q_0}\left( {{\bbeta ^{\left( 0 \right)}}} \right)} \right| \nonumber\\
  \phantom{=\,}& +\mathop {\sup }\limits_{k \in {\cal A}} \left| {{Q_0}\left( {{\bbeta ^{\left( 0,k \right)}}} \right) - {Q_0}\left( {{\bbeta ^{\left( 0 \right)}}} \right)} \right|\label{eq:qbetahat},
  \end{align}
  we can bound the first and second term on the right hand side of~\eqref{eq:qbetahat} according to Lemma~\ref{lem-5} as
  \begin{align*}
  \phantom{=\,}& \mathop {\sup }\limits_{k \in {\cal A}} \left| {{{\widehat Q}_0}\left( {{\hbbeta ^{\left( 0,k \right)}}} \right) - {{\widehat Q}_0}\left( {{\bbeta ^{\left( 0,k \right)}}} \right)} \right| +  \left| {{{\widehat Q}_0}\left( {{\hbbeta ^{\left( 0 \right)}_{\cal I}}} \right) - {{\widehat Q}_0}\left( {{\bbeta ^{\left( 0 \right)}}} \right)} \right| \\
  \le \, &\mathop {\sup }\limits_{k \ne 0} \left| {{{\widehat Q}_0}\left( {{\hbbeta ^{\left( 0,k \right)}}} \right) - {{\widehat Q}_0}\left( {{\bbeta ^{\left( 0,k \right)}}} \right)} \right| +  \left| {{{\widehat Q}_0}\left( {{\hbbeta ^{\left( 0 \right)}_{\cal I}}} \right) - {{\widehat Q}_0}\left( {{\bbeta ^{\left( 0 \right)}}} \right)} \right|\\
  \lesssim \, & \xi_{21} + \xi_{22},
  \end{align*}
  with probability at least $1 - c_1(\underline{n} \wedge n_0)^{ - c_2 \underline{s}} -c_3 \exp ( { - c_4(\underline{n} \wedge n_0)} ) - \exp (-c_5 \log p)$. Meanwhile, the fourth term on the right hand side of~\eqref{eq:qbetahat} can be bounded by Lemma~\ref{lem-4} as
  \begin{equation*}
  \mathop {\sup }\limits_{k \in {\cal A}}\left| {{Q_0}\left( {{\bbeta ^{\left( 0,k \right)}}} \right) - {Q_0}\left( {{\bbeta ^{\left( 0 \right)}}} \right)} \right| = O\left( {{d^2}} \right).
  \end{equation*}
  Next, the third term on the right hand side of~\eqref{eq:qbetahat} can be bounded by Lemma~\ref{lem-6} as 
  \begin{align*}
  \phantom{=\,}&\mathop {\sup }\limits_{k \in {\cal A}} \left|  {{{\widehat Q}_0}\left( {{\bbeta ^{\left( 0,k \right)}}} \right) - {{\widehat Q}_0}\left( {{\bbeta ^{\left( 0 \right)}}} \right) - {Q_0}\left( {{\bbeta ^{\left( 0,k \right)}}} \right) + {Q_0}\left( {{\bbeta ^{\left( 0 \right)}}} \right)} \right| \\
  \le \, & \mathop {\sup }\limits_{k \in {\cal A}} \left|  {{{\widehat Q}_0}\left( {{\bbeta ^{\left( 0,k \right)}}} \right) - {Q_0}\left( {{\bbeta ^{\left( 0,k \right)}}} \right)} \right|  + \mathop {\sup }\limits_{k \in {\cal A}} \left|  {{{\widehat Q}_0}\left( {{\bbeta ^{\left( 0 \right)}}} \right) - {Q_0}\left( {{\bbeta ^{\left( 0 \right)}}} \right)} \right| \\
  \le \, & \mathop {\sup }\limits_{k \ne 0} \left|  {{{\widehat Q}_0}\left( {{\bbeta ^{\left( 0,k \right)}}} \right) - {Q_0}\left( {{\bbeta ^{\left( 0,k \right)}}} \right)} \right|  + \mathop {\sup }\limits_{k \ne 0} \left|  {{{\widehat Q}_0}\left( {{\bbeta ^{\left( 0 \right)}}} \right) - {Q_0}\left( {{\bbeta ^{\left( 0 \right)}}} \right)}\right| \\
  \lesssim \, & \xi_{22},
  \end{align*}
  with probability at least $1- c_1n_0^{-c_2 s_0} - c_3\exp( -c_4n_0 )$. Hence, considering that $\xi_2 = \xi_{21} \vee \xi_{22}$, we have
  \begin{equation*}
  \mathop {\sup }\limits_{k \in {\cal A}} \left| {{{\widehat Q}_0}\left( {{\hbbeta ^{\left( 0,k \right)}}} \right) - {{\widehat Q}_0}\left( {{\hbbeta ^{\left( 0 \right)}_{\cal I}}} \right)} \right| 
  \lesssim {{d^2} + \xi_2},
  \end{equation*}
  with probability at least $1 - c_1(\underline{n} \wedge n_0)^{ - c_2 \underline{s}}-c_3\exp(-c_4(\underline{n}\wedge n_0)) - \exp (-c_5 \log p)$ for some positive constants $c_1$ to $c_5$. Hence, with the constant ${c_\varepsilon }$ mentioned in Assumption~\ref{assum:A8} and the fact that $| {{Q_0}( {{\bbeta ^{( 0 )}}} ) - {{\widehat Q}_0}( {{\hbbeta ^{( 0 )}_{\cal I}}} )} | \lesssim {\xi_{22}}$ from the combination of Lemma~\ref{lem-5} and~\ref{lem-6}, we have
  \begin{equation*}
  \mathop {\sup }\limits_{k \in {\cal A}} \left| {{{\widehat Q}_0}\left( {{\hbbeta ^{\left( 0,k \right)}}} \right) - {{\widehat Q}_0}\left( {{\hbbeta ^{\left( 0 \right)}_{\cal I}}} \right)} \right| \le {c_\varepsilon }{{\widehat Q}_0}\left( {{\hbbeta ^{\left( 0 \right)}_{\cal I}}} \right)
  \end{equation*}
  with the same probability mentioned above. We can further have 
  \begin{equation*}
  \mathop {\sup }\limits_{k \in {\cal A}}{{\widehat Q}_0}\left( {{\hbbeta ^{\left( 0,k \right)}}} \right)  \le \left(1+{c_\varepsilon }\right){{\widehat Q}_0}\left( {{\hbbeta ^{\left( 0 \right)}_{\cal I}}} \right).
  \end{equation*}

On the other hand, we have 
\begin{align*}
  \phantom{=\,}&\mathop {\inf }\limits_{k \in {{\cal A}^c}}\left| {{{\widehat Q}_0}\left( {{\hbbeta^{\left( 0,k \right)}}} \right) - {{\widehat Q}_0}\left( {{\hbbeta ^{\left( 0 \right)}_{\cal I}}} \right)} \right|\\
  \geq \,& \mathop {\inf }\limits_{k \in {{\cal A}^c}} \left| {{Q_0}\left( {{\bbeta ^{\left( 0,k \right)}}} \right) - {Q_0}\left( {{\bbeta ^{\left( 0 \right)}}} \right)} \right|\\
  \phantom{=\,} \,& -\mathop {\sup }\limits_{k \in {{\cal A}^c}} \left| {{{\widehat Q}_0}\left( {{\hbbeta ^{\left( 0,k \right)}}} \right) - {{\widehat Q}_0}\left( {{\bbeta ^{\left( 0,k \right)}}} \right)} \right| -  \left| {{{\widehat Q}_0}\left( {{\hbbeta ^{\left( 0 \right)}_{\cal I}}} \right) - {{\widehat Q}_0}\left( {{\bbeta ^{\left( 0 \right)}}} \right)} \right| \\
  \phantom{=\,}& - \mathop {\sup }\limits_{k \in {{\cal A}^c}} \left| {{{\widehat Q}_0}\left( {{\bbeta ^{\left( 0,k \right)}}} \right) - {{\widehat Q}_0}\left( {{\bbeta ^{\left( 0 \right)}}} \right) - {Q_0}\left( {{\bbeta ^{\left( 0,k \right)}}} \right) + {Q_0}\left( {{\bbeta ^{\left( 0 \right)}}} \right)} \right| \\
  \geq \, & \mathop {\inf }\limits_{k \in {{\cal A}^c}} \left| {{Q_0}\left( {{\bbeta ^{\left( 0,k \right)}}} \right) - {Q_0}\left( {{\bbeta ^{\left( 0 \right)}}} \right)} \right| - \xi_2,
  \end{align*}
  with the same probability mentioned above. Then, utilizing the Taylor expansion of ${Q_0}( {{\bbeta ^{( 0,k )}}} )$ on ${Q_0}( {{\bbeta ^{( 0 )}}} )$ and noticing ${Q'_0}( {{\bbeta ^{( 0 )}}} )=0$, under Assumption~\ref{assum:A8}, we can get
  \begin{align*}
\phantom{=\,}&\mathop {\inf }\limits_{k \in {{\cal A}^c}} \left| {{\widehat Q}_0}\left( {{\hbbeta ^{\left( 0,k \right)}}} \right) - {{\widehat Q}_0}\left( {{\hbbeta ^{\left( 0 \right)}_{\cal I}}} \right) \right|\\
 \geq \, & {\lambda _{\min }}\mathop {\inf }\limits_{k \in {{\cal A}^c}}\left\|{\bbeta ^{\left( 0,k \right)}} - {\bbeta ^{\left( 0 \right)}}\right\|^2 - \xi_2 \\
 \geq  \, &{ c_\varepsilon {Q_0}\left( {{\bbeta ^{\left( 0 \right)}}} \right)  + \xi_2} \\
 \geq \, & c_\varepsilon {{\widehat Q}_0}\left( {\hbbeta ^{\left( 0 \right)}_{\cal I}} \right),
  \end{align*}
with probability at least $1 - c_1(\underline{n} \wedge n_0)^{ - c_2 \underline{s}}-c_3\exp(-c_4(\underline{n} \wedge n_0)) - \exp(-c_5 \log p )$ for some positive constants $c_1$ to $c_5$, where the last line holds true due to $| {{Q_0}( {{\bbeta ^{( 0 )}}} ) - {{\widehat Q}_0}( {{\hbbeta ^{( 0 )}_{\cal I}}} )} | \lesssim {\xi_{22}}$ from the combination of Lemma~\ref{lem-5} and~\ref{lem-6}. Hence, we have
\begin{equation*}
\mathop {\inf }\limits_{k \in {{\cal A}^c}} {{\widehat Q}_0}\left( {{\hbbeta ^{\left( 0,k \right)}}} \right) \geq \left( {1 + {c_\varepsilon }} \right){{\widehat Q}_0}\left( {\hbbeta ^{\left( 0 \right)}_{\cal I}} \right)
\end{equation*}
with the mentioned probability above. Finally, we can reach to the desired conclusion, because
\begin{align*}
\phantom{=\,}&\P\left( {\widehat {\cal A}  \ne  {\cal A}} \right) \\
\le \,& \P\left( {\left\{ {\mathop {\inf }\limits_{k \in {{\cal A}^c}} {{\widehat Q}_0}\left( {{\hbbeta ^{\left( 0,k \right)}}} \right) < \left( {1 + {c_\varepsilon }} \right){{\widehat Q}_0}\left( {\hbbeta ^{\left( 0 \right)}_{\cal I}} \right)} \right\} \cup \left\{ {\mathop {\sup }\limits_{k \in {\cal A}} {{\widehat Q}_0}\left( {{\hbbeta ^{\left( 0,k \right)}}} \right) > \left( {1 + {c_\varepsilon }} \right){{\widehat Q}_0}\left( {\hbbeta ^{\left( 0 \right)}_{\cal I}} \right)} \right\}} \right)\\
\le \,& \P\left( { {\mathop {\inf }\limits_{k \in {{\cal A}^c}} {{\widehat Q}_0}\left( {{\hbbeta ^{\left( 0,k \right)}}} \right) < \left( {1 + {c_\varepsilon }} \right){{\widehat Q}_0}\left( {\hbbeta ^{\left( 0 \right)}_{\cal I}} \right)} }\right) + \P\left({\mathop {\sup }\limits_{k \in {\cal A}} {{\widehat Q}_0}\left( {{\hbbeta ^{\left( 0,k \right)}}} \right) > \left( {1 + {c_\varepsilon }} \right){{\widehat Q}_0}\left( {\hbbeta ^{\left( 0 \right)}_{\cal I}} \right)}\right)\\
\le \, & c_1\left(\underline{n} \wedge n_0\right)^{ - c_2 \underline{s} }+c_3\exp\left(-c_4\left(\underline{n} \wedge n_0\right)\right) + 2\exp\left(-c_5 \log p \right).
\end{align*}

\section{Additional Results for Simulations}\label{sec:addition-result}

We provide additional simulation results in this section. To recall the methods defined in Section~\ref{sec:simulation} and with $|\mathcal{A}|$ representing the number of informative sets, the following abbreviations are used in the subsequent tables:

\begin{itemize}
  \item \textbf{L} stands for "Lasso QR", which denotes $\ell_1$-penalized quantile regression with only
    target sample;
  \item \textbf{OTL} stands for "Oracle Trans-Lasso QR", which is our proposed method assuming a known informative set;
  \item \textbf{NTL} stands for "Naive Trans-Lasso QR", which is our proposed method that naively treats all sources as informative without a detection procedure;
  \item \textbf{PTL} stands for "Pseudo Trans-Lasso QR", which is our proposed method performing informative set detection with a known number of informative sources $|\mathcal{A}|$;
  \item \textbf{TL} stands for "Trans-Lasso QR", which is our proposed method performing informative set detection without prior knowledge;
  \item \textbf{OTLP} stands for "Oracle Trans-Lasso Pooled QR", assuming a known informative set, as presented in Section 2 of \citet{huang2022transfer}.
  \item \textbf{TLP} stands for "Trans-Lasso Pooled QR", which performs informative set detection, presented in Sections 4 and 5 of \citet{huang2022transfer}.
\end{itemize}

Our simulations, detailed in Section~\ref{sec:simulation}, vary across the following parameters: dimension size $p$, sample size $n$ such that $n = n_0 = \cdots = n_{20}$, signal-to-noise ratio $\eta$, residual distribution (either ``normal'' or ``Cauchy''), and the covariance type (either ``auto'' for auto-covariance in the homogeneous setting or ``Toeplitz'' for Toeplitz covariance in the heterogeneous setting). The parameter settings for all simulation results presented in this section are listed as follows:

\begin{itemize}
 \item For Tables~\ref{tab:ci-normal-mse}-\ref{tab:ci-normal-ql}, we set $p = 150$, $n = 150$, $\eta = 20$, residual as ``normal'', covariance as ``auto''. 
 \item For Tables~\ref{tab:ci-cauchy-mse}-\ref{tab:ci-cauchy-ql}, we set $p = 150$, $n = 150$, $\eta = 20$, residual as ``Cauchy'', covariance as ``auto''. 
 \item For Tables~\ref{tab:cii-normal-mse}-\ref{tab:cii-normal-ql}, we set $p = 150$, $n = 150$, $\eta = 20$, residual as ``normal'', covariance as ``Toeplitz''. 
 \item For Tables~\ref{tab:cii-cauchy-mse}-\ref{tab:cii-cauchy-ql}, we set $p = 150$, $n = 150$, $\eta = 20$, residual as ``Cauchy'', covariance as ``Toeplitz''. 
 \item For Tables~\ref{tab:c1h1eta10-mse}-\ref{tab:c1h1eta10-ql}, we set $p = 150$, $n = 150$, $\eta = 10$, residual as ``normal'', covariance as ``auto''. 
 \item For Tables~\ref{tab:c1h1eta5-mse}-\ref{tab:c1h1eta5-ql}, we set $p = 150$, $n = 150$, $\eta = 5$, residual as ``normal'', covariance as ``auto''. 
 \item For Tables~\ref{tab:c1h1p75-mse}-\ref{tab:c1h1p75-ql}, we set $p = 75$, $n = 150$, $\eta = 10$, residual as ``normal'', covariance as ``auto''. 
 \item For Tables~\ref{tab:c1h1p225-mse}-\ref{tab:c1h1p225-ql}, we set $p = 225$, $n = 150$, $\eta = 10$, residual as ``normal'', covariance as ``auto''. 
 \item For Tables~\ref{tab:c1h2p75-mse}-\ref{tab:c1h2p75-ql}, we set $p = 75$, $n = 150$, $\eta = 10$, residual as ``normal'', covariance as ``Toeplitz''. 
 \item For Tables~\ref{tab:c1h2p225-mse}-\ref{tab:c1h2p225-ql}, we set $p = 225$, $n = 150$, $\eta = 10$, residual as ``normal'', covariance as ``Toeplitz''.
 \item For Tables~\ref{tab:c2h1p75-mse}-\ref{tab:c2h1p75-ql}, we set $p = 75$, $n = 150$, $\eta = 10$, residual as ``Cauchy'', covariance as ``auto''.
 \item For Tables~\ref{tab:c2h1p225-mse}-\ref{tab:c2h1p225-ql}, we set $p = 225$, $n = 150$, $\eta = 10$, residual as ``Cauchy'', covariance as ``auto''.
 \item For Tables~\ref{tab:c2h2p75-mse}-\ref{tab:c2h2p75-ql}, we set $p = 75$, $n = 150$, $\eta = 10$, residual as ``Cauchy'', covariance as ``Toeplitz''.
 \item For Tables~\ref{tab:c2h2p225-mse}-\ref{tab:c2h2p225-ql}, we set $p = 225$, $n = 150$, $\eta = 10$, residual as ``Cauchy'', covariance as ``Toeplitz''.
 \item For Tables~\ref{tab:c1h1ls_mse}-\ref{tab:c1h1ls_time}, we set $p \in \{75, 150, 300, 600\}$, $n \in \{1000, 2500, 5000, 10000\}$, $\eta = 20$, residual as ``normal'', covariance as ``auto''.
\end{itemize}


\noindent\makebox[\textwidth]{%
  \fontsize{10}{8}\selectfont  
\begin{threeparttable}  
  \renewcommand\arraystretch{1}
  \caption{MSE on $\bbeta$ among different $d$ in homogeneous setting for normal error case (Case 1) with $\eta=20$. The results are presented as ``average (standard deviation)''.}  
  \label{tab:ci-normal-mse}  
  \setlength\extrarowheight{-3pt}
\begin{tabular}{ccccccccccccccc}  
    \toprule  
    \multirow{1}{*}{$\left| {{\mathcal{A}}} \right|$}&  
    \multicolumn{2}{c}{L}&\multicolumn{2}{c}{OTL}&\multicolumn{2}{c}{NTL}&\multicolumn{2}{c}{PTL}&\multicolumn{2}{c}{TL}&\multicolumn{2}{c}{OTLP}&\multicolumn{2}{c}{TLP}\cr
    \midrule  
    \multicolumn{2}{l}{\textbf{$d=2$}} & \\
    $0$ & \multicolumn{2}{c}{6.20 (1.48)} & \multicolumn{2}{c}{ 6.20 (1.48)} & \multicolumn{2}{c}{5.84 (1.27)} & \multicolumn{2}{c}{6.20 (1.48)} & \multicolumn{2}{c}{6.21 (1.48)} & \multicolumn{2}{c}{ 6.20 (1.48)} & \multicolumn{2}{c}{6.93 (2.69)} \\[0.1cm] 
  $1$ & \multicolumn{2}{c}{6.22 (1.49)} & \multicolumn{2}{c}{ 3.61 (0.93)} & \multicolumn{2}{c}{5.54 (1.25)} & \multicolumn{2}{c}{3.61 (0.93)} & \multicolumn{2}{c}{3.64 (1.00)} & \multicolumn{2}{c}{ 3.00 (0.81)} & \multicolumn{2}{c}{3.35 (1.32)} \\[0.1cm] 
  $2$ & \multicolumn{2}{c}{6.25 (1.50)} & \multicolumn{2}{c}{ 2.88 (0.66)} & \multicolumn{2}{c}{5.23 (1.20)} & \multicolumn{2}{c}{2.89 (0.65)} & \multicolumn{2}{c}{2.94 (0.86)} & \multicolumn{2}{c}{ 2.28 (0.51)} & \multicolumn{2}{c}{2.41 (0.75)} \\[0.1cm] 
  $4$ & \multicolumn{2}{c}{6.26 (1.49)} & \multicolumn{2}{c}{ 2.33 (0.61)} & \multicolumn{2}{c}{4.78 (1.10)} & \multicolumn{2}{c}{2.35 (0.60)} & \multicolumn{2}{c}{2.35 (0.63)} & \multicolumn{2}{c}{ 1.73 (0.41)} & \multicolumn{2}{c}{1.77 (0.45)} \\[0.1cm] 
  $6$ & \multicolumn{2}{c}{6.25 (1.51)} & \multicolumn{2}{c}{ 2.13 (0.55)} & \multicolumn{2}{c}{4.36 (1.01)} & \multicolumn{2}{c}{2.11 (0.54)} & \multicolumn{2}{c}{2.15 (0.57)} & \multicolumn{2}{c}{ 1.55 (0.36)} & \multicolumn{2}{c}{1.56 (0.35)} \\[0.1cm] 
  $8$ & \multicolumn{2}{c}{6.27 (1.48)} & \multicolumn{2}{c}{ 2.01 (0.53)} & \multicolumn{2}{c}{3.90 (0.87)} & \multicolumn{2}{c}{1.98 (0.56)} & \multicolumn{2}{c}{2.03 (0.55)} & \multicolumn{2}{c}{ 1.43 (0.38)} & \multicolumn{2}{c}{1.42 (0.38)} \\
\bottomrule  
    \end{tabular}  
    \end{threeparttable}  
}

\noindent\makebox[\textwidth]{%
  \fontsize{10}{8}\selectfont  
\begin{threeparttable}  
  \renewcommand\arraystretch{1}
  \caption{Quantile loss on $\bbeta$ among different $d$ in homogeneous setting for normal error case (Case 1) with $\eta=20$. The results are presented as ``average (standard deviation)''.}  
  \label{tab:ci-normal-ql}  
  \setlength\extrarowheight{-3pt}
    \begin{tabular}{ccccccccccccccc}  
    \toprule  
    \multirow{1}{*}{$\left| {{\mathcal{A}}} \right|$}&  
    \multicolumn{2}{c}{L}&\multicolumn{2}{c}{OTL}&\multicolumn{2}{c}{NTL}&\multicolumn{2}{c}{PTL}&\multicolumn{2}{c}{TL}&\multicolumn{2}{c}{OTLP}&\multicolumn{2}{c}{TLP}\cr
    \midrule  
    \multicolumn{2}{l}{\textbf{$d=2$}} & \\
    $0$ & \multicolumn{2}{c}{1.07 (0.07)} & \multicolumn{2}{c}{1.07 (0.07)} & \multicolumn{2}{c}{1.15 (0.09)} & \multicolumn{2}{c}{1.07 (0.07)} & \multicolumn{2}{c}{1.07 (0.07)} & \multicolumn{2}{c}{1.07 (0.07)} & \multicolumn{2}{c}{1.07 (0.08)} \\[0.1cm]
  $1$ & \multicolumn{2}{c}{1.07 (0.07)} & \multicolumn{2}{c}{0.97 (0.06)} & \multicolumn{2}{c}{1.14 (0.09)} & \multicolumn{2}{c}{0.97 (0.06)} & \multicolumn{2}{c}{0.97 (0.06)} & \multicolumn{2}{c}{0.92 (0.05)} & \multicolumn{2}{c}{0.92 (0.05)} \\[0.1cm] 
  $2$ & \multicolumn{2}{c}{1.07 (0.07)} & \multicolumn{2}{c}{0.95 (0.05)} & \multicolumn{2}{c}{1.13 (0.08)} & \multicolumn{2}{c}{0.95 (0.05)} & \multicolumn{2}{c}{0.96 (0.06)} & \multicolumn{2}{c}{0.89 (0.04)} & \multicolumn{2}{c}{0.89 (0.05)} \\[0.1cm] 
  $4$ & \multicolumn{2}{c}{1.07 (0.07)} & \multicolumn{2}{c}{0.93 (0.05)} & \multicolumn{2}{c}{1.10 (0.08)} & \multicolumn{2}{c}{0.93 (0.05)} & \multicolumn{2}{c}{0.94 (0.05)} & \multicolumn{2}{c}{0.87 (0.05)} & \multicolumn{2}{c}{0.87 (0.04)} \\[0.1cm] 
  $6$ & \multicolumn{2}{c}{1.07 (0.07)} & \multicolumn{2}{c}{0.93 (0.05)} & \multicolumn{2}{c}{1.07 (0.08)} & \multicolumn{2}{c}{0.92 (0.05)} & \multicolumn{2}{c}{0.93 (0.05)} & \multicolumn{2}{c}{0.86 (0.04)} & \multicolumn{2}{c}{0.86 (0.04)} \\[0.1cm] 
  $8$ & \multicolumn{2}{c}{1.07 (0.07)} & \multicolumn{2}{c}{0.92 (0.05)} & \multicolumn{2}{c}{1.04 (0.08)} & \multicolumn{2}{c}{0.92 (0.05)} & \multicolumn{2}{c}{0.92 (0.05)} & \multicolumn{2}{c}{0.86 (0.04)} & \multicolumn{2}{c}{0.86 (0.04)} \\
\bottomrule  
    \end{tabular}  
    \end{threeparttable}  
}

\noindent\makebox[\textwidth]{%
  \fontsize{10}{8}\selectfont  
\begin{threeparttable}  
  \renewcommand\arraystretch{1}
  \caption{MSE on $\bbeta$ among different $d$ in homogeneous setting for Cauchy error case (Case 2) with $\eta=20$. The results are presented as ``average (standard deviation)''.}  
  \label{tab:ci-cauchy-mse}  
  \setlength\extrarowheight{-3pt}
    \begin{tabular}{ccccccccccccccc}  
    \toprule  
    \multirow{1}{*}{$\left| {{\mathcal{A}}} \right|$}&  
    \multicolumn{2}{c}{L}&\multicolumn{2}{c}{OTL}&\multicolumn{2}{c}{NTL}&\multicolumn{2}{c}{PTL}&\multicolumn{2}{c}{TL}&\multicolumn{2}{c}{OTLP}&\multicolumn{2}{c}{TLP}\cr
    \midrule  
    \multicolumn{2}{l}{\textbf{$d=2$}} & \\
    $0$ & \multicolumn{2}{c}{22.69 (6.69)} & \multicolumn{2}{c}{22.69 (6.69)} & \multicolumn{2}{c}{12.16 (2.20)} & \multicolumn{2}{c}{22.69 (6.69)} & \multicolumn{2}{c}{21.40 (5.80)} & \multicolumn{2}{c}{22.69 ( 6.69)} & \multicolumn{2}{c}{31.08 (11.09)} \\[0.1cm]
  $1$ & \multicolumn{2}{c}{22.72 (6.66)} & \multicolumn{2}{c}{15.20 (5.94)} & \multicolumn{2}{c}{11.96 (2.27)} & \multicolumn{2}{c}{17.39 (7.10)} & \multicolumn{2}{c}{15.78 (4.70)} & \multicolumn{2}{c}{15.56 ( 4.82)} & \multicolumn{2}{c}{23.29 ( 8.84)} \\[0.1cm]
  $2$ & \multicolumn{2}{c}{22.69 (6.68)} & \multicolumn{2}{c}{12.55 (4.52)} & \multicolumn{2}{c}{11.61 (2.40)} & \multicolumn{2}{c}{13.58 (5.24)} & \multicolumn{2}{c}{13.39 (4.58)} & \multicolumn{2}{c}{13.61 ( 4.80)} & \multicolumn{2}{c}{19.13 ( 7.38)} \\[0.1cm]
  $4$ & \multicolumn{2}{c}{22.80 (6.61)} & \multicolumn{2}{c}{11.18 (4.81)} & \multicolumn{2}{c}{11.29 (2.79)} & \multicolumn{2}{c}{11.72 (5.17)} & \multicolumn{2}{c}{11.42 (4.59)} & \multicolumn{2}{c}{11.36 ( 3.98)} & \multicolumn{2}{c}{14.93 ( 5.60)} \\[0.1cm]
  $6$ & \multicolumn{2}{c}{23.12 (7.33)} & \multicolumn{2}{c}{10.43 (4.85)} & \multicolumn{2}{c}{10.55 (2.64)} & \multicolumn{2}{c}{10.84 (5.35)} & \multicolumn{2}{c}{10.26 (4.44)} & \multicolumn{2}{c}{10.45 ( 3.86)} & \multicolumn{2}{c}{12.66 ( 5.06)} \\[0.1cm]
  $8$ & \multicolumn{2}{c}{22.89 (7.32)} & \multicolumn{2}{c}{10.15 (5.69)} & \multicolumn{2}{c}{10.05 (3.10)} & \multicolumn{2}{c}{10.36 (5.54)} & \multicolumn{2}{c}{10.10 (5.09)} & \multicolumn{2}{c}{ 9.78 ( 3.69)} & \multicolumn{2}{c}{11.37 ( 4.41)} \\
\bottomrule  
    \end{tabular}  
    \end{threeparttable}  
}

\noindent\makebox[\textwidth]{%
  \fontsize{10}{8}\selectfont  
\begin{threeparttable}  
  \renewcommand\arraystretch{1}
  \caption{Quantile loss on $\bbeta$ among different $d$ in homogeneous setting for Cauchy error case (Case 2) with $\eta=20$. The results are presented as ``average (standard deviation)''.}  
  \label{tab:ci-cauchy-ql}  
  \setlength\extrarowheight{-3pt}
    \begin{tabular}{ccccccccccccccc}  
    \toprule  
    \multirow{1}{*}{$\left| {{\mathcal{A}}} \right|$}&  
    \multicolumn{2}{c}{L}&\multicolumn{2}{c}{OTL}&\multicolumn{2}{c}{NTL}&\multicolumn{2}{c}{PTL}&\multicolumn{2}{c}{TL}&\multicolumn{2}{c}{OTLP}&\multicolumn{2}{c}{TLP}\cr
    \midrule  
    \multicolumn{2}{l}{\textbf{$d=2$}} & \\
    $0$ & \multicolumn{2}{c}{5.88 (2.39)} & \multicolumn{2}{c}{5.88 (2.39)} & \multicolumn{2}{c}{5.71 (2.42)} & \multicolumn{2}{c}{5.88 (2.39)} & \multicolumn{2}{c}{5.85 (2.39)} & \multicolumn{2}{c}{5.88 (2.39)} & \multicolumn{2}{c}{5.88 (2.41)} \\[0.1cm] 
  $1$ & \multicolumn{2}{c}{5.82 (2.39)} & \multicolumn{2}{c}{5.65 (2.39)} & \multicolumn{2}{c}{5.64 (2.43)} & \multicolumn{2}{c}{5.71 (2.39)} & \multicolumn{2}{c}{5.67 (2.39)} & \multicolumn{2}{c}{5.60 (2.41)} & \multicolumn{2}{c}{5.69 (2.42)} \\[0.1cm]
  $2$ & \multicolumn{2}{c}{5.83 (2.38)} & \multicolumn{2}{c}{5.61 (2.40)} & \multicolumn{2}{c}{5.65 (2.41)} & \multicolumn{2}{c}{5.63 (2.39)} & \multicolumn{2}{c}{5.63 (2.39)} & \multicolumn{2}{c}{5.57 (2.41)} & \multicolumn{2}{c}{5.62 (2.39)} \\[0.1cm]
  $4$ & \multicolumn{2}{c}{5.68 (2.21)} & \multicolumn{2}{c}{5.43 (2.20)} & \multicolumn{2}{c}{5.48 (2.23)} & \multicolumn{2}{c}{5.45 (2.21)} & \multicolumn{2}{c}{5.44 (2.21)} & \multicolumn{2}{c}{5.39 (2.22)} & \multicolumn{2}{c}{5.43 (2.22)} \\[0.1cm] 
  $6$ & \multicolumn{2}{c}{5.69 (2.21)} & \multicolumn{2}{c}{5.42 (2.21)} & \multicolumn{2}{c}{5.46 (2.23)} & \multicolumn{2}{c}{5.43 (2.21)} & \multicolumn{2}{c}{5.42 (2.21)} & \multicolumn{2}{c}{5.38 (2.22)} & \multicolumn{2}{c}{5.41 (2.22)} \\[0.1cm]
  $8$ & \multicolumn{2}{c}{5.72 (2.31)} & \multicolumn{2}{c}{5.46 (2.31)} & \multicolumn{2}{c}{5.48 (2.34)} & \multicolumn{2}{c}{5.47 (2.31)} & \multicolumn{2}{c}{5.46 (2.31)} & \multicolumn{2}{c}{5.41 (2.32)} & \multicolumn{2}{c}{5.44 (2.31)} \\
\bottomrule  
    \end{tabular}  
    \end{threeparttable}  
}

\noindent\makebox[\textwidth]{%
  \fontsize{10}{8}\selectfont  
\begin{threeparttable}  
  \renewcommand\arraystretch{1}
  \caption{MSE on $\bbeta$ among different $d$ in heterogeneous setting for normal error case (Case 1) with $\eta=20$. The results are presented as ``average (standard deviation)''.}  
  \label{tab:cii-normal-mse}  
  \setlength\extrarowheight{-3pt}
    \begin{tabular}{ccccccccccccccc}  
    \toprule  
    \multirow{1}{*}{$\left| {{\mathcal{A}}} \right|$}&  
    \multicolumn{2}{c}{L}&\multicolumn{2}{c}{OTL}&\multicolumn{2}{c}{NTL}&\multicolumn{2}{c}{PTL}&\multicolumn{2}{c}{TL}&\multicolumn{2}{c}{OTLP}&\multicolumn{2}{c}{TLP}\cr
    \midrule  
    \multicolumn{2}{l}{\textbf{$d=2$}} & \\
    $0$ & \multicolumn{2}{c}{1.94 (0.77)} & \multicolumn{2}{c}{ 1.94 (0.77)} & \multicolumn{2}{c}{6.51 (0.95)} & \multicolumn{2}{c}{1.94 (0.77)} & \multicolumn{2}{c}{1.94 (0.75)} & \multicolumn{2}{c}{ 1.94 (0.77)} & \multicolumn{2}{c}{3.56 (3.12)} \\[0.1cm] 
  $1$ & \multicolumn{2}{c}{1.96 (0.79)} & \multicolumn{2}{c}{ 1.31 (0.38)} & \multicolumn{2}{c}{6.17 (0.92)} & \multicolumn{2}{c}{1.31 (0.38)} & \multicolumn{2}{c}{1.33 (0.42)} & \multicolumn{2}{c}{ 0.82 (0.23)} & \multicolumn{2}{c}{1.16 (0.83)} \\[0.1cm] 
  $2$ & \multicolumn{2}{c}{1.97 (0.81)} & \multicolumn{2}{c}{ 1.18 (0.33)} & \multicolumn{2}{c}{5.79 (0.87)} & \multicolumn{2}{c}{1.18 (0.33)} & \multicolumn{2}{c}{1.20 (0.37)} & \multicolumn{2}{c}{ 0.66 (0.19)} & \multicolumn{2}{c}{0.84 (0.51)} \\[0.1cm]
  $4$ & \multicolumn{2}{c}{1.98 (0.81)} & \multicolumn{2}{c}{ 1.05 (0.27)} & \multicolumn{2}{c}{5.13 (0.77)} & \multicolumn{2}{c}{1.05 (0.28)} & \multicolumn{2}{c}{1.10 (0.38)} & \multicolumn{2}{c}{ 0.52 (0.12)} & \multicolumn{2}{c}{0.57 (0.22)} \\[0.1cm]
  $6$ & \multicolumn{2}{c}{1.98 (0.81)} & \multicolumn{2}{c}{ 0.93 (0.22)} & \multicolumn{2}{c}{4.46 (0.67)} & \multicolumn{2}{c}{0.93 (0.22)} & \multicolumn{2}{c}{1.00 (0.36)} & \multicolumn{2}{c}{ 0.46 (0.10)} & \multicolumn{2}{c}{0.49 (0.13)} \\[0.1cm]
  $8$ & \multicolumn{2}{c}{1.98 (0.81)} & \multicolumn{2}{c}{ 0.83 (0.19)} & \multicolumn{2}{c}{3.95 (0.62)} & \multicolumn{2}{c}{0.84 (0.20)} & \multicolumn{2}{c}{0.89 (0.26)} & \multicolumn{2}{c}{ 0.42 (0.09)} & \multicolumn{2}{c}{0.43 (0.10)} \\
\bottomrule  
    \end{tabular}  
    \end{threeparttable}  
}

\noindent\makebox[\textwidth]{%
  \fontsize{10}{8}\selectfont  
\begin{threeparttable}  
  \renewcommand\arraystretch{1}
  \caption{Quantile loss on $\bbeta$ among different $d$ in heterogeneous setting for normal error case (Case 1) with $\eta=20$. The results are presented as ``average (standard deviation)''.}  
  \label{tab:cii-normal-ql}  
  \setlength\extrarowheight{-3pt}
    \begin{tabular}{ccccccccccccccc}  
    \toprule  
    \multirow{1}{*}{$\left| {{\mathcal{A}}} \right|$}&  
    \multicolumn{2}{c}{L}&\multicolumn{2}{c}{OTL}&\multicolumn{2}{c}{NTL}&\multicolumn{2}{c}{PTL}&\multicolumn{2}{c}{TL}&\multicolumn{2}{c}{OTLP}&\multicolumn{2}{c}{TLP}\cr
    \midrule  
    \multicolumn{2}{l}{\textbf{$d=2$}} & \\
    $0$ & \multicolumn{2}{c}{0.50 (0.06)} & \multicolumn{2}{c}{0.50 (0.06)} & \multicolumn{2}{c}{0.87 (0.07)} & \multicolumn{2}{c}{0.50 (0.06)} & \multicolumn{2}{c}{0.50 (0.06)} & \multicolumn{2}{c}{0.50 (0.06)} & \multicolumn{2}{c}{0.57 (0.12)} \\[0.1cm]
  $1$ & \multicolumn{2}{c}{0.51 (0.06)} & \multicolumn{2}{c}{0.44 (0.04)} & \multicolumn{2}{c}{0.85 (0.07)} & \multicolumn{2}{c}{0.44 (0.04)} & \multicolumn{2}{c}{0.44 (0.04)} & \multicolumn{2}{c}{0.39 (0.03)} & \multicolumn{2}{c}{0.41 (0.05)} \\[0.1cm]
  $2$ & \multicolumn{2}{c}{0.51 (0.06)} & \multicolumn{2}{c}{0.42 (0.03)} & \multicolumn{2}{c}{0.83 (0.07)} & \multicolumn{2}{c}{0.42 (0.03)} & \multicolumn{2}{c}{0.43 (0.03)} & \multicolumn{2}{c}{0.37 (0.02)} & \multicolumn{2}{c}{0.38 (0.03)} \\[0.1cm]
  $4$ & \multicolumn{2}{c}{0.51 (0.06)} & \multicolumn{2}{c}{0.41 (0.03)} & \multicolumn{2}{c}{0.78 (0.06)} & \multicolumn{2}{c}{0.41 (0.03)} & \multicolumn{2}{c}{0.41 (0.03)} & \multicolumn{2}{c}{0.35 (0.02)} & \multicolumn{2}{c}{0.36 (0.02)} \\[0.1cm]
  $6$ & \multicolumn{2}{c}{0.51 (0.06)} & \multicolumn{2}{c}{0.39 (0.03)} & \multicolumn{2}{c}{0.72 (0.06)} & \multicolumn{2}{c}{0.40 (0.03)} & \multicolumn{2}{c}{0.40 (0.03)} & \multicolumn{2}{c}{0.34 (0.02)} & \multicolumn{2}{c}{0.35 (0.02)} \\[0.1cm]
  $8$ & \multicolumn{2}{c}{0.51 (0.06)} & \multicolumn{2}{c}{0.39 (0.03)} & \multicolumn{2}{c}{0.68 (0.06)} & \multicolumn{2}{c}{0.39 (0.03)} & \multicolumn{2}{c}{0.39 (0.03)} & \multicolumn{2}{c}{0.34 (0.02)} & \multicolumn{2}{c}{0.34 (0.02)} \\
\bottomrule  
    \end{tabular}  
    \end{threeparttable}  
}

\noindent\makebox[\textwidth]{%
  \fontsize{10}{8}\selectfont  
\begin{threeparttable}  
  \renewcommand\arraystretch{1}
  \caption{MSE on $\bbeta$ among different $d$ in heterogeneous setting for Cauchy error case (Case 2) with $\eta=20$. The results are presented as ``average (standard deviation)''.}  
  \label{tab:cii-cauchy-mse}  
  \setlength\extrarowheight{-3pt}
    \begin{tabular}{ccccccccccccccc}  
    \toprule  
    \multirow{1}{*}{$\left| {{\mathcal{A}}} \right|$}&  
    \multicolumn{2}{c}{L}&\multicolumn{2}{c}{OTL}&\multicolumn{2}{c}{NTL}&\multicolumn{2}{c}{PTL}&\multicolumn{2}{c}{TL}&\multicolumn{2}{c}{OTLP}&\multicolumn{2}{c}{TLP}\cr
    \midrule  
    \multicolumn{2}{l}{\textbf{$d=2$}} & \\
    $0$ & \multicolumn{2}{c}{11.02 (3.27)} & \multicolumn{2}{c}{11.02 (3.27)} & \multicolumn{2}{c}{12.26 (2.12)} & \multicolumn{2}{c}{11.02 (3.27)} & \multicolumn{2}{c}{10.99 (3.22)} & \multicolumn{2}{c}{11.02 (3.27)} & \multicolumn{2}{c}{14.59 (7.26)} \\[0.1cm]
  $1$ & \multicolumn{2}{c}{11.01 (3.27)} & \multicolumn{2}{c}{ 7.79 (2.16)} & \multicolumn{2}{c}{11.75 (2.05)} & \multicolumn{2}{c}{ 7.79 (2.16)} & \multicolumn{2}{c}{ 8.35 (2.58)} & \multicolumn{2}{c}{ 5.13 (1.88)} & \multicolumn{2}{c}{ 7.60 (5.03)} \\[0.1cm]
  $2$ & \multicolumn{2}{c}{11.00 (3.27)} & \multicolumn{2}{c}{ 6.70 (1.64)} & \multicolumn{2}{c}{11.20 (1.95)} & \multicolumn{2}{c}{ 6.73 (1.74)} & \multicolumn{2}{c}{ 7.49 (2.17)} & \multicolumn{2}{c}{ 4.08 (1.33)} & \multicolumn{2}{c}{ 5.47 (3.19)} \\[0.1cm]
  $4$ & \multicolumn{2}{c}{11.08 (3.20)} & \multicolumn{2}{c}{ 5.60 (1.41)} & \multicolumn{2}{c}{10.16 (1.85)} & \multicolumn{2}{c}{ 5.59 (1.40)} & \multicolumn{2}{c}{ 6.71 (2.11)} & \multicolumn{2}{c}{ 3.26 (0.94)} & \multicolumn{2}{c}{ 4.01 (2.06)} \\[0.1cm]
  $6$ & \multicolumn{2}{c}{11.13 (3.24)} & \multicolumn{2}{c}{ 4.88 (1.27)} & \multicolumn{2}{c}{ 9.22 (1.75)} & \multicolumn{2}{c}{ 4.91 (1.33)} & \multicolumn{2}{c}{ 6.19 (1.93)} & \multicolumn{2}{c}{ 2.99 (0.93)} & \multicolumn{2}{c}{ 3.49 (1.72)} \\[0.1cm]
  $8$ & \multicolumn{2}{c}{11.05 (3.27)} & \multicolumn{2}{c}{ 4.41 (1.27)} & \multicolumn{2}{c}{ 8.31 (1.73)} & \multicolumn{2}{c}{ 4.46 (1.25)} & \multicolumn{2}{c}{ 5.76 (1.90)} & \multicolumn{2}{c}{ 2.65 (0.77)} & \multicolumn{2}{c}{ 3.06 (1.38)} \\
\bottomrule  
    \end{tabular}  
    \end{threeparttable}  
}

\noindent\makebox[\textwidth]{%
  \fontsize{10}{8}\selectfont  
\begin{threeparttable}  
  \renewcommand\arraystretch{1}
  \caption{Quantile loss on $\bbeta$ among different $d$ in heterogeneous setting for Cauchy error case (Case 2) with $\eta=20$. The results are presented as ``average (standard deviation)''.}  
  \label{tab:cii-cauchy-ql}  
  \setlength\extrarowheight{-3pt}
    \begin{tabular}{ccccccccccccccc}  
    \toprule  
    \multirow{1}{*}{$\left| {{\mathcal{A}}} \right|$}&  
    \multicolumn{2}{c}{L}&\multicolumn{2}{c}{OTL}&\multicolumn{2}{c}{NTL}&\multicolumn{2}{c}{PTL}&\multicolumn{2}{c}{TL}&\multicolumn{2}{c}{OTLP}&\multicolumn{2}{c}{TLP}\cr
    \midrule  
    \multicolumn{2}{l}{\textbf{$d=2$}} & \\
    $0$ & \multicolumn{2}{c}{2.44 (0.86)} & \multicolumn{2}{c}{2.44 (0.86)} & \multicolumn{2}{c}{2.46 (0.84)} & \multicolumn{2}{c}{2.44 (0.86)} & \multicolumn{2}{c}{2.44 (0.86)} & \multicolumn{2}{c}{2.44 (0.86)} & \multicolumn{2}{c}{2.47 (0.83)} \\[0.1cm]
  $1$ & \multicolumn{2}{c}{2.42 (0.86)} & \multicolumn{2}{c}{2.26 (0.86)} & \multicolumn{2}{c}{2.41 (0.84)} & \multicolumn{2}{c}{2.26 (0.86)} & \multicolumn{2}{c}{2.29 (0.87)} & \multicolumn{2}{c}{2.18 (0.88)} & \multicolumn{2}{c}{2.22 (0.87)} \\[0.1cm]
  $2$ & \multicolumn{2}{c}{2.43 (0.85)} & \multicolumn{2}{c}{2.22 (0.86)} & \multicolumn{2}{c}{2.40 (0.83)} & \multicolumn{2}{c}{2.22 (0.85)} & \multicolumn{2}{c}{2.25 (0.85)} & \multicolumn{2}{c}{2.13 (0.85)} & \multicolumn{2}{c}{2.16 (0.85)} \\[0.1cm]
  $4$ & \multicolumn{2}{c}{2.38 (0.79)} & \multicolumn{2}{c}{2.12 (0.79)} & \multicolumn{2}{c}{2.31 (0.78)} & \multicolumn{2}{c}{2.12 (0.79)} & \multicolumn{2}{c}{2.17 (0.79)} & \multicolumn{2}{c}{2.05 (0.79)} & \multicolumn{2}{c}{2.07 (0.78)} \\[0.1cm]
  $6$ & \multicolumn{2}{c}{2.38 (0.79)} & \multicolumn{2}{c}{2.08 (0.79)} & \multicolumn{2}{c}{2.26 (0.78)} & \multicolumn{2}{c}{2.08 (0.79)} & \multicolumn{2}{c}{2.15 (0.80)} & \multicolumn{2}{c}{2.03 (0.79)} & \multicolumn{2}{c}{2.05 (0.78)} \\[0.1cm]
  $8$ & \multicolumn{2}{c}{2.39 (0.82)} & \multicolumn{2}{c}{2.08 (0.82)} & \multicolumn{2}{c}{2.24 (0.82)} & \multicolumn{2}{c}{2.08 (0.82)} & \multicolumn{2}{c}{2.15 (0.84)} & \multicolumn{2}{c}{2.03 (0.83)} & \multicolumn{2}{c}{2.04 (0.83)} \\
\bottomrule  
    \end{tabular}  
    \end{threeparttable}  
}

\noindent\makebox[\textwidth]{%
  \fontsize{10}{8}\selectfont  
\begin{threeparttable}  
  \renewcommand\arraystretch{1}
  \caption{MSE on $\bbeta$ among different $d$ in homogeneous setting for normal error case (Case 1) with $\eta=10$. The results are presented as ``average (standard deviation)''.}  
  \label{tab:c1h1eta10-mse}  
  \setlength\extrarowheight{-3pt}
    \begin{tabular}{ccccccccccccccc}  
    \toprule  
    \multirow{1}{*}{$\left| {{\mathcal{A}}} \right|$}&  
    \multicolumn{2}{c}{L}&\multicolumn{2}{c}{OTL}&\multicolumn{2}{c}{NTL}&\multicolumn{2}{c}{PTL}&\multicolumn{2}{c}{TL}&\multicolumn{2}{c}{OTLP}&\multicolumn{2}{c}{TLP}\cr
    \midrule  
    \multicolumn{2}{l}{\textbf{$d=2$}} & \\
    $0$ & \multicolumn{2}{c}{15.01 (2.89)} & \multicolumn{2}{c}{15.01 (2.89)} & \multicolumn{2}{c}{9.19 (1.85)} & \multicolumn{2}{c}{15.01 (2.89)} & \multicolumn{2}{c}{15.16 (3.06)} & \multicolumn{2}{c}{15.01 (2.89)} & \multicolumn{2}{c}{19.12 (5.05)} \\[0.1cm]
  $1$ & \multicolumn{2}{c}{15.12 (3.08)} & \multicolumn{2}{c}{10.01 (1.96)} & \multicolumn{2}{c}{8.86 (1.76)} & \multicolumn{2}{c}{10.70 (2.92)} & \multicolumn{2}{c}{10.72 (2.70)} & \multicolumn{2}{c}{10.36 (2.03)} & \multicolumn{2}{c}{12.48 (2.93)} \\[0.1cm]
  $2$ & \multicolumn{2}{c}{15.19 (3.14)} & \multicolumn{2}{c}{ 8.33 (1.76)} & \multicolumn{2}{c}{8.55 (1.78)} & \multicolumn{2}{c}{ 8.62 (2.07)} & \multicolumn{2}{c}{ 9.02 (2.32)} & \multicolumn{2}{c}{ 8.89 (1.89)} & \multicolumn{2}{c}{10.19 (2.34)} \\[0.1cm]
  $4$ & \multicolumn{2}{c}{15.28 (3.17)} & \multicolumn{2}{c}{ 6.90 (1.49)} & \multicolumn{2}{c}{8.02 (1.63)} & \multicolumn{2}{c}{ 7.07 (1.68)} & \multicolumn{2}{c}{ 7.44 (1.98)} & \multicolumn{2}{c}{ 7.08 (1.94)} & \multicolumn{2}{c}{ 7.83 (2.18)} \\[0.1cm]
  $6$ & \multicolumn{2}{c}{15.30 (3.17)} & \multicolumn{2}{c}{ 6.38 (1.55)} & \multicolumn{2}{c}{7.56 (1.66)} & \multicolumn{2}{c}{ 6.35 (1.63)} & \multicolumn{2}{c}{ 6.83 (1.97)} & \multicolumn{2}{c}{ 6.28 (1.90)} & \multicolumn{2}{c}{ 6.79 (1.95)} \\[0.1cm]
  $8$ & \multicolumn{2}{c}{15.29 (3.16)} & \multicolumn{2}{c}{ 5.94 (1.61)} & \multicolumn{2}{c}{7.11 (1.55)} & \multicolumn{2}{c}{ 6.01 (1.67)} & \multicolumn{2}{c}{ 6.41 (2.10)} & \multicolumn{2}{c}{ 5.83 (1.75)} & \multicolumn{2}{c}{ 6.17 (1.73)} \\
\bottomrule  
    \end{tabular}  
    \end{threeparttable}  
}

\noindent\makebox[\textwidth]{%
  \fontsize{10}{8}\selectfont  
\begin{threeparttable}  
  \renewcommand\arraystretch{1}
  \caption{Quantile loss on $\bbeta$ among different $d$ in homogeneous setting for normal error case (Case 1) with $\eta=10$. The results are presented as ``average (standard deviation)''.}  
  \label{tab:c1h1eta10-ql}  
  \setlength\extrarowheight{-3pt}
    \begin{tabular}{ccccccccccccccc}  
    \toprule  
    \multirow{1}{*}{$\left| {{\mathcal{A}}} \right|$}&  
    \multicolumn{2}{c}{L}&\multicolumn{2}{c}{OTL}&\multicolumn{2}{c}{NTL}&\multicolumn{2}{c}{PTL}&\multicolumn{2}{c}{TL}&\multicolumn{2}{c}{OTLP}&\multicolumn{2}{c}{TLP}\cr
    \midrule  
    \multicolumn{2}{l}{\textbf{$d=2$}} & \\
    $0$ & \multicolumn{2}{c}{1.95 (0.14)} & \multicolumn{2}{c}{1.95 (0.14)} & \multicolumn{2}{c}{1.89 (0.12)} & \multicolumn{2}{c}{1.95 (0.14)} & \multicolumn{2}{c}{1.96 (0.14)} & \multicolumn{2}{c}{1.95 (0.14)} & \multicolumn{2}{c}{2.01 (0.14)} \\[0.1cm] 
  $1$ & \multicolumn{2}{c}{1.95 (0.14)} & \multicolumn{2}{c}{1.84 (0.12)} & \multicolumn{2}{c}{1.89 (0.12)} & \multicolumn{2}{c}{1.86 (0.12)} & \multicolumn{2}{c}{1.87 (0.13)} & \multicolumn{2}{c}{1.79 (0.10)} & \multicolumn{2}{c}{1.84 (0.10)} \\[0.1cm]
  $2$ & \multicolumn{2}{c}{1.95 (0.14)} & \multicolumn{2}{c}{1.81 (0.12)} & \multicolumn{2}{c}{1.88 (0.12)} & \multicolumn{2}{c}{1.82 (0.11)} & \multicolumn{2}{c}{1.83 (0.12)} & \multicolumn{2}{c}{1.76 (0.11)} & \multicolumn{2}{c}{1.79 (0.11)} \\[0.1cm] 
  $4$ & \multicolumn{2}{c}{1.95 (0.14)} & \multicolumn{2}{c}{1.78 (0.11)} & \multicolumn{2}{c}{1.86 (0.12)} & \multicolumn{2}{c}{1.78 (0.11)} & \multicolumn{2}{c}{1.80 (0.12)} & \multicolumn{2}{c}{1.72 (0.09)} & \multicolumn{2}{c}{1.73 (0.10)} \\[0.1cm]
  $6$ & \multicolumn{2}{c}{1.96 (0.14)} & \multicolumn{2}{c}{1.78 (0.11)} & \multicolumn{2}{c}{1.84 (0.12)} & \multicolumn{2}{c}{1.78 (0.11)} & \multicolumn{2}{c}{1.79 (0.12)} & \multicolumn{2}{c}{1.71 (0.10)} & \multicolumn{2}{c}{1.72 (0.11)} \\[0.1cm] 
  $8$ & \multicolumn{2}{c}{1.96 (0.14)} & \multicolumn{2}{c}{1.77 (0.10)} & \multicolumn{2}{c}{1.83 (0.12)} & \multicolumn{2}{c}{1.78 (0.11)} & \multicolumn{2}{c}{1.79 (0.11)} & \multicolumn{2}{c}{1.70 (0.09)} & \multicolumn{2}{c}{1.71 (0.10)} \\
\bottomrule  
    \end{tabular}  
    \end{threeparttable}  
}

\noindent\makebox[\textwidth]{%
  \fontsize{10}{8}\selectfont  
\begin{threeparttable}  
  \renewcommand\arraystretch{1}
  \caption{MSE on $\bbeta$ among different $d$ in homogeneous setting for normal error case (Case 1) with $\eta=5$. The results are presented as ``average (standard deviation)''.}  
  \label{tab:c1h1eta5-mse}  
  \setlength\extrarowheight{-3pt}
    \begin{tabular}{ccccccccccccccc}  
    \toprule  
    \multirow{1}{*}{$\left| {{\mathcal{A}}} \right|$}&  
    \multicolumn{2}{c}{L}&\multicolumn{2}{c}{OTL}&\multicolumn{2}{c}{NTL}&\multicolumn{2}{c}{PTL}&\multicolumn{2}{c}{TL}&\multicolumn{2}{c}{OTLP}&\multicolumn{2}{c}{TLP}\cr
    \midrule  
    \multicolumn{2}{l}{\textbf{$d=2$}} & \\
    $0$ & \multicolumn{2}{c}{34.51 (8.05)} & \multicolumn{2}{c}{34.51 ( 8.05)} & \multicolumn{2}{c}{19.71 (5.01)} & \multicolumn{2}{c}{34.51 (8.05)} & \multicolumn{2}{c}{33.09 ( 8.73)} & \multicolumn{2}{c}{34.51 ( 8.05)} & \multicolumn{2}{c}{42.39 (12.60)} \\[0.1cm]
  $1$ & \multicolumn{2}{c}{34.42 (7.96)} & \multicolumn{2}{c}{24.63 ( 6.55)} & \multicolumn{2}{c}{19.01 (4.99)} & \multicolumn{2}{c}{27.27 (8.76)} & \multicolumn{2}{c}{24.88 ( 7.11)} & \multicolumn{2}{c}{28.39 ( 5.91)} & \multicolumn{2}{c}{34.41 ( 8.85)} \\[0.1cm]
  $2$ & \multicolumn{2}{c}{34.60 (8.12)} & \multicolumn{2}{c}{22.23 ( 6.43)} & \multicolumn{2}{c}{18.69 (5.10)} & \multicolumn{2}{c}{23.97 (7.07)} & \multicolumn{2}{c}{22.31 ( 6.70)} & \multicolumn{2}{c}{26.90 ( 5.67)} & \multicolumn{2}{c}{30.58 ( 7.31)} \\[0.1cm]
  $4$ & \multicolumn{2}{c}{34.63 (8.44)} & \multicolumn{2}{c}{19.49 ( 6.07)} & \multicolumn{2}{c}{17.96 (5.21)} & \multicolumn{2}{c}{20.66 (6.48)} & \multicolumn{2}{c}{19.36 ( 6.09)} & \multicolumn{2}{c}{24.27 ( 5.59)} & \multicolumn{2}{c}{26.80 ( 7.30)} \\[0.1cm]
  $6$ & \multicolumn{2}{c}{34.55 (8.35)} & \multicolumn{2}{c}{18.98 ( 6.91)} & \multicolumn{2}{c}{17.47 (5.31)} & \multicolumn{2}{c}{19.60 (7.77)} & \multicolumn{2}{c}{18.58 ( 6.45)} & \multicolumn{2}{c}{22.80 ( 6.04)} & \multicolumn{2}{c}{24.83 ( 6.81)} \\[0.1cm]
  $8$ & \multicolumn{2}{c}{34.41 (8.35)} & \multicolumn{2}{c}{17.82 ( 6.94)} & \multicolumn{2}{c}{17.19 (5.39)} & \multicolumn{2}{c}{18.31 (7.00)} & \multicolumn{2}{c}{17.60 ( 6.43)} & \multicolumn{2}{c}{21.92 ( 5.50)} & \multicolumn{2}{c}{23.28 ( 6.77)} \\
\bottomrule  
    \end{tabular}  
    \end{threeparttable}  
}

\noindent\makebox[\textwidth]{%
  \fontsize{10}{8}\selectfont  
\begin{threeparttable}  
  \renewcommand\arraystretch{1}
  \caption{Quantile loss on $\bbeta$ among different $d$ in homogeneous setting for normal error case (Case 1) with $\eta=5$. The results are presented as ``average (standard deviation)''.}  
  \label{tab:c1h1eta5-ql}  
  \setlength\extrarowheight{-3pt}
    \begin{tabular}{ccccccccccccccc}  
    \toprule  
    \multirow{1}{*}{$\left| {{\mathcal{A}}} \right|$}&  
    \multicolumn{2}{c}{L}&\multicolumn{2}{c}{OTL}&\multicolumn{2}{c}{NTL}&\multicolumn{2}{c}{PTL}&\multicolumn{2}{c}{TL}&\multicolumn{2}{c}{OTLP}&\multicolumn{2}{c}{TLP}\cr
    \midrule  
    \multicolumn{2}{l}{\textbf{$d=2$}} & \\
    $0$ & \multicolumn{2}{c}{3.64 (0.18)} & \multicolumn{2}{c}{3.64 (0.18)} & \multicolumn{2}{c}{3.48 (0.20)} & \multicolumn{2}{c}{3.64 (0.18)} & \multicolumn{2}{c}{3.63 (0.20)} & \multicolumn{2}{c}{3.64 (0.18)} & \multicolumn{2}{c}{3.70 (0.21)} \\[0.1cm]
  $1$ & \multicolumn{2}{c}{3.64 (0.18)} & \multicolumn{2}{c}{3.51 (0.18)} & \multicolumn{2}{c}{3.48 (0.19)} & \multicolumn{2}{c}{3.55 (0.18)} & \multicolumn{2}{c}{3.52 (0.19)} & \multicolumn{2}{c}{3.51 (0.17)} & \multicolumn{2}{c}{3.58 (0.19)} \\[0.1cm]
  $2$ & \multicolumn{2}{c}{3.65 (0.19)} & \multicolumn{2}{c}{3.48 (0.17)} & \multicolumn{2}{c}{3.47 (0.18)} & \multicolumn{2}{c}{3.51 (0.18)} & \multicolumn{2}{c}{3.49 (0.18)} & \multicolumn{2}{c}{3.46 (0.16)} & \multicolumn{2}{c}{3.50 (0.18)} \\[0.1cm]
  $4$ & \multicolumn{2}{c}{3.65 (0.19)} & \multicolumn{2}{c}{3.44 (0.18)} & \multicolumn{2}{c}{3.46 (0.19)} & \multicolumn{2}{c}{3.45 (0.18)} & \multicolumn{2}{c}{3.45 (0.19)} & \multicolumn{2}{c}{3.44 (0.17)} & \multicolumn{2}{c}{3.48 (0.19)} \\[0.1cm]
  $6$ & \multicolumn{2}{c}{3.66 (0.19)} & \multicolumn{2}{c}{3.44 (0.18)} & \multicolumn{2}{c}{3.45 (0.20)} & \multicolumn{2}{c}{3.45 (0.18)} & \multicolumn{2}{c}{3.43 (0.19)} & \multicolumn{2}{c}{3.43 (0.17)} & \multicolumn{2}{c}{3.46 (0.17)} \\[0.1cm]
  $8$ & \multicolumn{2}{c}{3.66 (0.18)} & \multicolumn{2}{c}{3.42 (0.17)} & \multicolumn{2}{c}{3.43 (0.19)} & \multicolumn{2}{c}{3.43 (0.17)} & \multicolumn{2}{c}{3.42 (0.17)} & \multicolumn{2}{c}{3.42 (0.18)} & \multicolumn{2}{c}{3.43 (0.16)} \\
\bottomrule  
    \end{tabular}  
    \end{threeparttable}  
}

\noindent\makebox[\textwidth]{%
  \fontsize{10}{8}\selectfont  
\begin{threeparttable}  
  \renewcommand\arraystretch{1}
  \caption{MSE on $\bbeta$ among different $d$ in homogeneous setting for normal error case (Case 1) with $\eta=10$ and $p=75$. The results are presented as ``average (standard deviation)''.}  
  \label{tab:c1h1p75-mse}  
  \setlength\extrarowheight{-3pt}
    \begin{tabular}{ccccccccccccccc}  
    \toprule  
    \multirow{1}{*}{$\left| {{\mathcal{A}}} \right|$}&  
    \multicolumn{2}{c}{L}&\multicolumn{2}{c}{OTL}&\multicolumn{2}{c}{NTL}&\multicolumn{2}{c}{PTL}&\multicolumn{2}{c}{TL}&\multicolumn{2}{c}{OTLP}&\multicolumn{2}{c}{TLP}\cr
    \midrule  
    \multicolumn{2}{l}{\textbf{$d=2$}} & \\
    $0$ & \multicolumn{2}{c}{14.00 (3.06)} & \multicolumn{2}{c}{14.00 (3.06)} & \multicolumn{2}{c}{13.20 (2.57)} & \multicolumn{2}{c}{14.00 (3.06)} & \multicolumn{2}{c}{14.03 (2.99)} & \multicolumn{2}{c}{14.00 (3.06)} & \multicolumn{2}{c}{15.64 (4.94)} \\[0.1cm]
  $1$ & \multicolumn{2}{c}{14.02 (3.04)} & \multicolumn{2}{c}{ 9.99 (3.00)} & \multicolumn{2}{c}{12.69 (2.65)} & \multicolumn{2}{c}{ 9.99 (3.00)} & \multicolumn{2}{c}{10.47 (3.32)} & \multicolumn{2}{c}{ 9.79 (2.56)} & \multicolumn{2}{c}{11.48 (3.74)} \\[0.1cm]
  $2$ & \multicolumn{2}{c}{13.98 (3.04)} & \multicolumn{2}{c}{ 8.59 (2.71)} & \multicolumn{2}{c}{12.27 (2.57)} & \multicolumn{2}{c}{ 8.69 (2.86)} & \multicolumn{2}{c}{ 9.13 (3.09)} & \multicolumn{2}{c}{ 7.99 (2.40)} & \multicolumn{2}{c}{ 9.06 (3.34)} \\[0.1cm]
  $4$ & \multicolumn{2}{c}{13.99 (3.00)} & \multicolumn{2}{c}{ 7.60 (2.82)} & \multicolumn{2}{c}{11.43 (2.41)} & \multicolumn{2}{c}{ 7.59 (2.68)} & \multicolumn{2}{c}{ 8.10 (2.98)} & \multicolumn{2}{c}{ 6.30 (2.00)} & \multicolumn{2}{c}{ 6.73 (2.04)} \\[0.1cm]
  $6$ & \multicolumn{2}{c}{13.96 (3.09)} & \multicolumn{2}{c}{ 7.11 (2.29)} & \multicolumn{2}{c}{10.35 (2.27)} & \multicolumn{2}{c}{ 7.14 (2.37)} & \multicolumn{2}{c}{ 7.51 (2.62)} & \multicolumn{2}{c}{ 5.73 (1.97)} & \multicolumn{2}{c}{ 6.10 (2.00)} \\[0.1cm]
  $8$ & \multicolumn{2}{c}{13.91 (3.10)} & \multicolumn{2}{c}{ 6.99 (2.59)} & \multicolumn{2}{c}{ 9.65 (2.09)} & \multicolumn{2}{c}{ 6.99 (2.62)} & \multicolumn{2}{c}{ 7.38 (2.82)} & \multicolumn{2}{c}{ 5.22 (1.77)} & \multicolumn{2}{c}{ 5.53 (1.87)} \\
\bottomrule  
    \end{tabular}  
    \end{threeparttable}  
}

\noindent\makebox[\textwidth]{%
  \fontsize{10}{8}\selectfont  
\begin{threeparttable}  
  \renewcommand\arraystretch{1}
  \caption{Quantile loss on $\bbeta$ among different $d$ in homogeneous setting for normal error case (Case 1) with $\eta=10$ and $p=75$. The results are presented as ``average (standard deviation)''.}  
  \label{tab:c1h1p75-ql}  
  \setlength\extrarowheight{-3pt}
    \begin{tabular}{ccccccccccccccc}  
    \toprule  
    \multirow{1}{*}{$\left| {{\mathcal{A}}} \right|$}&  
    \multicolumn{2}{c}{L}&\multicolumn{2}{c}{OTL}&\multicolumn{2}{c}{NTL}&\multicolumn{2}{c}{PTL}&\multicolumn{2}{c}{TL}&\multicolumn{2}{c}{OTLP}&\multicolumn{2}{c}{TLP}\cr
    \midrule  
    \multicolumn{2}{l}{\textbf{$d=2$}} & \\
    $0$ & \multicolumn{2}{c}{1.91 (0.11)} & \multicolumn{2}{c}{1.91 (0.11)} & \multicolumn{2}{c}{1.95 (0.13)} & \multicolumn{2}{c}{1.91 (0.11)} & \multicolumn{2}{c}{1.92 (0.11)} & \multicolumn{2}{c}{1.91 (0.11)} & \multicolumn{2}{c}{1.94 (0.12)} \\[0.1cm] 
  $1$ & \multicolumn{2}{c}{1.91 (0.11)} & \multicolumn{2}{c}{1.82 (0.10)} & \multicolumn{2}{c}{1.93 (0.13)} & \multicolumn{2}{c}{1.82 (0.10)} & \multicolumn{2}{c}{1.84 (0.10)} & \multicolumn{2}{c}{1.78 (0.09)} & \multicolumn{2}{c}{1.81 (0.09)} \\[0.1cm] 
  $2$ & \multicolumn{2}{c}{1.91 (0.11)} & \multicolumn{2}{c}{1.80 (0.10)} & \multicolumn{2}{c}{1.92 (0.12)} & \multicolumn{2}{c}{1.80 (0.10)} & \multicolumn{2}{c}{1.82 (0.11)} & \multicolumn{2}{c}{1.74 (0.08)} & \multicolumn{2}{c}{1.77 (0.09)} \\[0.1cm] 
  $4$ & \multicolumn{2}{c}{1.92 (0.11)} & \multicolumn{2}{c}{1.79 (0.10)} & \multicolumn{2}{c}{1.90 (0.11)} & \multicolumn{2}{c}{1.79 (0.10)} & \multicolumn{2}{c}{1.80 (0.10)} & \multicolumn{2}{c}{1.72 (0.08)} & \multicolumn{2}{c}{1.72 (0.08)} \\[0.1cm] 
  $6$ & \multicolumn{2}{c}{1.92 (0.11)} & \multicolumn{2}{c}{1.79 (0.10)} & \multicolumn{2}{c}{1.88 (0.10)} & \multicolumn{2}{c}{1.79 (0.10)} & \multicolumn{2}{c}{1.80 (0.10)} & \multicolumn{2}{c}{1.72 (0.08)} & \multicolumn{2}{c}{1.73 (0.08)} \\[0.1cm] 
  $8$ & \multicolumn{2}{c}{1.92 (0.11)} & \multicolumn{2}{c}{1.79 (0.10)} & \multicolumn{2}{c}{1.86 (0.10)} & \multicolumn{2}{c}{1.79 (0.10)} & \multicolumn{2}{c}{1.80 (0.10)} & \multicolumn{2}{c}{1.70 (0.08)} & \multicolumn{2}{c}{1.71 (0.08)} \\
\bottomrule  
    \end{tabular}  
    \end{threeparttable}  
}

\noindent\makebox[\textwidth]{%
  \fontsize{10}{8}\selectfont  
\begin{threeparttable}  
  \renewcommand\arraystretch{1}
  \caption{MSE on $\bbeta$ among different $d$ in homogeneous setting for normal error case (Case 1) with $\eta=10$ and $p=225$. The results are presented as ``average (standard deviation)''.}  
  \label{tab:c1h1p225-mse}  
  \setlength\extrarowheight{-3pt}
    \begin{tabular}{ccccccccccccccc}  
    \toprule  
    \multirow{1}{*}{$\left| {{\mathcal{A}}} \right|$}&  
    \multicolumn{2}{c}{L}&\multicolumn{2}{c}{OTL}&\multicolumn{2}{c}{NTL}&\multicolumn{2}{c}{PTL}&\multicolumn{2}{c}{TL}&\multicolumn{2}{c}{OTLP}&\multicolumn{2}{c}{TLP}\cr
    \midrule  
    \multicolumn{2}{l}{\textbf{$d=2$}} & \\
    $0$ & \multicolumn{2}{c}{15.99 (3.22)} & \multicolumn{2}{c}{15.99 (3.22)} & \multicolumn{2}{c}{7.68 (1.46)} & \multicolumn{2}{c}{15.99 (3.22)} & \multicolumn{2}{c}{15.30 (3.76)} & \multicolumn{2}{c}{15.99 (3.22)} & \multicolumn{2}{c}{17.89 (3.90)} \\[0.1cm] 
  $1$ & \multicolumn{2}{c}{15.99 (3.21)} & \multicolumn{2}{c}{10.10 (2.09)} & \multicolumn{2}{c}{7.45 (1.41)} & \multicolumn{2}{c}{12.16 (3.73)} & \multicolumn{2}{c}{10.95 (2.54)} & \multicolumn{2}{c}{11.12 (2.20)} & \multicolumn{2}{c}{13.35 (3.75)} \\[0.1cm] 
  $2$ & \multicolumn{2}{c}{15.93 (3.26)} & \multicolumn{2}{c}{ 8.29 (1.74)} & \multicolumn{2}{c}{7.22 (1.41)} & \multicolumn{2}{c}{ 9.42 (2.80)} & \multicolumn{2}{c}{ 8.84 (2.00)} & \multicolumn{2}{c}{ 9.05 (1.69)} & \multicolumn{2}{c}{10.97 (3.10)} \\[0.1cm] 
  $4$ & \multicolumn{2}{c}{15.95 (3.26)} & \multicolumn{2}{c}{ 6.66 (1.52)} & \multicolumn{2}{c}{6.76 (1.41)} & \multicolumn{2}{c}{ 7.06 (1.56)} & \multicolumn{2}{c}{ 7.02 (1.53)} & \multicolumn{2}{c}{ 7.30 (1.80)} & \multicolumn{2}{c}{ 8.50 (2.41)} \\[0.1cm] 
  $6$ & \multicolumn{2}{c}{16.00 (3.28)} & \multicolumn{2}{c}{ 5.83 (1.41)} & \multicolumn{2}{c}{6.34 (1.38)} & \multicolumn{2}{c}{ 6.13 (1.43)} & \multicolumn{2}{c}{ 6.09 (1.36)} & \multicolumn{2}{c}{ 6.57 (1.73)} & \multicolumn{2}{c}{ 7.38 (2.18)} \\[0.1cm] 
  $8$ & \multicolumn{2}{c}{16.03 (3.26)} & \multicolumn{2}{c}{ 5.48 (1.41)} & \multicolumn{2}{c}{5.90 (1.34)} & \multicolumn{2}{c}{ 5.69 (1.46)} & \multicolumn{2}{c}{ 5.63 (1.39)} & \multicolumn{2}{c}{ 6.02 (1.61)} & \multicolumn{2}{c}{ 6.51 (1.88)} \\
\bottomrule  
    \end{tabular}  
    \end{threeparttable}  
}

\noindent\makebox[\textwidth]{%
  \fontsize{10}{8}\selectfont  
\begin{threeparttable}  
  \renewcommand\arraystretch{1}
  \caption{Quantile loss on $\bbeta$ among different $d$ in homogeneous setting for normal error case (Case 1) with $\eta=10$ and $p=225$. The results are presented as ``average (standard deviation)''.}  
  \label{tab:c1h1p225-ql}  
  \setlength\extrarowheight{-3pt}
    \begin{tabular}{ccccccccccccccc}  
    \toprule  
    \multirow{1}{*}{$\left| {{\mathcal{A}}} \right|$}&  
    \multicolumn{2}{c}{L}&\multicolumn{2}{c}{OTL}&\multicolumn{2}{c}{NTL}&\multicolumn{2}{c}{PTL}&\multicolumn{2}{c}{TL}&\multicolumn{2}{c}{OTLP}&\multicolumn{2}{c}{TLP}\cr
    \midrule  
    \multicolumn{2}{l}{\textbf{$d=2$}} & \\
    $0$ & \multicolumn{2}{c}{2.00 (0.12)} & \multicolumn{2}{c}{2.00 (0.12)} & \multicolumn{2}{c}{1.87 (0.10)} & \multicolumn{2}{c}{2.00 (0.12)} & \multicolumn{2}{c}{2.00 (0.13)} & \multicolumn{2}{c}{2.00 (0.12)} & \multicolumn{2}{c}{1.99 (0.11)} \\[0.1cm]
  $1$ & \multicolumn{2}{c}{2.00 (0.12)} & \multicolumn{2}{c}{1.86 (0.11)} & \multicolumn{2}{c}{1.87 (0.10)} & \multicolumn{2}{c}{1.90 (0.12)} & \multicolumn{2}{c}{1.90 (0.11)} & \multicolumn{2}{c}{1.83 (0.11)} & \multicolumn{2}{c}{1.87 (0.11)} \\[0.1cm]
  $2$ & \multicolumn{2}{c}{1.99 (0.12)} & \multicolumn{2}{c}{1.82 (0.11)} & \multicolumn{2}{c}{1.85 (0.10)} & \multicolumn{2}{c}{1.84 (0.11)} & \multicolumn{2}{c}{1.85 (0.11)} & \multicolumn{2}{c}{1.78 (0.09)} & \multicolumn{2}{c}{1.82 (0.09)} \\[0.1cm]
  $4$ & \multicolumn{2}{c}{2.00 (0.12)} & \multicolumn{2}{c}{1.79 (0.10)} & \multicolumn{2}{c}{1.84 (0.10)} & \multicolumn{2}{c}{1.80 (0.10)} & \multicolumn{2}{c}{1.81 (0.10)} & \multicolumn{2}{c}{1.74 (0.09)} & \multicolumn{2}{c}{1.76 (0.09)} \\[0.1cm]
  $6$ & \multicolumn{2}{c}{2.00 (0.12)} & \multicolumn{2}{c}{1.78 (0.10)} & \multicolumn{2}{c}{1.83 (0.09)} & \multicolumn{2}{c}{1.78 (0.10)} & \multicolumn{2}{c}{1.79 (0.10)} & \multicolumn{2}{c}{1.73 (0.09)} & \multicolumn{2}{c}{1.74 (0.08)} \\[0.1cm]
  $8$ & \multicolumn{2}{c}{2.00 (0.12)} & \multicolumn{2}{c}{1.77 (0.10)} & \multicolumn{2}{c}{1.81 (0.09)} & \multicolumn{2}{c}{1.77 (0.10)} & \multicolumn{2}{c}{1.78 (0.10)} & \multicolumn{2}{c}{1.72 (0.09)} & \multicolumn{2}{c}{1.73 (0.09)} \\
\bottomrule  
    \end{tabular}  
    \end{threeparttable}  
}

\noindent\makebox[\textwidth]{%
  \fontsize{10}{8}\selectfont  
\begin{threeparttable}  
  \renewcommand\arraystretch{1}
  \caption{MSE on $\bbeta$ among different $d$ in heterogeneous setting for normal error case (Case 1) with $\eta=10$ and $p=75$. The results are presented as ``average (standard deviation)''.}  
  \label{tab:c1h2p75-mse}  
  \setlength\extrarowheight{-3pt}
    \begin{tabular}{ccccccccccccccc}  
    \toprule  
    \multirow{1}{*}{$\left| {{\mathcal{A}}} \right|$}&  
    \multicolumn{2}{c}{L}&\multicolumn{2}{c}{OTL}&\multicolumn{2}{c}{NTL}&\multicolumn{2}{c}{PTL}&\multicolumn{2}{c}{TL}&\multicolumn{2}{c}{OTLP}&\multicolumn{2}{c}{TLP}\cr
    \midrule  
    \multicolumn{2}{l}{\textbf{$d=2$}} & \\
    $0$ & \multicolumn{2}{c}{5.08 (1.69)} & \multicolumn{2}{c}{ 5.08 (1.69)} & \multicolumn{2}{c}{13.29 (2.44)} & \multicolumn{2}{c}{ 5.08 (1.69)} & \multicolumn{2}{c}{5.08 (1.72)} & \multicolumn{2}{c}{ 5.08 (1.69)} & \multicolumn{2}{c}{5.39 (2.04)} \\[0.1cm] 
  $1$ & \multicolumn{2}{c}{5.08 (1.69)} & \multicolumn{2}{c}{ 3.59 (1.03)} & \multicolumn{2}{c}{12.69 (2.48)} & \multicolumn{2}{c}{ 3.59 (1.03)} & \multicolumn{2}{c}{4.06 (1.54)} & \multicolumn{2}{c}{ 2.45 (0.70)} & \multicolumn{2}{c}{2.66 (0.91)} \\[0.1cm] 
  $2$ & \multicolumn{2}{c}{5.09 (1.69)} & \multicolumn{2}{c}{ 3.24 (0.86)} & \multicolumn{2}{c}{12.07 (2.47)} & \multicolumn{2}{c}{ 3.24 (0.87)} & \multicolumn{2}{c}{3.81 (1.53)} & \multicolumn{2}{c}{ 2.07 (0.63)} & \multicolumn{2}{c}{2.19 (0.68)} \\[0.1cm] 
  $4$ & \multicolumn{2}{c}{5.12 (1.69)} & \multicolumn{2}{c}{ 2.86 (0.76)} & \multicolumn{2}{c}{11.01 (2.34)} & \multicolumn{2}{c}{ 2.87 (0.76)} & \multicolumn{2}{c}{3.73 (1.53)} & \multicolumn{2}{c}{ 1.67 (0.50)} & \multicolumn{2}{c}{1.74 (0.53)} \\[0.1cm] 
  $6$ & \multicolumn{2}{c}{5.09 (1.69)} & \multicolumn{2}{c}{ 2.55 (0.65)} & \multicolumn{2}{c}{ 9.66 (2.01)} & \multicolumn{2}{c}{ 2.55 (0.65)} & \multicolumn{2}{c}{3.52 (1.45)} & \multicolumn{2}{c}{ 1.48 (0.36)} & \multicolumn{2}{c}{1.55 (0.45)} \\[0.1cm] 
  $8$ & \multicolumn{2}{c}{5.07 (1.69)} & \multicolumn{2}{c}{ 2.37 (0.54)} & \multicolumn{2}{c}{ 8.43 (1.77)} & \multicolumn{2}{c}{ 2.39 (0.54)} & \multicolumn{2}{c}{3.31 (1.46)} & \multicolumn{2}{c}{ 1.32 (0.32)} & \multicolumn{2}{c}{1.42 (0.46)} \\
\bottomrule  
    \end{tabular}  
    \end{threeparttable}  
}

\noindent\makebox[\textwidth]{%
  \fontsize{10}{8}\selectfont  
\begin{threeparttable}  
  \renewcommand\arraystretch{1}
  \caption{Quantile loss on $\bbeta$ among different $d$ in heterogeneous setting for normal error case (Case 1) with $\eta=10$ and $p=75$. The results are presented as ``average (standard deviation)''.}  
  \label{tab:c1h2p75-ql}  
  \setlength\extrarowheight{-3pt}
    \begin{tabular}{ccccccccccccccc}  
    \toprule  
    \multirow{1}{*}{$\left| {{\mathcal{A}}} \right|$}&  
    \multicolumn{2}{c}{L}&\multicolumn{2}{c}{OTL}&\multicolumn{2}{c}{NTL}&\multicolumn{2}{c}{PTL}&\multicolumn{2}{c}{TL}&\multicolumn{2}{c}{OTLP}&\multicolumn{2}{c}{TLP}\cr
    \midrule  
    \multicolumn{2}{l}{\textbf{$d=2$}} & \\
    $0$ & \multicolumn{2}{c}{0.87 (0.08)} & \multicolumn{2}{c}{0.87 (0.08)} & \multicolumn{2}{c}{1.27 (0.09)} & \multicolumn{2}{c}{0.87 (0.08)} & \multicolumn{2}{c}{0.87 (0.08)} & \multicolumn{2}{c}{0.87 (0.08)} & \multicolumn{2}{c}{0.88 (0.08)} \\[0.1cm]
  $1$ & \multicolumn{2}{c}{0.87 (0.08)} & \multicolumn{2}{c}{0.78 (0.05)} & \multicolumn{2}{c}{1.25 (0.09)} & \multicolumn{2}{c}{0.78 (0.05)} & \multicolumn{2}{c}{0.81 (0.07)} & \multicolumn{2}{c}{0.73 (0.04)} & \multicolumn{2}{c}{0.73 (0.05)} \\[0.1cm]
  $2$ & \multicolumn{2}{c}{0.87 (0.08)} & \multicolumn{2}{c}{0.76 (0.05)} & \multicolumn{2}{c}{1.23 (0.10)} & \multicolumn{2}{c}{0.76 (0.05)} & \multicolumn{2}{c}{0.79 (0.07)} & \multicolumn{2}{c}{0.70 (0.05)} & \multicolumn{2}{c}{0.70 (0.05)} \\[0.1cm]
  $4$ & \multicolumn{2}{c}{0.87 (0.08)} & \multicolumn{2}{c}{0.74 (0.04)} & \multicolumn{2}{c}{1.18 (0.09)} & \multicolumn{2}{c}{0.74 (0.04)} & \multicolumn{2}{c}{0.79 (0.07)} & \multicolumn{2}{c}{0.67 (0.04)} & \multicolumn{2}{c}{0.68 (0.04)} \\[0.1cm]
  $6$ & \multicolumn{2}{c}{0.87 (0.08)} & \multicolumn{2}{c}{0.73 (0.05)} & \multicolumn{2}{c}{1.12 (0.08)} & \multicolumn{2}{c}{0.73 (0.04)} & \multicolumn{2}{c}{0.78 (0.07)} & \multicolumn{2}{c}{0.66 (0.04)} & \multicolumn{2}{c}{0.67 (0.04)} \\[0.1cm]
  $8$ & \multicolumn{2}{c}{0.87 (0.08)} & \multicolumn{2}{c}{0.71 (0.04)} & \multicolumn{2}{c}{1.05 (0.08)} & \multicolumn{2}{c}{0.71 (0.04)} & \multicolumn{2}{c}{0.76 (0.07)} & \multicolumn{2}{c}{0.65 (0.03)} & \multicolumn{2}{c}{0.65 (0.04)} \\
\bottomrule  
    \end{tabular}  
    \end{threeparttable}  
}

\noindent\makebox[\textwidth]{%
  \fontsize{10}{8}\selectfont  
\begin{threeparttable}  
  \renewcommand\arraystretch{1}
  \caption{MSE on $\bbeta$ among different $d$ in heterogeneous setting for normal error case (Case 1) with $\eta=10$ and $p=225$. The results are presented as ``average (standard deviation)''.}  
  \label{tab:c1h2p225-mse} 
  \setlength\extrarowheight{-3pt} 
    \begin{tabular}{ccccccccccccccc}  
    \toprule  
    \multirow{1}{*}{$\left| {{\mathcal{A}}} \right|$}&  
    \multicolumn{2}{c}{L}&\multicolumn{2}{c}{OTL}&\multicolumn{2}{c}{NTL}&\multicolumn{2}{c}{PTL}&\multicolumn{2}{c}{TL}&\multicolumn{2}{c}{OTLP}&\multicolumn{2}{c}{TLP}\cr
    \midrule  
    \multicolumn{2}{l}{\textbf{$d=2$}} & \\
    $0$ & \multicolumn{2}{c}{8.51 (2.14)} & \multicolumn{2}{c}{ 8.51 (2.14)} & \multicolumn{2}{c}{7.88 (1.05)} & \multicolumn{2}{c}{ 8.51 (2.14)} & \multicolumn{2}{c}{8.59 (2.12)} & \multicolumn{2}{c}{ 8.51 (2.14)} & \multicolumn{2}{c}{9.66 (2.77)} \\[0.1cm]
  $1$ & \multicolumn{2}{c}{8.49 (2.16)} & \multicolumn{2}{c}{ 5.55 (1.21)} & \multicolumn{2}{c}{7.58 (1.05)} & \multicolumn{2}{c}{ 5.58 (1.25)} & \multicolumn{2}{c}{5.90 (1.47)} & \multicolumn{2}{c}{ 3.63 (0.95)} & \multicolumn{2}{c}{4.58 (1.44)} \\[0.1cm]
  $2$ & \multicolumn{2}{c}{8.49 (2.16)} & \multicolumn{2}{c}{ 4.67 (0.98)} & \multicolumn{2}{c}{7.20 (0.99)} & \multicolumn{2}{c}{ 4.72 (1.02)} & \multicolumn{2}{c}{5.23 (1.34)} & \multicolumn{2}{c}{ 2.85 (0.66)} & \multicolumn{2}{c}{3.47 (0.99)} \\[0.1cm]
  $4$ & \multicolumn{2}{c}{8.53 (2.13)} & \multicolumn{2}{c}{ 3.91 (0.86)} & \multicolumn{2}{c}{6.45 (0.91)} & \multicolumn{2}{c}{ 4.00 (0.90)} & \multicolumn{2}{c}{4.83 (1.42)} & \multicolumn{2}{c}{ 2.37 (0.58)} & \multicolumn{2}{c}{2.66 (0.70)} \\[0.1cm]
  $6$ & \multicolumn{2}{c}{8.48 (2.11)} & \multicolumn{2}{c}{ 3.38 (0.67)} & \multicolumn{2}{c}{5.86 (0.77)} & \multicolumn{2}{c}{ 3.48 (0.67)} & \multicolumn{2}{c}{4.45 (1.35)} & \multicolumn{2}{c}{ 2.08 (0.43)} & \multicolumn{2}{c}{2.30 (0.50)} \\[0.1cm]
  $8$ & \multicolumn{2}{c}{8.47 (2.12)} & \multicolumn{2}{c}{ 3.05 (0.59)} & \multicolumn{2}{c}{5.25 (0.69)} & \multicolumn{2}{c}{ 3.15 (0.63)} & \multicolumn{2}{c}{4.13 (1.32)} & \multicolumn{2}{c}{ 1.82 (0.36)} & \multicolumn{2}{c}{1.94 (0.41)} \\
\bottomrule  
    \end{tabular}  
    \end{threeparttable}  
}

\noindent\makebox[\textwidth]{%
  \fontsize{10}{8}\selectfont  
\begin{threeparttable}  
  \renewcommand\arraystretch{1}
  \caption{Quantile loss on $\bbeta$ among different $d$ in heterogeneous setting for normal error case (Case 1) with $\eta=10$ and $p=225$. The results are presented as ``average (standard deviation)''.}  
  \label{tab:c1h2p225-ql}  
  \setlength\extrarowheight{-3pt}
    \begin{tabular}{ccccccccccccccc}  
    \toprule  
    \multirow{1}{*}{$\left| {{\mathcal{A}}} \right|$}&  
    \multicolumn{2}{c}{L}&\multicolumn{2}{c}{OTL}&\multicolumn{2}{c}{NTL}&\multicolumn{2}{c}{PTL}&\multicolumn{2}{c}{TL}&\multicolumn{2}{c}{OTLP}&\multicolumn{2}{c}{TLP}\cr
    \midrule  
    \multicolumn{2}{l}{\textbf{$d=2$}} & \\
    $0$ & \multicolumn{2}{c}{1.07 (0.09)} & \multicolumn{2}{c}{1.07 (0.09)} & \multicolumn{2}{c}{1.02 (0.07)} & \multicolumn{2}{c}{1.07 (0.09)} & \multicolumn{2}{c}{1.07 (0.09)} & \multicolumn{2}{c}{1.07 (0.09)} & \multicolumn{2}{c}{1.10 (0.11)} \\[0.1cm]
  $1$ & \multicolumn{2}{c}{1.06 (0.09)} & \multicolumn{2}{c}{0.90 (0.06)} & \multicolumn{2}{c}{1.00 (0.07)} & \multicolumn{2}{c}{0.90 (0.06)} & \multicolumn{2}{c}{0.92 (0.07)} & \multicolumn{2}{c}{0.80 (0.06)} & \multicolumn{2}{c}{0.85 (0.07)} \\[0.1cm]
  $2$ & \multicolumn{2}{c}{1.06 (0.09)} & \multicolumn{2}{c}{0.86 (0.07)} & \multicolumn{2}{c}{0.99 (0.07)} & \multicolumn{2}{c}{0.86 (0.07)} & \multicolumn{2}{c}{0.89 (0.08)} & \multicolumn{2}{c}{0.75 (0.04)} & \multicolumn{2}{c}{0.78 (0.06)} \\[0.1cm]
  $4$ & \multicolumn{2}{c}{1.06 (0.09)} & \multicolumn{2}{c}{0.81 (0.06)} & \multicolumn{2}{c}{0.95 (0.07)} & \multicolumn{2}{c}{0.81 (0.06)} & \multicolumn{2}{c}{0.86 (0.08)} & \multicolumn{2}{c}{0.72 (0.05)} & \multicolumn{2}{c}{0.74 (0.05)} \\[0.1cm]
  $6$ & \multicolumn{2}{c}{1.06 (0.09)} & \multicolumn{2}{c}{0.77 (0.05)} & \multicolumn{2}{c}{0.92 (0.06)} & \multicolumn{2}{c}{0.78 (0.05)} & \multicolumn{2}{c}{0.84 (0.08)} & \multicolumn{2}{c}{0.70 (0.04)} & \multicolumn{2}{c}{0.71 (0.04)} \\[0.1cm]
  $8$ & \multicolumn{2}{c}{1.06 (0.09)} & \multicolumn{2}{c}{0.75 (0.04)} & \multicolumn{2}{c}{0.89 (0.06)} & \multicolumn{2}{c}{0.76 (0.05)} & \multicolumn{2}{c}{0.82 (0.07)} & \multicolumn{2}{c}{0.68 (0.04)} & \multicolumn{2}{c}{0.69 (0.04)} \\
\bottomrule  
    \end{tabular}  
    \end{threeparttable}  
}

\noindent\makebox[\textwidth]{%
  \fontsize{10}{8}\selectfont  
\begin{threeparttable}  
  \renewcommand\arraystretch{1}
  \caption{MSE on $\bbeta$ among different $d$ in homogeneous setting for Cauchy error case (Case 2) with $\eta=10$ and $p=75$. The results are presented as ``average (standard deviation)''.}  
  \label{tab:c2h1p75-mse}  
  \setlength\extrarowheight{-3pt} 
    \begin{tabular}{ccccccccccccccc}  
    \toprule  
    \multirow{1}{*}{$\left| {{\mathcal{A}}} \right|$}&  
    \multicolumn{2}{c}{L}&\multicolumn{2}{c}{OTL}&\multicolumn{2}{c}{NTL}&\multicolumn{2}{c}{PTL}&\multicolumn{2}{c}{TL}&\multicolumn{2}{c}{OTLP}&\multicolumn{2}{c}{TLP}\cr
    \midrule  
    \multicolumn{2}{l}{\textbf{$d=2$}} & \\
    $0$ & \multicolumn{2}{c}{40.12 (14.41)} & \multicolumn{2}{c}{40.12 (14.41)} & \multicolumn{2}{c}{34.85 (11.57)} & \multicolumn{2}{c}{40.12 (14.41)} & \multicolumn{2}{c}{37.96 (12.91)} & \multicolumn{2}{c}{40.12 (14.41)} & \multicolumn{2}{c}{51.10 (22.99)} \\[0.1cm]
  $1$ & \multicolumn{2}{c}{39.97 (14.33)} & \multicolumn{2}{c}{31.66 (12.38)} & \multicolumn{2}{c}{34.03 (10.73)} & \multicolumn{2}{c}{33.86 (14.09)} & \multicolumn{2}{c}{32.38 (12.94)} & \multicolumn{2}{c}{34.15 (12.34)} & \multicolumn{2}{c}{40.85 (17.87)} \\[0.1cm]
  $2$ & \multicolumn{2}{c}{39.97 (14.33)} & \multicolumn{2}{c}{28.98 (12.44)} & \multicolumn{2}{c}{33.06 (10.29)} & \multicolumn{2}{c}{30.29 (12.69)} & \multicolumn{2}{c}{30.62 (13.74)} & \multicolumn{2}{c}{33.10 (14.09)} & \multicolumn{2}{c}{36.84 (16.34)} \\[0.1cm] 
  $4$ & \multicolumn{2}{c}{39.81 (14.30)} & \multicolumn{2}{c}{26.36 (11.90)} & \multicolumn{2}{c}{30.99 ( 9.84)} & \multicolumn{2}{c}{27.10 (11.94)} & \multicolumn{2}{c}{27.81 (12.42)} & \multicolumn{2}{c}{29.28 (13.42)} & \multicolumn{2}{c}{32.73 (16.22)} \\[0.1cm] 
  $6$ & \multicolumn{2}{c}{39.76 (14.23)} & \multicolumn{2}{c}{25.15 (11.32)} & \multicolumn{2}{c}{28.98 (10.10)} & \multicolumn{2}{c}{25.16 (11.51)} & \multicolumn{2}{c}{26.17 (12.93)} & \multicolumn{2}{c}{28.70 (13.43)} & \multicolumn{2}{c}{30.35 (14.39)} \\[0.1cm] 
  $8$ & \multicolumn{2}{c}{40.22 (14.21)} & \multicolumn{2}{c}{24.33 (11.78)} & \multicolumn{2}{c}{27.71 (10.51)} & \multicolumn{2}{c}{24.38 (11.32)} & \multicolumn{2}{c}{25.78 (12.65)} & \multicolumn{2}{c}{28.77 (13.38)} & \multicolumn{2}{c}{29.46 (14.21)} \\
\bottomrule  
    \end{tabular}  
    \end{threeparttable}  
}

\noindent\makebox[\textwidth]{%
  \fontsize{10}{8}\selectfont  
\begin{threeparttable}  
  \renewcommand\arraystretch{1}
  \caption{Quantile loss on $\bbeta$ among different $d$ in homogeneous setting for Cauchy error case (Case 2) with $\eta=10$ and $p=75$. The results are presented as ``average (standard deviation)''.}  
  \label{tab:c2h1p75-ql}   
  \setlength\extrarowheight{-3pt}
    \begin{tabular}{ccccccccccccccc}  
    \toprule  
    \multirow{1}{*}{$\left| {{\mathcal{A}}} \right|$}&  
    \multicolumn{2}{c}{L}&\multicolumn{2}{c}{OTL}&\multicolumn{2}{c}{NTL}&\multicolumn{2}{c}{PTL}&\multicolumn{2}{c}{TL}&\multicolumn{2}{c}{OTLP}&\multicolumn{2}{c}{TLP}\cr
    \midrule  
    \multicolumn{2}{l}{\textbf{$d=2$}} & \\
    $0$ & \multicolumn{2}{c}{12.58 (4.89)} & \multicolumn{2}{c}{12.58 (4.89)} & \multicolumn{2}{c}{12.52 (4.86)} & \multicolumn{2}{c}{12.58 (4.89)} & \multicolumn{2}{c}{12.57 (4.86)} & \multicolumn{2}{c}{12.58 (4.89)} & \multicolumn{2}{c}{12.65 (4.89)} \\[0.1cm] 
  $1$ & \multicolumn{2}{c}{12.55 (4.91)} & \multicolumn{2}{c}{12.45 (4.90)} & \multicolumn{2}{c}{12.48 (4.87)} & \multicolumn{2}{c}{12.48 (4.91)} & \multicolumn{2}{c}{12.47 (4.90)} & \multicolumn{2}{c}{12.41 (4.89)} & \multicolumn{2}{c}{12.48 (4.92)} \\[0.1cm] 
  $2$ & \multicolumn{2}{c}{12.57 (4.90)} & \multicolumn{2}{c}{12.44 (4.89)} & \multicolumn{2}{c}{12.48 (4.87)} & \multicolumn{2}{c}{12.45 (4.90)} & \multicolumn{2}{c}{12.45 (4.87)} & \multicolumn{2}{c}{12.40 (4.90)} & \multicolumn{2}{c}{12.46 (4.89)} \\[0.1cm] 
  $4$ & \multicolumn{2}{c}{12.70 (4.81)} & \multicolumn{2}{c}{12.53 (4.81)} & \multicolumn{2}{c}{12.58 (4.79)} & \multicolumn{2}{c}{12.53 (4.81)} & \multicolumn{2}{c}{12.55 (4.79)} & \multicolumn{2}{c}{12.48 (4.82)} & \multicolumn{2}{c}{12.53 (4.82)} \\[0.1cm] 
  $6$ & \multicolumn{2}{c}{12.80 (4.80)} & \multicolumn{2}{c}{12.62 (4.81)} & \multicolumn{2}{c}{12.66 (4.78)} & \multicolumn{2}{c}{12.62 (4.80)} & \multicolumn{2}{c}{12.64 (4.80)} & \multicolumn{2}{c}{12.58 (4.82)} & \multicolumn{2}{c}{12.61 (4.81)} \\[0.1cm] 
  $8$ & \multicolumn{2}{c}{12.96 (4.86)} & \multicolumn{2}{c}{12.78 (4.86)} & \multicolumn{2}{c}{12.80 (4.85)} & \multicolumn{2}{c}{12.79 (4.86)} & \multicolumn{2}{c}{12.80 (4.85)} & \multicolumn{2}{c}{12.75 (4.88)} & \multicolumn{2}{c}{12.77 (4.87)} \\
\bottomrule  
    \end{tabular}  
    \end{threeparttable}  
}

\noindent\makebox[\textwidth]{%
  \fontsize{10}{8}\selectfont  
\begin{threeparttable}  
  \renewcommand\arraystretch{1}
  \caption{MSE on $\bbeta$ among different $d$ in homogeneous setting for Cauchy error case (Case 2) with $\eta=10$ and $p=225$. The results are presented as ``average (standard deviation)''.}  
  \label{tab:c2h1p225-mse}    
  \setlength\extrarowheight{-3pt}
    \begin{tabular}{ccccccccccccccc}  
    \toprule  
    \multirow{1}{*}{$\left| {{\mathcal{A}}} \right|$}&  
    \multicolumn{2}{c}{L}&\multicolumn{2}{c}{OTL}&\multicolumn{2}{c}{NTL}&\multicolumn{2}{c}{PTL}&\multicolumn{2}{c}{TL}&\multicolumn{2}{c}{OTLP}&\multicolumn{2}{c}{TLP}\cr
    \midrule  
    \multicolumn{2}{l}{\textbf{$d=2$}} & \\
    $0$ & \multicolumn{2}{c}{57.12 (23.96)} & \multicolumn{2}{c}{57.12 (23.96)} & \multicolumn{2}{c}{23.91 (10.18)} & \multicolumn{2}{c}{57.12 (23.96)} & \multicolumn{2}{c}{38.11 (13.63)} & \multicolumn{2}{c}{57.12 (23.96)} & \multicolumn{2}{c}{58.95 (22.48)} \\[0.1cm] 
  $1$ & \multicolumn{2}{c}{55.70 (21.32)} & \multicolumn{2}{c}{36.95 (12.66)} & \multicolumn{2}{c}{23.17 ( 8.33)} & \multicolumn{2}{c}{50.23 (23.09)} & \multicolumn{2}{c}{32.14 (12.18)} & \multicolumn{2}{c}{46.42 (16.55)} & \multicolumn{2}{c}{53.02 (19.06)} \\[0.1cm] 
  $2$ & \multicolumn{2}{c}{55.77 (21.31)} & \multicolumn{2}{c}{32.18 (11.60)} & \multicolumn{2}{c}{22.93 ( 8.35)} & \multicolumn{2}{c}{40.57 (18.26)} & \multicolumn{2}{c}{30.73 (11.79)} & \multicolumn{2}{c}{44.25 (15.96)} & \multicolumn{2}{c}{50.43 (19.73)} \\[0.1cm] 
  $4$ & \multicolumn{2}{c}{55.79 (21.28)} & \multicolumn{2}{c}{28.56 (12.06)} & \multicolumn{2}{c}{22.01 ( 8.56)} & \multicolumn{2}{c}{33.26 (15.27)} & \multicolumn{2}{c}{27.16 (11.32)} & \multicolumn{2}{c}{43.72 (18.01)} & \multicolumn{2}{c}{46.71 (19.39)} \\[0.1cm] 
  $6$ & \multicolumn{2}{c}{55.63 (21.38)} & \multicolumn{2}{c}{27.34 (12.42)} & \multicolumn{2}{c}{21.90 ( 7.82)} & \multicolumn{2}{c}{29.92 (13.57)} & \multicolumn{2}{c}{26.41 (11.55)} & \multicolumn{2}{c}{41.11 (17.65)} & \multicolumn{2}{c}{45.14 (19.32)} \\[0.1cm] 
  $8$ & \multicolumn{2}{c}{55.03 (21.00)} & \multicolumn{2}{c}{26.17 (12.90)} & \multicolumn{2}{c}{21.56 ( 8.71)} & \multicolumn{2}{c}{27.32 (11.77)} & \multicolumn{2}{c}{25.59 (11.72)} & \multicolumn{2}{c}{39.27 (16.85)} & \multicolumn{2}{c}{42.31 (18.52)} \\
\bottomrule  
    \end{tabular}  
    \end{threeparttable}  
}

\noindent\makebox[\textwidth]{%
  \fontsize{10}{8}\selectfont  
\begin{threeparttable}  
  \renewcommand\arraystretch{1}
  \caption{Quantile loss on $\bbeta$ among different $d$ in homogeneous setting for Cauchy error case (Case 2) with $\eta=10$ and $p=225$. The results are presented as ``average (standard deviation)''.}  
  \label{tab:c2h1p225-ql}   
  \setlength\extrarowheight{-3pt} 
    \begin{tabular}{ccccccccccccccc}  
    \toprule  
    \multirow{1}{*}{$\left| {{\mathcal{A}}} \right|$}&  
    \multicolumn{2}{c}{L}&\multicolumn{2}{c}{OTL}&\multicolumn{2}{c}{NTL}&\multicolumn{2}{c}{PTL}&\multicolumn{2}{c}{TL}&\multicolumn{2}{c}{OTLP}&\multicolumn{2}{c}{TLP}\cr
    \midrule  
    \multicolumn{2}{l}{\textbf{$d=2$}} & \\
    $0$ & \multicolumn{2}{c}{12.47 (6.09)} & \multicolumn{2}{c}{12.47 (6.09)} & \multicolumn{2}{c}{12.09 (6.08)} & \multicolumn{2}{c}{12.47 (6.09)} & \multicolumn{2}{c}{12.31 (6.09)} & \multicolumn{2}{c}{12.47 (6.09)} & \multicolumn{2}{c}{12.42 (6.06)} \\[0.1cm]
  $1$ & \multicolumn{2}{c}{12.73 (6.26)} & \multicolumn{2}{c}{12.50 (6.29)} & \multicolumn{2}{c}{12.34 (6.25)} & \multicolumn{2}{c}{12.66 (6.28)} & \multicolumn{2}{c}{12.48 (6.26)} & \multicolumn{2}{c}{12.57 (6.25)} & \multicolumn{2}{c}{12.61 (6.23)} \\[0.1cm]
  $2$ & \multicolumn{2}{c}{12.72 (6.26)} & \multicolumn{2}{c}{12.43 (6.30)} & \multicolumn{2}{c}{12.33 (6.25)} & \multicolumn{2}{c}{12.57 (6.28)} & \multicolumn{2}{c}{12.44 (6.26)} & \multicolumn{2}{c}{12.55 (6.27)} & \multicolumn{2}{c}{12.57 (6.23)} \\[0.1cm]
  $4$ & \multicolumn{2}{c}{12.79 (6.27)} & \multicolumn{2}{c}{12.46 (6.27)} & \multicolumn{2}{c}{12.39 (6.27)} & \multicolumn{2}{c}{12.52 (6.27)} & \multicolumn{2}{c}{12.44 (6.27)} & \multicolumn{2}{c}{12.61 (6.25)} & \multicolumn{2}{c}{12.61 (6.24)} \\[0.1cm]
  $6$ & \multicolumn{2}{c}{13.11 (6.44)} & \multicolumn{2}{c}{12.76 (6.45)} & \multicolumn{2}{c}{12.71 (6.46)} & \multicolumn{2}{c}{12.79 (6.45)} & \multicolumn{2}{c}{12.75 (6.44)} & \multicolumn{2}{c}{12.86 (6.42)} & \multicolumn{2}{c}{12.89 (6.41)} \\[0.1cm]
  $8$ & \multicolumn{2}{c}{13.34 (6.43)} & \multicolumn{2}{c}{12.99 (6.43)} & \multicolumn{2}{c}{12.94 (6.45)} & \multicolumn{2}{c}{13.00 (6.43)} & \multicolumn{2}{c}{12.97 (6.44)} & \multicolumn{2}{c}{13.08 (6.39)} & \multicolumn{2}{c}{13.10 (6.42)} \\
\bottomrule  
    \end{tabular}  
    \end{threeparttable}  
}

\noindent\makebox[\textwidth]{%
  \fontsize{10}{8}\selectfont  
\begin{threeparttable}  
  \renewcommand\arraystretch{1}
  \caption{MSE on $\bbeta$ among different $d$ in heterogeneous setting for Cauchy error case (Case 2) with $\eta=10$ and $p=75$. The results are presented as ``average (standard deviation)''.}  
  \label{tab:c2h2p75-mse}  
  \setlength\extrarowheight{-3pt}  
    \begin{tabular}{ccccccccccccccc}  
    \toprule  
    \multirow{1}{*}{$\left| {{\mathcal{A}}} \right|$}&  
    \multicolumn{2}{c}{L}&\multicolumn{2}{c}{OTL}&\multicolumn{2}{c}{NTL}&\multicolumn{2}{c}{PTL}&\multicolumn{2}{c}{TL}&\multicolumn{2}{c}{OTLP}&\multicolumn{2}{c}{TLP}\cr
    \midrule  
    \multicolumn{2}{l}{\textbf{$d=2$}} & \\
    $0$ & \multicolumn{2}{c}{15.49 (4.13)} & \multicolumn{2}{c}{15.49 (4.13)} & \multicolumn{2}{c}{36.22 (9.12)} & \multicolumn{2}{c}{15.49 (4.13)} & \multicolumn{2}{c}{15.46 (4.19)} & \multicolumn{2}{c}{15.49 (4.13)} & \multicolumn{2}{c}{16.82 (5.07)} \\[0.1cm] 
  $1$ & \multicolumn{2}{c}{15.54 (4.16)} & \multicolumn{2}{c}{12.78 (4.25)} & \multicolumn{2}{c}{34.90 (8.55)} & \multicolumn{2}{c}{12.83 (4.28)} & \multicolumn{2}{c}{14.25 (4.70)} & \multicolumn{2}{c}{11.14 (3.59)} & \multicolumn{2}{c}{13.61 (4.78)} \\[0.1cm] 
  $2$ & \multicolumn{2}{c}{15.58 (4.13)} & \multicolumn{2}{c}{11.43 (3.30)} & \multicolumn{2}{c}{33.57 (8.30)} & \multicolumn{2}{c}{11.44 (3.40)} & \multicolumn{2}{c}{13.89 (4.56)} & \multicolumn{2}{c}{ 9.87 (3.27)} & \multicolumn{2}{c}{12.27 (4.63)} \\[0.1cm] 
  $4$ & \multicolumn{2}{c}{15.61 (4.12)} & \multicolumn{2}{c}{10.24 (2.85)} & \multicolumn{2}{c}{30.23 (8.14)} & \multicolumn{2}{c}{10.44 (3.14)} & \multicolumn{2}{c}{13.69 (4.40)} & \multicolumn{2}{c}{ 9.36 (3.23)} & \multicolumn{2}{c}{11.25 (3.83)} \\[0.1cm] 
  $6$ & \multicolumn{2}{c}{15.60 (4.19)} & \multicolumn{2}{c}{ 9.21 (2.72)} & \multicolumn{2}{c}{27.46 (7.02)} & \multicolumn{2}{c}{ 9.29 (2.70)} & \multicolumn{2}{c}{13.34 (4.05)} & \multicolumn{2}{c}{ 8.34 (2.59)} & \multicolumn{2}{c}{10.58 (4.07)} \\[0.1cm] 
  $8$ & \multicolumn{2}{c}{15.70 (4.14)} & \multicolumn{2}{c}{ 8.57 (2.84)} & \multicolumn{2}{c}{24.55 (6.37)} & \multicolumn{2}{c}{ 8.62 (2.84)} & \multicolumn{2}{c}{12.97 (4.03)} & \multicolumn{2}{c}{ 7.63 (2.45)} & \multicolumn{2}{c}{ 9.83 (3.89)} \\
\bottomrule  
    \end{tabular}  
    \end{threeparttable}  
}

\noindent\makebox[\textwidth]{%
  \fontsize{10}{8}\selectfont  
\begin{threeparttable}  
  \renewcommand\arraystretch{1}
  \caption{Quantile loss on $\bbeta$ among different $d$ in heterogeneous setting for Cauchy error case (Case 2) with $\eta=10$ and $p=75$. The results are presented as ``average (standard deviation)''.}  
  \label{tab:c2h2p75-ql}    
  \setlength\extrarowheight{-3pt}
    \begin{tabular}{ccccccccccccccc}  
    \toprule  
    \multirow{1}{*}{$\left| {{\mathcal{A}}} \right|$}&  
    \multicolumn{2}{c}{L}&\multicolumn{2}{c}{OTL}&\multicolumn{2}{c}{NTL}&\multicolumn{2}{c}{PTL}&\multicolumn{2}{c}{TL}&\multicolumn{2}{c}{OTLP}&\multicolumn{2}{c}{TLP}\cr
    \midrule  
    \multicolumn{2}{l}{\textbf{$d=2$}} & \\
    $0$ & \multicolumn{2}{c}{4.78 (1.74)} & \multicolumn{2}{c}{4.78 (1.74)} & \multicolumn{2}{c}{5.24 (1.72)} & \multicolumn{2}{c}{4.78 (1.74)} & \multicolumn{2}{c}{4.78 (1.74)} & \multicolumn{2}{c}{4.78 (1.74)} & \multicolumn{2}{c}{4.80 (1.75)} \\[0.1cm] 
  $1$ & \multicolumn{2}{c}{4.77 (1.75)} & \multicolumn{2}{c}{4.68 (1.75)} & \multicolumn{2}{c}{5.20 (1.72)} & \multicolumn{2}{c}{4.68 (1.75)} & \multicolumn{2}{c}{4.73 (1.75)} & \multicolumn{2}{c}{4.64 (1.73)} & \multicolumn{2}{c}{4.71 (1.74)} \\[0.1cm] 
  $2$ & \multicolumn{2}{c}{4.78 (1.74)} & \multicolumn{2}{c}{4.64 (1.74)} & \multicolumn{2}{c}{5.17 (1.72)} & \multicolumn{2}{c}{4.64 (1.74)} & \multicolumn{2}{c}{4.73 (1.74)} & \multicolumn{2}{c}{4.61 (1.73)} & \multicolumn{2}{c}{4.69 (1.74)} \\[0.1cm] 
  $4$ & \multicolumn{2}{c}{4.82 (1.71)} & \multicolumn{2}{c}{4.65 (1.71)} & \multicolumn{2}{c}{5.14 (1.69)} & \multicolumn{2}{c}{4.65 (1.71)} & \multicolumn{2}{c}{4.76 (1.71)} & \multicolumn{2}{c}{4.64 (1.72)} & \multicolumn{2}{c}{4.71 (1.72)} \\[0.1cm] 
  $6$ & \multicolumn{2}{c}{4.86 (1.72)} & \multicolumn{2}{c}{4.65 (1.71)} & \multicolumn{2}{c}{5.11 (1.70)} & \multicolumn{2}{c}{4.66 (1.70)} & \multicolumn{2}{c}{4.78 (1.71)} & \multicolumn{2}{c}{4.66 (1.71)} & \multicolumn{2}{c}{4.74 (1.72)} \\[0.1cm] 
  $8$ & \multicolumn{2}{c}{4.91 (1.74)} & \multicolumn{2}{c}{4.69 (1.73)} & \multicolumn{2}{c}{5.10 (1.71)} & \multicolumn{2}{c}{4.69 (1.73)} & \multicolumn{2}{c}{4.82 (1.72)} & \multicolumn{2}{c}{4.70 (1.74)} & \multicolumn{2}{c}{4.76 (1.73)} \\
\bottomrule  
    \end{tabular}  
    \end{threeparttable}  
}

\noindent\makebox[\textwidth]{%
  \fontsize{10}{8}\selectfont  
\begin{threeparttable}  
  \renewcommand\arraystretch{1}
  \caption{MSE on $\bbeta$ among different $d$ in heterogeneous setting for Cauchy error case (Case 2) with $\eta=10$ and $p=225$. The results are presented as ``average (standard deviation)''.}  
  \label{tab:c2h2p225-mse}  
  \setlength\extrarowheight{-3pt}
    %
  
    \end{threeparttable}  
}

\noindent\makebox[\textwidth]{%
  \fontsize{10}{8}\selectfont  
\begin{threeparttable}  
  \renewcommand\arraystretch{1}
  \caption{Quantile loss on $\bbeta$ among different $d$ in heterogeneous setting for Cauchy error case (Case 2) with $\eta=10$ and $p=225$. The results are presented as ``average (standard deviation)''.}  
  \label{tab:c2h2p225-ql}  
  \setlength\extrarowheight{-3pt}
    %
  
    \end{threeparttable}  
}

\noindent\makebox[\textwidth]{%
  \fontsize{10}{8}\selectfont  
\begin{threeparttable}  
  \renewcommand\arraystretch{1}
  \caption{MSE on $\bbeta$ among different $d$ in homogeneous setting for normal error case (Case 1) with $\eta=20$, $K = 20$ with different $p$ and $n = {n_0} =  \cdots  = {n_K}$. The results are presented as ``average (standard deviation)''.}  
  \label{tab:c1h1ls_mse}  
    %
  
    \end{threeparttable}  
}\\

\noindent\makebox[\textwidth]{%
 \fontsize{10}{8}\selectfont  
\begin{threeparttable}  
  \renewcommand\arraystretch{1}
  \caption{Average running time in seconds among different $d$ in homogeneous setting for normal error case (Case 1) with $\eta=20$, $K = 20$ with different $p$ and $n = {n_0} =  \cdots  = {n_K}$. The results are presented as ``average (standard deviation)''.}  
  \label{tab:c1h1ls_time} 
    %
  
    \end{threeparttable}  
}

\begin{figure}[H]
\centering
\includegraphics[width=\textwidth]{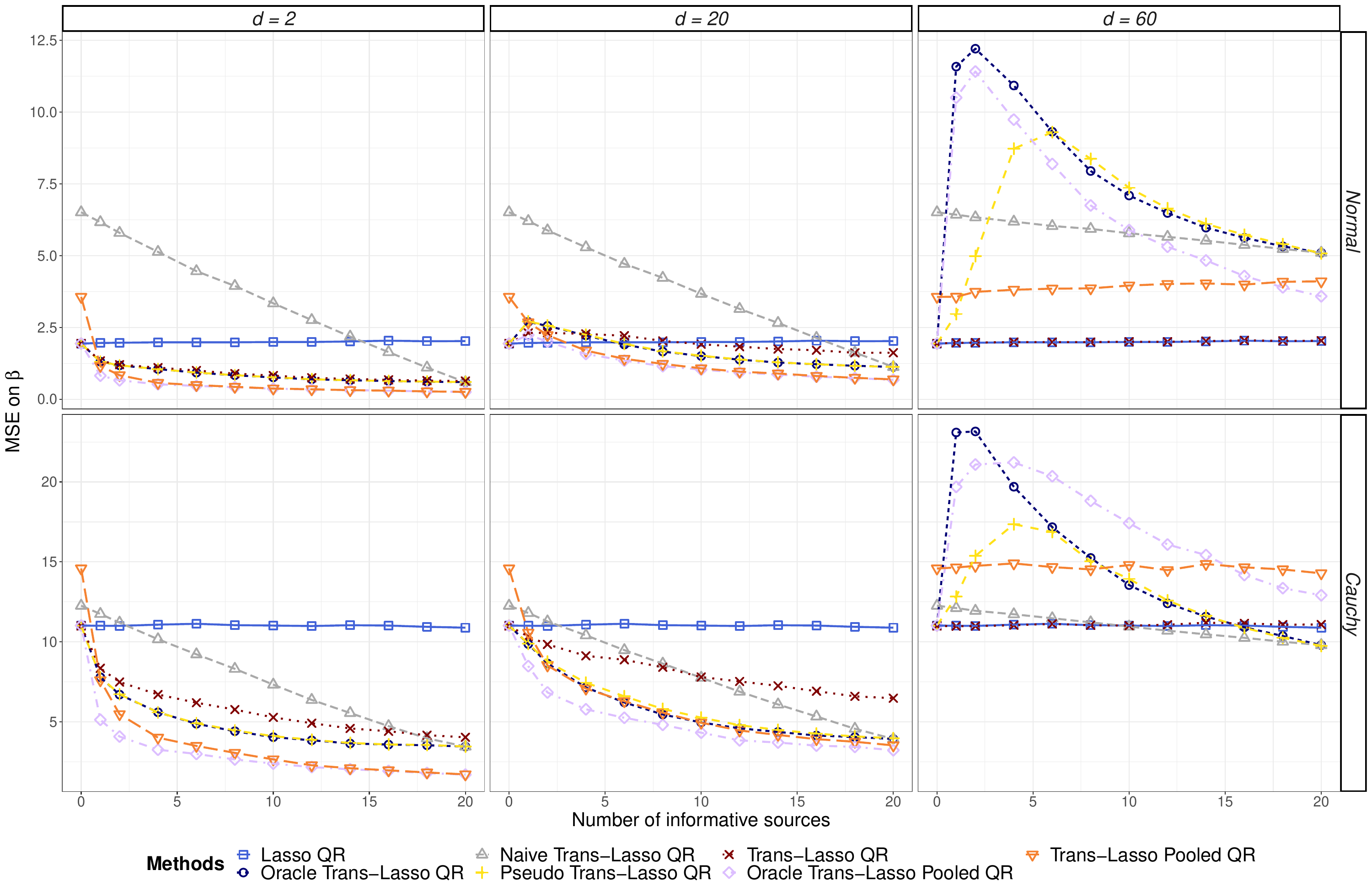}
\caption{Average MSE on $\bbeta$ among different $d$ with $\eta=20$ in heterogeneous setting for normal and Cauchy error.}
\label{fig:hete-beta}
\end{figure}

\begin{figure}[H]
\centering
\includegraphics[width=\textwidth]{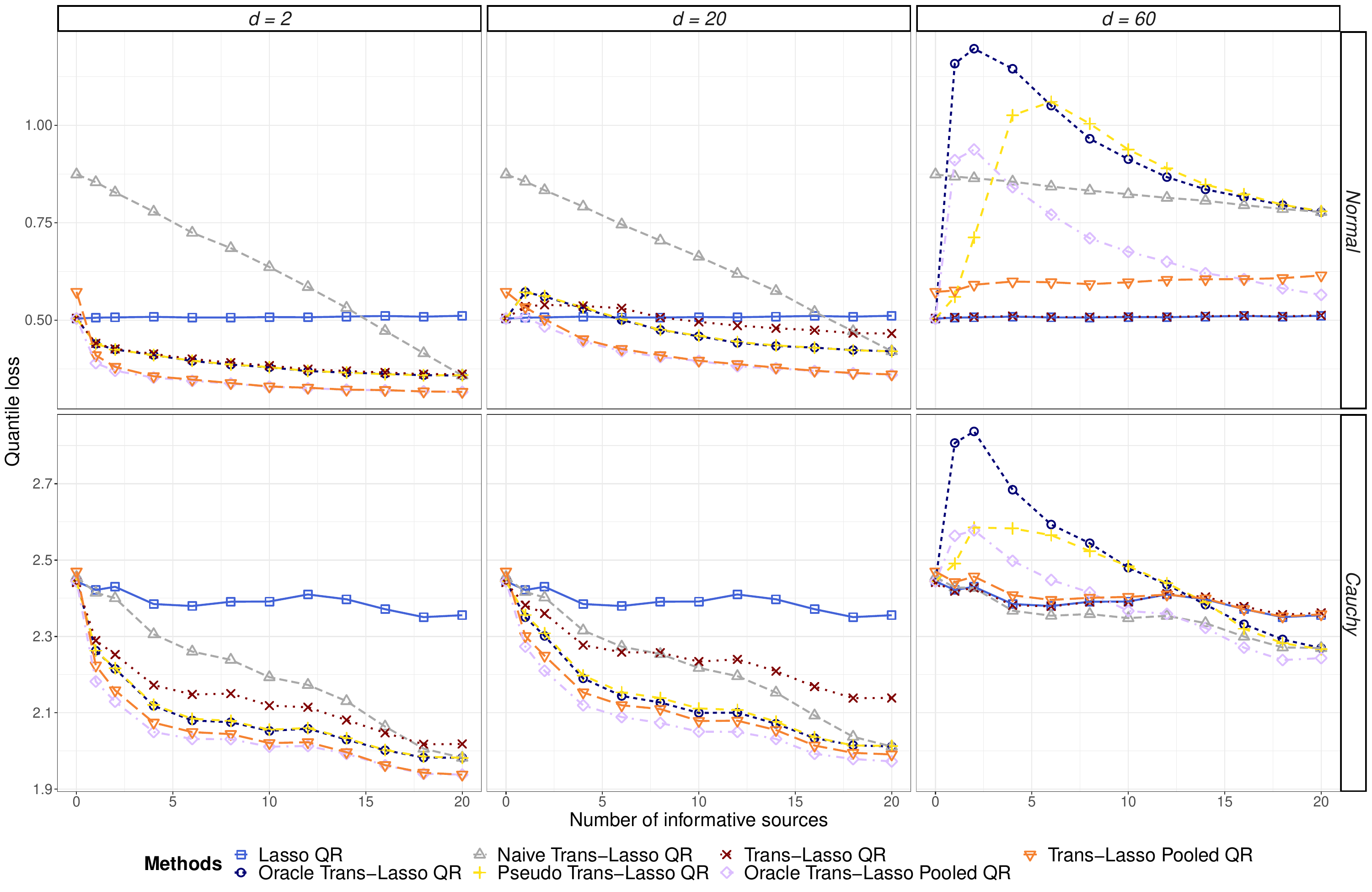}
\caption{Average quantile loss among different $d$ with $\eta=20$ in heterogeneous setting for normal and Cauchy error.}
\label{fig:hete-loss}
\end{figure}

\begin{figure}[H]
\centering
\begin{subfigure}[b]{1\textwidth}
   \includegraphics[width=15cm,height=5cm]{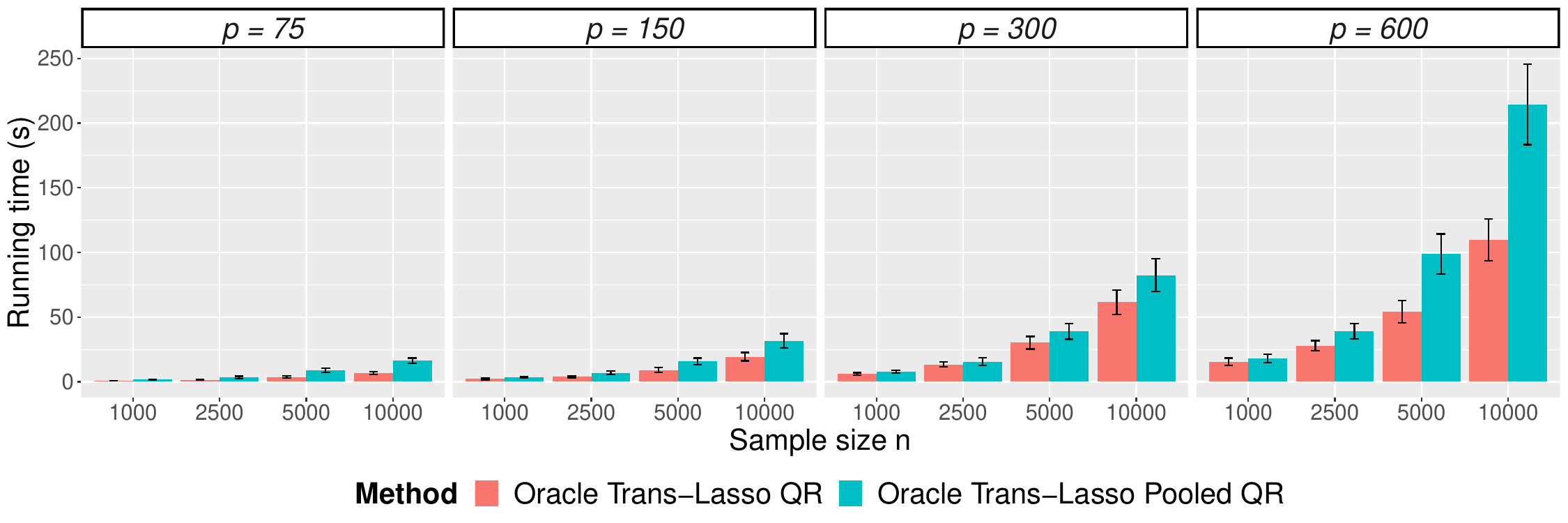}
   \caption{}
   \label{fig:lstime_oracle} 
\end{subfigure}

\begin{subfigure}[b]{1\textwidth}
   \includegraphics[width=15cm,height=5cm]{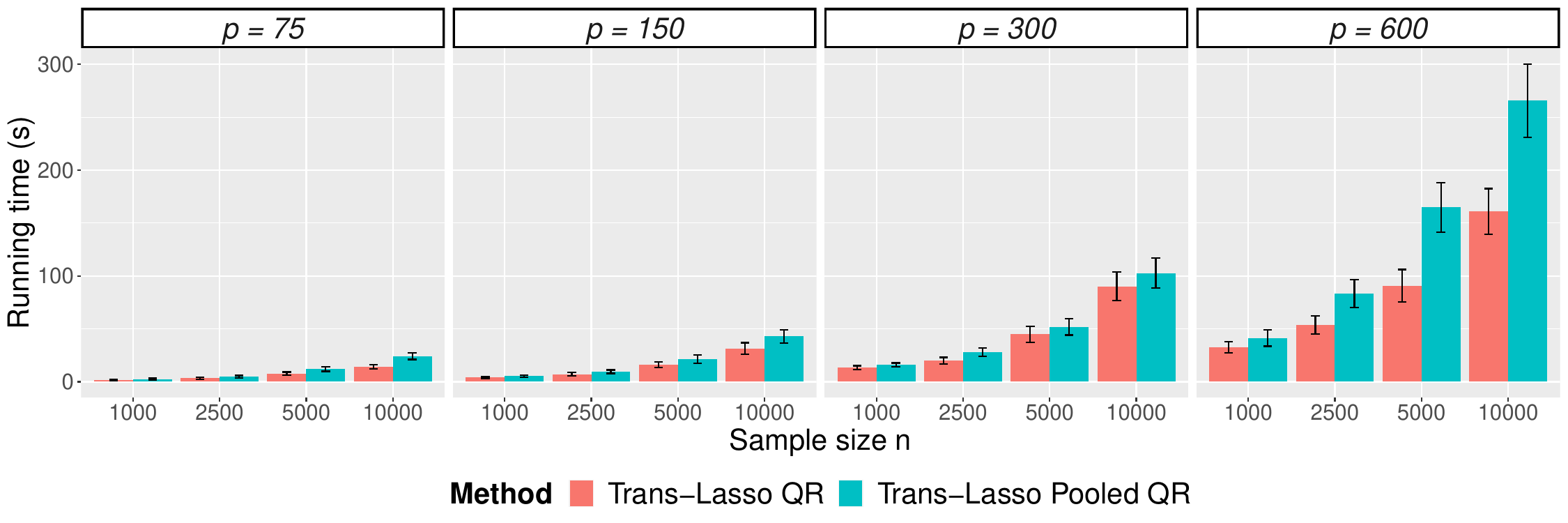}
   \caption{}
   \label{fig:lstime_trans}
\end{subfigure}
\caption{Average running time comparison for our framework and pooling framework from \cite{huang2022transfer} under (a) oracle setting and (b) with informative set detection, where the error bar is estimated based on $mean \pm 1.96 \times sd$ from 100 repetitions.}
\label{fig:lstime}
\end{figure}

\section{Features Dictionary of Flight QAR Analysis}\label{sec:feature-dictionary}

\begin{singlespace}
\begin{longtable}{| p{.1\textwidth} | p{.7\textwidth} | p{.15\textwidth} |} 
\caption{Feature dictionary for flight QAR data analysis of aggregated measurements.}\\
\hline
\textbf{Code} & \textbf{Definition} & \textbf{Type} \\  
\hline
\endfirsthead
\multicolumn{3}{c}%
{\tablename\ \thetable\ -- \textit{Continued from previous page}} \\
\hline
\textbf{Code} & \textbf{Definition} & \textbf{Type}\\ 
\hline
\endhead
\hline \multicolumn{3}{r}{\textit{Continued on next page}} \\
\endfoot
\hline
\endlastfoot
3000          & Max Airspeed during Flaps Extended                                 & Speed    \\
3001          & Max Airspeed during Landing Gear Extended                          & Speed      \\
3002          & Max Airspeed during Speed Brakes Extended                          & Speed   \\
3005          & Average Airspeed during Climb                                      & Speed   \\
3007          & Average Airspeed during Cruise                                     & Speed   \\
3009          & Average Airspeed during Descent                                    & Speed   \\
3013          & Average Airspeed during Engine Stop                                & Speed   \\
3023          & Average Airspeed during PreFlight                                  & Speed   \\
3027          & Average Airspeed during Taxi In                                    & Speed   \\
3029          & Average Airspeed during Taxi Out                                   & Speed   \\
3035          & Max Groundspeed during Turn                                        & Speed   \\
3036          & Max Groundspeed during Pre Lift Off                                & Speed   \\
3037          & Max Groundspeed during Straight Line                               & Speed   \\
3057          & Time during Lift Off to Landing Gear Selected Up                   & Time \& Dist   \\
3058          & Time during Landing Gear Selected Down to Touchdown                & Time \& Dist  \\
3070          & Time during Climb                                                  & Time \& Dist  \\
3071          & Time during Cruise                                                 & Time \& Dist  \\
3072          & Time during Descent                                                & Time \& Dist  \\
3089          & Time during Touchdown to Start of All Engine Reverser Deployed     & Time \& Dist  \\
3090          & Time during Final 200 knots CAS to Touchdown                       & Time \& Dist  \\
3093          & Max Vertical Acceleration during Flight                            & Acceleration  \\
3094          & Max Vertical Acceleration during Landing                           & Response  \\
3096          & Max Vertical Acceleration during Take Off                          & Acceleration  \\
3104          & Max Vertical Acceleration during Pre Take Off                      & Acceleration  \\
3105          & Max Vertical Acceleration during Post Landing                      & Acceleration  \\
3112          & Average Wind Direction during First 20ft AFE to First 100ft AFE    & External  \\
3113          & Average Angle of Attack during Cruise                              & Attitude \\
3119          & Difference in Fuel Consumed Engine 1 during Lift Off to First Cruise & Fuel\\
3123          & Difference in Fuel Consumed Engine 1 during Climb                  & Fuel  \\
3124          & Difference in Fuel Consumed Engine 1 during Cruise                 & Fuel  \\
3125          & Difference in Fuel Consumed Engine 1 during Descent                & Fuel  \\
3192          & Average Fuel Flow Engine 1 during Climb                            & Fuel  \\
3193          & Average Fuel Flow Engine 1 during Cruise                           & Fuel  \\
3194          & Average Fuel Flow Engine 1 during Descent                          & Fuel  \\
3195          & Average Fuel Flow Engine 1 during Engine Start                     & Fuel  \\
3201          & Average Fuel Flow Engine 1 during PreFlight                        & Fuel  \\
3203          & Average Fuel Flow Engine 1 during Taxi In                          & Fuel  \\
3204          & Average Fuel Flow Engine 1 during Taxi Out                         & Fuel  \\
3253          & Total Ground Distance during Climb                                 & Time \& Dist  \\
3254          & Total Ground Distance during Cruise                                & Time \& Dist  \\
3255          & Total Ground Distance during Descent                               & Time \& Dist  \\
3256          & Total Ground Distance during Engine Start                          & Time \& Dist  \\
3265          & Total Ground Distance during Taxi Out                              & Time \& Dist  \\
3268          & Max Absolute Longitudinal Acceleration during Landing              & Acceleration  \\
3271          & Average Mach during Climb                                          & Speed  \\
3273          & Average Mach during Cruise                                         & Speed  \\
3275          & Average Mach during Descent                                        & Speed  \\
3279          & Average Mach during Engine Stop                                    & Speed  \\
3293          & Average Mach during Taxi In                                        & Speed  \\
3295          & Average Mach during Taxi Out                                       & Speed  \\
3344          & Max Absolute Bank Angle during Final 50ft AFE to Touchdown         & Attitude  \\
3361          & Max Airspeed - Take Off Safety Speed during Slats Only Extended    & Speed  \\
3362          & Max Airspeed - Take Off Safety Speed during Flap Setting 1         & Speed  \\
3363          & Max Airspeed - Take Off Safety Speed during Flap Setting 2         & Speed  \\
3364          & Max Airspeed - Take Off Safety Speed during Flap Setting 3         & Speed  \\
3365          & Max Airspeed - Take Off Safety Speed during Flap Setting 4         & Speed  \\
3366          & Max Airspeed - Take Off Safety Speed during Flap Setting 5         & Speed  \\
3367          & Max Airspeed - Take Off Safety Speed during Flap Setting 6         & Speed  \\
3368          & Max Airspeed - Take Off Safety Speed during Flap Setting 7         & Speed  \\
3384          & Max Wind Speed during First 20ft AFE to First 100ft AFE            & External  \\
3385          & Max Absolute Lateral Acceleration during Landing                   & Acceleration  \\
3387          & Min Maximum Operating Speed - Airspeed during Flight               & Speed  \\
3418          & Max Airspeed - Landing Reference Speed during Slats Only Extended  & Speed  \\
3419          & Max Airspeed - Landing Reference Speed during Flap Setting 1       & Speed  \\
3420          & Max Airspeed - Landing Reference Speed during Flap Setting 2       & Speed  \\
3421          & Max Airspeed - Landing Reference Speed during Flap Setting 3       & Speed  \\
3422          & Max Airspeed - Landing Reference Speed during Flap Setting 4       & Speed  \\
3423          & Max Airspeed - Landing Reference Speed during Flap Setting 5       & Speed  \\
3424          & Max Airspeed - Landing Reference Speed during Flap Setting 6       & Speed  \\
3425          & Max Airspeed - Landing Reference Speed during Flap Setting 7       & Speed  \\
3426          & Max Airspeed - Landing Reference Speed during Flap Setting 8       & Speed \\
\hline
\label{tab:dictionary1}
\end{longtable}
\end{singlespace}

\begin{singlespace}
\begin{longtable}{| p{.1\textwidth} | p{.7\textwidth} | p{.15\textwidth} |} 
\caption{Feature dictionary for flight QAR data analysis of scalar measurements.}\\
\hline
\textbf{Code} & \textbf{Definition} & \textbf{Type} \\  
\hline
\endfirsthead
\multicolumn{3}{c}%
{\tablename\ \thetable\ -- \textit{Continued from previous page}} \\
\hline
\textbf{Code} & \textbf{Definition} & \textbf{Type}\\ 
\hline
\endhead
\hline \multicolumn{3}{r}{\textit{Continued on next page}} \\
\endfoot
\hline
\endlastfoot
1000          & Airspeed at Final Wing Change Configuration before Touchdown & Speed\\
1001          & Airspeed at First Wing Change Configuration after Lift Off   & Speed\\
1002          & Airspeed at First Wing Clean Configuration after Lift Off    & Speed\\
1005          & Airspeed at Landing Gear Selected Down                       & Speed\\
1006          & Airspeed at Landing Gear Selected Up                         & Speed\\
1018          & Altitude at Landing Gear Selected Down                       & Attitude\\
1019          & Altitude at Landing Gear Selected Up                         & Attitude\\
1026          & Altitude at Final 200 knots CAS                              & Attitude\\
1034          & Gross Weight at Start of Flight                              & Weight\\
1050          & Gross Weight at Start of Taxi Out                            & Weight\\
1056          & Groundspeed at Start of Off Runway                           & Speed\\
1057          & Landing Reference Speed at Threshold                         & Speed\\
1058          & Landing Reference Speed at Touchdown                         & Speed\\
1062          & Landing Reference Speed at Final 50ft AFE                    & Attitude\\
1065          & Pitch at Lift Off                                            & Attitude\\
1070          & Pitch Rate at Lift Off                                       & Attitude\\
1075          & Rotation Speed at Rotation                                   & Speed\\
1076          & Take Off Decision Speed, V1 at Lift Off                      & Speed\\
1077          & Take Off Safety Speed, V2 at Lift Off                        & Speed\\
1080          & Temperature at Lift Off                                      & External\\
1081          & Temperature at Touchdown                                     & External\\
1087          & Time at Start of Climb                                       & Time \& Dist\\
1088          & Time at Start of Cruise                                      & Time \& Dist\\
1096          & Time at Start of PreFlight                                   & Time \& Dist\\
1103          & Vertical Speed at Touchdown                                  & Speed\\
1109          & Wind Direction at Lift Off                                   & External\\
1110          & Wind Direction at Touchdown                                  & External\\
1113          & Bank Angle at Touchdown                                      & Attitude\\
1114          & Stabilizer Angle at Rotation                                 & Attitude\\
1115          & Wind Speed at Touchdown                                      &  External\\
\hline
\label{tab:dictionary2}
\end{longtable}
\end{singlespace}

\section{Supplementary Figures of Flight QAR Analysis}\label{sec:supp-fig-qar}

\begin{figure}[H]
\centering
\begin{subfigure}[b]{0.45\textwidth}
   \includegraphics[width=1\linewidth]{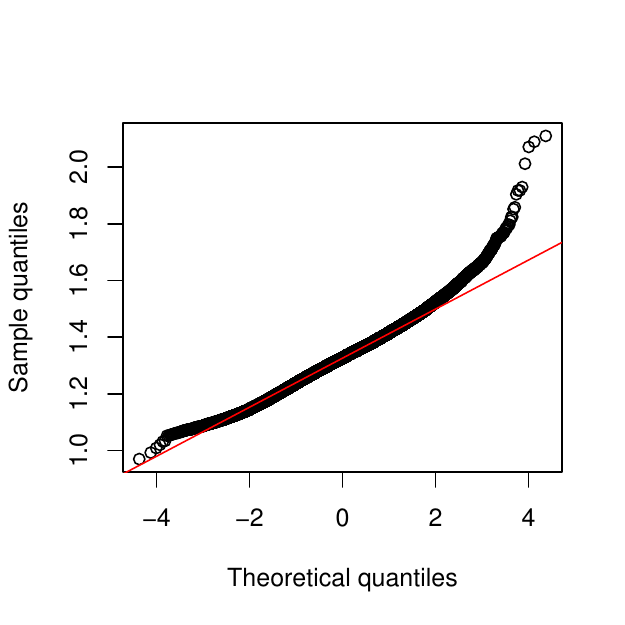}
   \caption{Normal Q-Q plot for the MVA from Airbus 320 records.}
   \label{fig:Np1} 
\end{subfigure}
~
\begin{subfigure}[b]{0.45\textwidth}
   \includegraphics[width=1\linewidth]{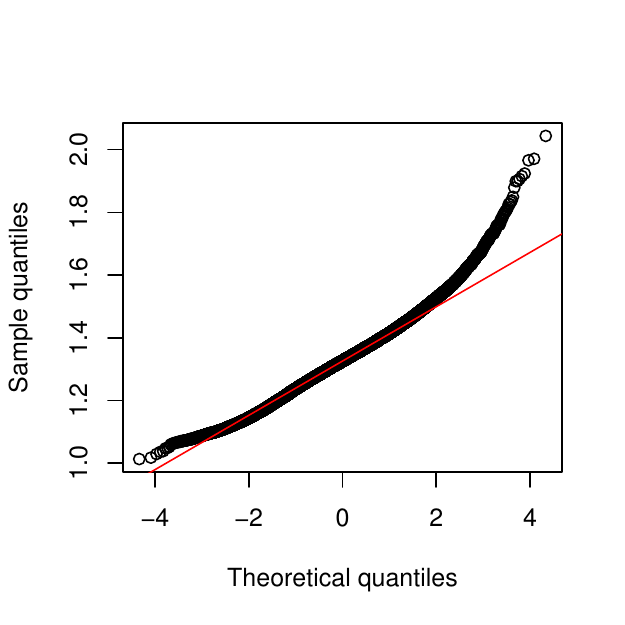}
   \caption{Normal Q-Q plot for the MVA from Boeing 737 records.}
   \label{fig:Np2}
\end{subfigure}
~
\begin{subfigure}[b]{0.45\textwidth}
   \includegraphics[width=1\linewidth]{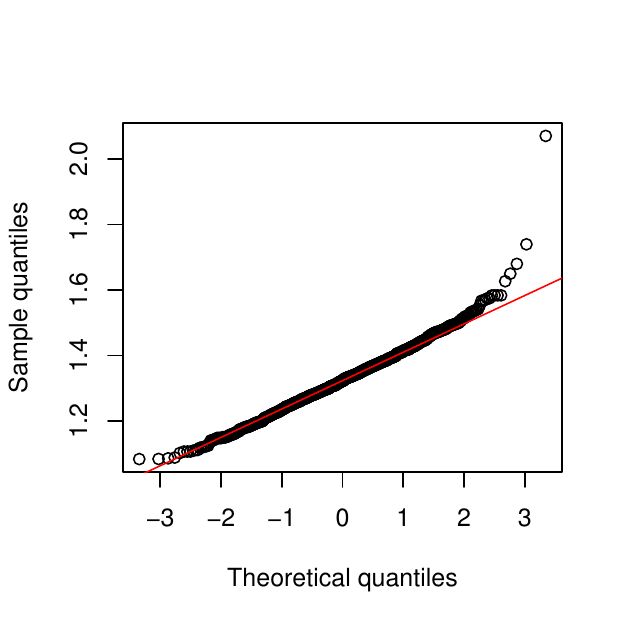}
   \caption{Normal Q-Q plot for the MVA from Airbus 380 records.}
   \label{fig:Np3}
\end{subfigure}

\caption{Evidences that the MVA are highly right skewed and further are not in normal distribution.}
\label{fig:qqmva}
\end{figure}

\begin{figure}[H]
\centering
\includegraphics[width=0.6\textwidth]{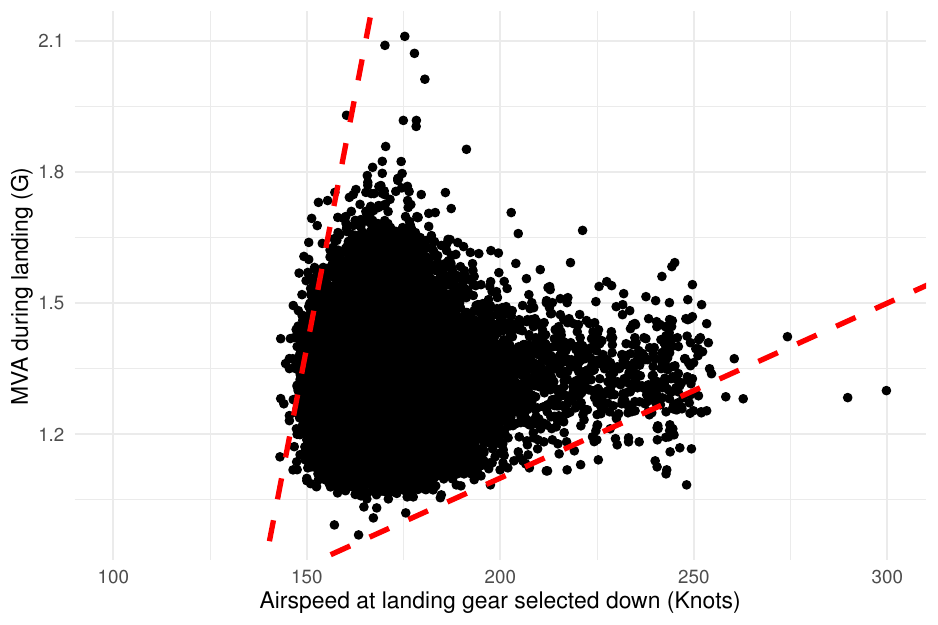}
\caption{Heteroscedasticity of MVA with respect of Airspeed at Landing Gear Selected Down in Airbus 320 dataset.}
\label{fig:hetero-show}
\end{figure}

\bibliographystyle{chicago}

\bibliography{tlqr}
\end{document}